\documentclass[12pt,letter]{amsart}
\usepackage{amscd}
\usepackage{amssymb}
\usepackage[centering,text={15.5cm,24cm}]{geometry}
\usepackage{graphicx}
\usepackage{color}
\usepackage[all]{xy}
\usepackage{mathrsfs}
\usepackage{marvosym}
\usepackage{stmaryrd}
\usepackage{srcltx}
\usepackage{hyperref}
\usepackage{tikz}

\numberwithin{equation}{section}

\numberwithin{subsection}{section}

\allowdisplaybreaks[1]

\newtheorem*{namedtheorem}{\theoremname}
\newcommand{\theoremname}{testing}

\theoremstyle{plain}

\newtheorem{theorem}{Theorem}[section]
\newtheorem{proposition}[theorem]{Proposition}
\newtheorem{proposition-definition}[theorem]{Proposition-Definition}
\newtheorem{lemma-definition}[theorem]{Lemma-Definition}
\newtheorem{corollary}[theorem]{Corollary}
\newtheorem{lemma}[theorem]{Lemma}

\newtheorem{assumptions}[theorem]{Assumptions}

\theoremstyle{definition}
\newtheorem{definition}[theorem]{Definition}
\newtheorem{notation}[theorem]{Notation}
\newtheorem{example}[theorem]{Example}

\newtheorem{remark}[theorem]{Remark}
\newtheorem{remarks}[theorem]{Remarks}

\newtheorem{construction}[theorem]{Construction}

\theoremstyle{remark}
\newtheorem*{claim}{Claim}

\newcommand\ul[1]{\underline{#1}}

\newcommand\out{\mathrm{out}}

\newcommand\fs{\mathrm{fs}}

\newcommand\btau{{\boldsymbol{\tau}}}

\newcommand\tD{\widetilde{D}}

\newcommand\scrM{\mathscr{M}}

\newcommand\sch{\mathrm{sch}}

\newcommand\Int{\operatorname{Int}}

\newcommand\tX{\widetilde{X}}

\newcommand{\ev}{\mathrm{ev}}

\newcommand\cA{\mathcal{A}}
\newcommand\shA{\mathcal{A}}

\newcommand\shD{\mathcal{D}}

\newcommand\shL{\mathcal{L}}
\newcommand\cM{\mathcal{M}}
\newcommand\shM{\mathcal{M}}

\newcommand\cO{\mathcal{O}}

\newcommand\cX{\mathcal{X}}
\newcommand\shX{\mathcal{X}}

\newcommand\Div{\mathrm{Div}}
\renewcommand\AA{\mathbb{A}}

\newcommand\GG{\mathbb{G}}

\newcommand\kk{\Bbbk}

\newcommand\NN{\mathbb{N}}

\newcommand\PP{\mathbb{P}}
\renewcommand\P{\mathscr{P}}
\newcommand\QQ{\mathbb{Q}}
\newcommand\RR{\mathbb{R}}

\newcommand\ZZ{\mathbb{Z}}

\newcommand\bA{\mathbf{A}}

\newcommand\bE{\mathbf{E}}

\newcommand\bg{\mathbf{g}}
\newcommand\bsigma{\boldsymbol{\sigma}}

\newcommand\bL{\mathbf{L}}

\newcommand\bu{\mathbf{u}}

\newcommand\fC{\mathfrak{C}}

\newcommand\fM{\mathfrak{M}}
\newcommand\foM{\mathfrak{M}}
\newcommand\tfM{\widetilde{\mathfrak{M}}}
\newcommand\tfC{\widetilde{\mathfrak{C}}}

\newcommand\Pic{\mathrm{Pic}}

\newcommand\arr{\ifinner\to\else\longrightarrow\fi}

\def\displaytimes_#1{\mathrel{\mathop{\times}\limits_{#1}}}

\def\displayotimes_#1{\mathrel{\mathop{\bigotimes}\limits_{#1}}}

\newcommand\ext{\operatorname{Ext}}
\newcommand\coker{\operatorname{coker}}

\newcommand\Aut{\operatorname{Aut}}

\newcommand\tors{{\mathrm{tors}}}

\newcommand\Spec{\operatorname{Spec}}

\newcommand\codim{\operatorname{codim}}

\newcommand\virt{{\operatorname{virt}}}

\newdir{ >}{{}*!/-5pt/@{>}}

\newcommand\doublelong[2]{\mathbin{\xymatrix{{}\ar@<3pt>[r]^{#1}
\ar@<-3pt>[r]_{#2}&}}}

\newlength{\ignora}

\newcommand{\red}{{\mathrm{red}}}

\newcommand{\gp}{{\mathrm{gp}}}
\newcommand{\trop}{{\mathrm{trop}}}

\renewcommand{\setminus}{\smallsetminus}

\newcommand{\Hom}{{\operatorname{Hom}}}

\begin{document}

\title{Remarks on gluing punctured logarithmic maps}

\author{Mark Gross}

\date{\today}
\begin{abstract}
We consider some well-behaved cases of the gluing formalism for 
punctured stable log maps of \cite{ACGSI,ACGSII}. This gives a
gluing formula for log Gromov-Witten invariants in a diverse
range of situations. While this does not give a complete solution
to the general stable log map gluing problem, the cases considered
here are sufficient for many purposes. The gluing formulae of 
\cite{LR, Li, KLR} becomes an easy special case. The last section gives an application
of this gluing formalism to canonical wall structures for K3 surfaces
as constructed in \cite{Walls}.
\end{abstract} 

\maketitle

\setcounter{tocdepth}{2}
\tableofcontents

\section{Introduction}
One of the key motivations for developing logarithmic Gromov-Witten theory
\cite{AC,Chen,JAMS}
was to generalize the gluing formulae of Li--Ruan \cite{LR} and
Jun Li \cite{Li} to more general degenerations. The original gluing formulae
consider flat families $\pi:X\rightarrow B$, where $B$ is a non-singular
curve with a special point $b_0\in B$, $\pi$ is a normal crossing
morphism with $\pi|_{\pi^{-1}(B\setminus \{b_0\})}$ smooth and
$\pi^{-1}(b_0)=:X_0$ a union of two irreducible divisors $Y_1,Y_2$
meeting transerversally. The above-mentioned gluing formulae then
relate the Gromov-Witten invariants of the general fibre of
$\pi$ to relative Gromov-Witten invariants of the pairs
$(Y_1,D)$, $(Y_2,D)$, where $D=Y_1\cap Y_2$. In principal, one would
like to allow the fibre $X_0$ to have many irreducible components
and deeper strata, where more than two irreducible components meet.

Logarithmic Gromov-Witten theory defines the notion of a logarithmic
stable map to such targets $X\rightarrow B$; more generally, this morphism
just needs to be toroidal rather than normal crossings, i.e., log smooth.
However, completely satisfactory generalizations of the gluing formulae 
have remained elusive. The work of Abramovich, Chen, Gross and Siebert
\cite{ACGSI, ACGSII} sets up a framework for thinking about gluing
formulae. This required, in particular, the development of
\emph{punctured}
stable maps, further generalizing logarithmic stable maps. The essential
reason for this is that, if given a log smooth curve, the restriction of the
log structure to an irreducible component need not yield a log smooth
curve. In
particular, this requires allowing somewhat more general domains. As
a side benefit, this generalization introduces
the notion of negative contact order. In turn, this gives
a richer set of invariants which have proved invaluable for mirror
symmetry constructions, see e.g., \cite{Assoc,Walls}. 

The basic setup for a gluing problem for log or punctured maps with log smooth
target $X\rightarrow B$ is given by the data of
a \emph{decorated tropical type}, reviewed in \S\ref{sec:preliminaries}.
This is data $\btau:=(G,\bsigma,\bu,\bA)$.
Here, $G$ is a dual intersection graph for a domain curve, with vertices
$V(G)$ corresponding to (unions of) irreducible components, edges $E(G)$
corresponding to nodes and legs $L(G)$ corresponding to marked or 
punctured points. The map $\bsigma:V(G)\cup E(G)\cup L(G)\rightarrow 
\Sigma(X)$ records which stratum of $X$ (strata of $X$ being indexed
by cones in the tropicalization $\Sigma(X)$ of $X$) the corresponding
curve feature maps into. The data $\bu$ records the contact orders
associated to edges and legs, and the data $\bA$ associates a curve
class in $X$ to each vertex of $G$. Together, this data determines
a moduli space $\scrM(X/B,\btau)$ of punctured log maps marked by
$\btau$, as defined in
\cite{ACGSII}. This moduli space is a proper Deligne-Mumford
stack over $B$, assuming $X$ is projective over $B$, and carries a
virtual fundamental class $[\scrM(X/B,\btau)]^{\virt}$.

The question of gluing is then as follows. Suppose given a decorated
tropical type $\btau$ as above, and a subset of edges $\mathbf{E}\subseteq
E(G)$. By splitting $G$ at the edges of $\mathbf{E}$, we obtain connected
graphs $G_1,\ldots,G_r$. Here, each edge $E\in \mathbf{E}$ with endpoints
$v_1,v_2$ is replaced by two legs with endpoints $v_1,v_2$ respectively.
By restricting $\bsigma,\bu$ and $\bA$ to $G_i$, we obtain types
$\btau_1,\ldots,\btau_r$. Further, there is a canonical splitting map
\[
\scrM(X/B,\btau)\rightarrow \prod_{i=1}^r \scrM(X/B,\btau_i).
\]
This splitting map is defined by normalizing the domain curves
at the nodes corresponding to edges of $\mathbf{E}$.
The key question is then: can we describe $[\scrM(X/B,\btau)]^{\virt}$
in terms of $\prod_{i=1}^r [\scrM(X/B,\btau_i)]^{\virt}$?

There have been a number of approaches to this question. First,
Kim, Lho and Ruddat
\cite{KLR} proved the Li--Ruan and Jun Li degeneration formula
in the context of logarithmic Gromov-Witten theory. Yixian Wu, in 
\cite{Wu}, gave a very general gluing formula, under the hypothesis
that all gluing strata are toric. This has already proven to be very
useful in \cite{Walls}, but doesn't give a proof of the Li--Ruan/Jun Li
formula (unless the divisor $D$ is in fact a point). 

A very different approach has been pursued by Dhruv Ranganathan, 
using expanded degenerations, in \cite{Ra}. There, the theoretical
difficulties of gluing are removed by allowing target expansions
in such a way that the gluing always happens along codimension one
strata.
However, this can result in an explosion in combinatorial complexity
of the problem. It is not clear whether such an explosion can be
avoided in any general approach to the gluing problem.

The basic problem is that the description of the glued moduli
space $\scrM(X/B,\btau)$ in terms of the moduli spaces $\scrM(X/B,\btau_i)$
involves a fibre product in the category of fs log schemes; this is
roughly encapsulated in one of the main gluing theorems of \cite{ACGSII},
quoted here in Theorem \ref{Thm: Gluing theorem}. One of the basic
difficulties in log geometry is that the underlying scheme of an fs
fibre product can be quite far from the underlying scheme of the
ordinary fibre product. Indeed, the classical gluing situation investigated
in the log category in \cite{KLR} works because in this case the relevant
fs fibre product is \'etale over the ordinary fibre product, while this
is definitely not true in complete generality.

Here, we consider a case of gluing in which, at the virtual level,
the ordinary fibre product and fs fibre product are still not wildly
divergent. Understanding fs fibre products even of log points
is non-trivial. 
Results about fibre products of points in
the category of schemes are known as four-point lemmas, and generalisations
of such to the fs log category have been explored to some extent 
in \cite{Ogus}.
In \S\ref{sec:four point}, we give some general
results in a slightly different direction about such fibre products
helpful for our situation. 
In particular, we give a sufficient
criteria for non-emptiness of an fs fibre product of log points,
as well as a computation for the number of connected components of
this fs fibre product if it is non-empty.

Happily, these criteria for non-emptiness have a simple tropical
interpretation. In \S\ref{sec:gluing one curve}, we consider
gluing a single curve. In other words, we consider a type $\btau$
as described above, a set of splitting edges $\mathbf{E}$, and
the types $\btau_1,\ldots,\btau_r$ obtained from splitting at the
edges of $\mathbf{E}$. We consider log points $W_i$ and punctured
log maps $f_i:C_i^{\circ}/W_i \rightarrow X$ of type $\btau_i$,
$1\le i \le r$. If $E\in \mathbf{E}$ is an edge with vertices
$v_1,v_2$, $v_j\in V(G_{i_j})$, $j=1,2$, let $p_{v_j,E}
\in C_{i_j}^{\circ}$ be the punctured point corresponding to the leg
of $G_{i_j}$ indexed by the flag $v_j\in E$. Assume that we have
for each $E$ the equality
$f_{i_1}(p_{v_1,E})=f_{i_2}(p_{v_2,E})$. Then the maps $f_i$
may be glued schematically. The question is then: is there
a logarithmic gluing and if so, how many logarithmic
gluings of the $f_i$'s are there?

In Definition
\ref{def:tropical gluing map}, we define a map of lattices,
the \emph{tropical gluing map} $\Psi$, which 
depends only on the tropical data of $\btau, \mathbf{E}$, and define
the \emph{tropical multiplicity} $\mu(\btau,\mathbf{E})$ as the order
of the torsion part of $\coker\Psi$. This lattice
map gives the obstruction to gluing tropical maps of types $\tau_1,
\ldots,\tau_q$ to obtain a tropical map of type $\tau$. In addition,
let $f:C^{\circ}/W\rightarrow X$ be the universal
gluing of the punctured maps $f_i$. Using the four-point lemmas of
\S\ref{sec:four point}, we obtain the first result, 
Theorem \ref{thm:gluing count}, which is:

\begin{theorem} 
If $W$ is non-empty, then it has $\mu(\btau,\mathbf{E})$ connected components.
\end{theorem}

We remark that this is not complete information about $W$, as it may
have some non-reduced structure. However, when gluing questions are
set up properly, this becomes unimportant, as is seen in \cite{Wu} or
\S\ref{sec:degeneration gluing}.

We say the gluing situation is \emph{tropically transverse} if
$\coker\Psi$ is in fact finite. In this case, we have
Theorem \ref{thm:tropically transverse case}, again from the
four-point lemmas of \S\ref{sec:four point}:

\begin{theorem}
If the gluing situation is tropically transverse, then $W$ is non-empty.
\end{theorem}

These results can be viewed as a generalization of the more hands-on
constructions of stable log maps beginning with work of Nishinou--Siebert
\cite{NS}, Arg\"uz \cite{Arguz}, Cheung--Fantini--Park--Ulirsch
\cite{CFPU}, Mandel--Ruddat \cite{MR}, and \cite[\S4.2]{ACGSI}. In fact,
versions of tropical gluing maps already appeared in \cite{NS}.

In \S\ref{sec:degeneration gluing}, we now apply these observations
to gluing moduli spaces, with an aim to describe $[\scrM(X/B,\btau)]^{\virt}$
in terms of $\prod_{i=1}^r [\scrM(X/B,\btau_i)]^{\virt}$. We define
an intermediate moduli space $\scrM^{\mathrm{sch}}(X,\btau)$. A point
in this moduli space is represented by a point in $\prod_{i=1}^r 
\scrM(X/B,\btau_i)$ corresponding to a collection of punctured maps
$f_i:C_i^{\circ}/W_i\rightarrow X$, such that the $f_i$ glue schematically.
This moduli stack can be defined via a Cartesian diagram in the category
of ordinary stacks, see Theorem \ref{thm:gluing factorization} for details.
The results of \S\ref{sec:gluing one curve} then apply to give us
some information about the natural map $\phi':\scrM(X/B,\btau)\rightarrow
\scrM^{\mathrm{sch}}(X/B,\btau)$. Unfortunately, in general it is difficult
to extract useful results from this. However, tropical transversality of
the gluing situation, along with a flatness hypothesis which can also
be tested tropically,
implies virtual surjectivity of this map, with
degree given by the tropical multiplicity, so that
\[
\phi'_*[\scrM(X/B,\btau)]^{\virt} = \mu(\btau,\mathbf{E}) [\scrM^{\mathrm{sch}}
(X/B,\btau)]^{\virt}.
\]
On the other hand, if the gluing strata are sufficiently nice (e.g., 
smooth) and certain other conditions hold, then $[\scrM^{\mathrm{sch}}(X/B,
\btau)]^{\virt}$ can be calculated as an ordinary Gysin pull-back
of $\prod_{i=1}^r [\scrM(X/B,\btau_i)]^{\virt}$. See Remark 
\ref{rem:this is the best in general} for details.

In \S\ref{subsec:rigid}, we specialize to the degeneration situation considered
in \cite{ACGSI}. Here, $B$ is a curve or spectrum of a DVR over $\kk$,
with divisorial log structure coming from a closed point $b_0\in B$. Let
$X_0$ be the fibre over $b_0$.
In \cite{ACGSI}, we showed that virtual irreducible components of
$\scrM(X_0/b_0)$ were indexed by rigid tropical curves. We recast the
earlier discussion in the gluing situation provided by a rigid tropical
curve.

While these results are not yet the dreamed-of general
gluing formula, they in fact appear to 
be strong enough to be useful in many circumstances. In particular, in 
\S\ref{sec:classical},
we give a very short proof of the Li-Ruan/Jun Li degeneration formula.
This is not new even in the logarithmic setup: \cite{KLR} first obtained
this result. However, it is pleasant to see that the more general setup
proves this special case without pain. The reason this works easily
is that the gluing situation is always tropically transverse in this
case.

Along the way, in \S\ref{sec:punctured versus relative}, we first
prove a more generally useful comparison result between 
punctured and log invariants for irreducible components of degenerations.
Explicitly, given $X\rightarrow B$ as in the degeneration situation,
often one needs to look at moduli spaces of punctured maps into strata
of $X_0$. In general, this may involve additional information, but
if the stratum is an irreducible component $Y\subseteq X_0$, life becomes
simpler. In particular, $Y$ carries two possible log structures, one
induced from $X$, and one the divisorial log structure coming from the
union of substrata of $Y$. We write this latter log structure as
$\overline{Y}$. We then obtain in Theorem \ref{thm:punctured to relative} 
an isomorphism of underlying stacks
$\ul{\scrM(Y/b_0,\btau)}\cong \ul{\scrM(\overline{Y},\bar\btau)}$,
where the type $\bar\btau$ is derived from the type $\btau$.
Happily, the latter type does not involve punctures, and hence gives
a more familiar moduli space. We note these stacks are not isomorphic
as log stacks, but they do carry the same virtual fundamental class.

The final section is an extended application of the gluing techniques
given in this paper. We study the genus zero punctured Gromov-Witten theory of
maximally unipotent degenerations of K3 surfaces, giving an inductive
description of the so-called canonical wall structure of \cite{Walls}.
This, along with a number of other results proved using our gluing technology,
will be of use in \cite{K3}, which explores mirror symmetry for
K3 surfaces and uses the mirror construction of \cite{Walls} to
build geometrically meaningful compactifications of the moduli space
of K3 surfaces.

One of the key reasons our approach to gluing is applicable in this
case is that the degenerate varieties being considered only have
at worst triple points. It was originally observed by Brett Parker
in \cite{Parker} that this was a particularly amenable situation for
gluing.
\medskip

\emph{Conventions}:
All logarithmic schemes and stacks are defined over an algebraically closed 
field $\kk$ of characteristic $0$.
We follow the convention that if $X$ is a log scheme or
stack, then $\ul{X}$ is the underlying scheme or stack. We almost always
write $\cM_X$ for the sheaf of monoids on $X$ and $\alpha_X:\cM_X
\rightarrow\cO_X$ for the structure map. If $P$ is a
monoid, we write $P^{\vee}:=\Hom(P,\NN)$ and $P^*=\Hom(P,\ZZ)$.

\medskip

\emph{Acknowledgements}: Some of the material here, in an earlier form, was
originally written with the intention of appearing in \cite{ACGSII}.
So it has had a lot of influence from my co-authors of that project,
Dan Abramovich, Qile Chen and Bernd Siebert. In addition, the
last section was greatly influenced by discussions with my coauthors
on \cite{K3}, i.e., Paul Hacking, Sean Keel, and Bernd Siebert.
This paper has also benefited very much from discussions with 
Evgeny Goncharov, Sam Johnston, Xuanchun Lu, 
Dhruv Ranganathan, Yu Wang, and Peter Zaika,
and
comments from Helge Ruddat.
It was supported by the ERC Advance Grant MSAG.

\section{Preliminaries}
\label{sec:preliminaries}

\subsection{Tropical maps and moduli of punctured curves}

We will work with a relative target space $X\rightarrow B$, a
projective log smooth morphism. We further assume
that the log structure on $X$ is Zariski, and that $X$ satisfies
assumptions required to guarantee finite type moduli spaces
of punctured curves. At the moment, \cite{ACGSII} requires that
$\overline\shM_X^{\gp}\otimes_{\ZZ}\QQ$ be generated by global
sections, so what follows will be written with this assumption.
However, see \cite{SJ} for finiteness results without this condition.
Typically $B$ itself is taken to be an affine scheme and 
$B$ is either log smooth over $\Spec\kk$ 
or a log point $\Spec(Q_B\rightarrow \kk)$. 

\medskip

We briefly review notation from \cite{ACGSI,ACGSII} for tropical maps
to $\Sigma(X)$ and punctured log maps to $X$
as developed in \cite{JAMS}, \cite[\S2.5]{ACGSI} and \cite[\S2.2]{ACGSII}.

In what follows, $\mathbf{Cones}$ denotes the category of
rational polyhedral cones with integral structure, i.e., objects
are rational polyhedral cones $\omega \subseteq 
N_{\omega}\otimes_{\ZZ}\RR$ for
$N_{\omega}$
the lattice of integral tangent vectors to $\omega$.
Morphisms are maps of cones induced by maps of the corresponding lattices.
We write $\omega_{\ZZ}=\omega\cap N_{\omega}$ for the set of integral
points of $\omega$.

A \emph{generalized cone complex} is a topological space with a presentation
as the colimit of an arbitrary diagram in the category 
$\mathbf{Cones}$ with all morphisms being face morphisms. If $\Sigma$
is such a generalized cone complex, we write $\sigma\in \Sigma$ if $\sigma$ is
a cone in the presentation and $|\Sigma|$ for the underlying topological
space. A morphism of generalized cone complexes is a continuous map
$f:|\Sigma|\rightarrow |\Sigma'|$ such that for each
$\sigma\in \Sigma$, the induced map $\sigma\rightarrow |\Sigma'|$
factors through a cone map $\sigma\rightarrow\sigma'\in\Sigma'$.

There is a functor from fine saturated log schemes to generalized
cone complexes, written as $X\mapsto \Sigma(X)$. There is a one-to-one
correspondence between elements in the presentation $\Sigma(X)$
and logarithmic strata of $X$. If $\shM_X$ denotes the log structure
on $X$ with ghost sheaf $\overline\shM_X$, 
and $\bar\eta$ is a geometric generic point of a log stratum, then the
corresponding cone is $\Hom(\overline\shM_{X,\bar\eta},\RR_{\ge 0})$.
If $\sigma\in\Sigma(X)$, we write $X_{\sigma}\subseteq X$ for the corresponding
(closed) stratum.

We consider graphs $G$, with sets of vertices $V(G)$, edges $E(G)$ and legs
$L(G)$. In what follows, we will frequently confuse $G$ with its topological
realisation $|G|$. Legs will correspond to marked or punctured points of
punctured curves, and are rays in the marked case and compact line segments in
the punctured case. We view a compact leg as having only one vertex. An abstract
tropical curve over $\omega\in \mathbf{Cones}$ is data $(G, {\mathbf g}, \ell)$
where ${\mathbf g}:V(G)\rightarrow\NN$ is a genus function and
$\ell:E(G)\rightarrow \Hom(\omega_{\ZZ},\NN)\setminus\{0\}$ determines edge
lengths. 

Associated to the data
$(G,\ell)$ is a generalized cone complex 
$\Gamma(G,\ell)$ along with a morphism of cone complexes
$\Gamma(G,\ell)\rightarrow \omega$ with fibre over $s\in\Int(\omega)$
being a tropical curve, i.e., a metric graph, with underlying graph
$G$ and affine edge length of $E\in E(G)$ being $\ell(E)(s)\in \RR_{\ge 0}$.
Associated to each vertex $v\in V(G)$
of $G$ is a copy $\omega_v$ of $\omega$
in $\Gamma(G,\ell)$. Associated to each edge or leg $E\in E(G) \cup L(G)$
is a cone $\omega_E \in \Gamma(G,\ell)$ with
$\omega_E\subseteq \omega\times\RR_{\ge 0}$ and the map to $\omega$
given by projection onto the first coordinate. This projection
fibres $\omega_E$ in compact intervals or rays over $\omega$
(rays only in the case of a leg representing a marked point).

A \emph{family of tropical maps}
to $\Sigma(X)$ over $\omega\in \mathbf{Cones}$ is a morphism of cone complexes
\[
h:\Gamma(G,\ell)\rightarrow \Sigma(X).
\]
If $s\in\Int(\omega)$, we may view $G$ as the fibre of
$\Gamma(G,\ell)\rightarrow \omega$ over $s$ as a metric graph, and
write
\[
h_s:G\rightarrow \Sigma(X)
\]
for the corresponding tropical map with domain $G$.
The \emph{type} of such a family consists of the data
$\tau:=(G,\mathbf{g},\bsigma,\mathbf{u})$ where
\[
\bsigma:V(G)\cup E(G)\cup L(G)\rightarrow \Sigma(X)
\]
associates to $x\in V(G)\cup E(G)\cup L(G)$ the minimal
cone of $\Sigma(X)$ containing $h(\omega_x)$. Further,
$\mathbf{u}$ associates to each (oriented) edge or leg $E\in E(G)\cup L(G)$
the corresponding \emph{contact order} $\mathbf{u}(E)\in N_{\bsigma(E)}$,
the image of the tangent vector $(0,1)\in N_{\omega_E}=
N_{\omega}\oplus\ZZ$ under the map $h$. Legs are always oriented
away from their unique vertex.

As we shall only consider tropicalizations of
pre-stable punctured curves (see \cite[Def.\ 2.6]{ACGSII},
following \cite[Prop.\ 2.23]{ACGSII} we may assume that for $L\in L(G)$
with adjacent vertex $v\in V(G)$
giving $\omega_L,\omega_v\subseteq \Gamma(G,\ell)$, we have
\begin{equation}
\label{eq:leg}
h(\omega_L)=(h(\omega_v)+\RR_{\ge 0}{\bf u}(L))\cap \bsigma(L)
\subseteq N_{\bsigma(L),\RR}.
\end{equation}
In other words, the images of legs extend as far as possible inside their
cones.

A \emph{decorated type} is data $\btau=(G,\mathbf{g},
\bsigma,\mathbf{u},\mathbf{A})$
where $\mathbf{A}:V(G)\rightarrow H_2(X)$ associates a curve class to
each vertex of $G$. The \emph{total curve class} of ${\bf A}$
is $A=\sum_{v\in V(G)} {\bf A}(v)$.\footnote{We recall that $H_2(X)$
represents some choice of group of curve classes. It could be integral
homology of $X$, or the group of curve classes modulo algebraic or
numerical equivalence, but other choices are also possible. See e.g.,
\cite[\S2.3.8]{ACGSI} for a discussion.}

We also have a notion of a contraction morphism of types
$\phi:\tau\rightarrow \tau'$, see \cite[Def.\ 2.24]{ACGSI}. This
is a contraction of edges on the underlying graphs, and the
additional data satisfies some relations as follows.
If $x\in V(G)\cup E(G)\cup L(G)$, then $\bsigma'(\phi(x))
\subseteq \bsigma(x)$ (if $x$ is an edge, it may be contracted
to a vertex by $\phi$). Further, if $E\in E(G)\cup L(G)$
then $\mathbf{u}(E)=\mathbf{u}'(\phi(E))$ under the inclusion
$N_{\bsigma'(\phi(E))}\subseteq N_{\bsigma(E)}$, provided that $E$ is
not an edge contracted by $\phi$.

We say a type $\tau$ is \emph{realizable} if there exists a
family of tropical maps to $\Sigma(X)$ of type $\tau$.
We also say $\btau=(\tau,{\bf A})$ is realizable if $\tau$ is
realizable.  In this paper, we will only
deal with realizable types. As a consequence, we will not need the more
general notion of global type discussed in \cite[\S3]{ACGSII}.
However, in the case that there are no punctures, but only marked points,
we will use the notion of a \emph{class of logarithmic map}
$\beta$, consisting of data of a genus $g$, a curve class
$A$, and contact orders $u_i\in \Int(\sigma_i)$ with
$\sigma_i\in\Sigma(X)$ for $1\le i\le k$
for $k$ marked points. This may be viewed as a decorated type $(G,\mathbf{g},
\bsigma,\mathbf{u},\mathbf{A})$ where the underlying graph $G$
has only one vertex $v$ and no edges, $\mathbf{g}(v)=g$,
$\bsigma(v)=\{0\}$, $\bsigma(L_i)=\sigma_i$, $\mathbf{u}(L_i)=u_i$,
and $\mathbf{A}(v)=A$.

If a type $\tau$ is realizable, then there is a universal family of
tropical maps of type $\tau$, parameterized by an object of
$\mathbf{Cones}$. Hopefully without confusion, we will generally write
this cone as $\tau$. Hence we have a cone complex $\Gamma(G,\ell)$ equipped
with a map to $\tau$ and a map of cone complexes $h=h_{\tau}:\Gamma(G,\ell)
\rightarrow \Sigma(X)$. Generally we write $h$ rather than $h_{\tau}$
when unambiguous. Note that for each $x\in E(G)\cup L(G)\cup V(G)$,
we thus obtain $\tau_x \in \Gamma(G,\ell)$ the corresponding cone.

\medskip
We write $\cA_X$ for the Artin fan of $X$, see
\cite{ACMW17}, as well as \cite[\S2.2]{ACGSI} for a summary. With
$X\rightarrow B$ log smooth with $X$ Zariski, we obtain a morphism
of Artin fans $\shA_X\rightarrow\shA_B$ and define
\[
\shX:= \shA_X\times_{\shA_B} B.
\]

We refer to \cite[Defs.\ 2.11, 2.14, 2.15]{ACGSII} for the notion of a family
$\pi:C^{\circ}\rightarrow W$ of punctured curves and pre-stable or
stable punctured log maps $f:C^{\circ}/W\rightarrow X$ or
$f:C^{\circ}/W\rightarrow \cX$ defined over $B$. 

Given a punctured log map with domain $C^\circ\rightarrow W$ and
$W=\Spec(Q\rightarrow\kappa)$ for $\kappa$ an algebraically closed
field and target $X$ or $\cX$, we obtain by
functoriality of tropicalizations a family of tropical maps
\begin{equation}
\label{eq:trop diag}
\xymatrix@C=30pt
{
\Sigma(C)=\Gamma(G,\ell)\ar[d]\ar[r]& \Sigma(X)\ar[d]\\
\Sigma(W)=\omega=Q^{\vee}_{\RR}\ar[r] & \Sigma(B)
}
\end{equation}
parameterized by $W$.
The \emph{type} of the punctured map is then the type
$\tau=(G,\mathbf{g},\bsigma,\mathbf{u})$
of this family of tropical maps. We recall that the punctured
map $f:C^{\circ}/W\rightarrow X$ is \emph{basic} if \eqref{eq:trop diag}
is the universal family of tropical maps of type $\tau$.

Given a type $\tau=(G,\mathbf{g},\bsigma,\mathbf{u})$, 
\cite[Def.\ 3.4]{ACGSII}
defines the notion of a \emph{marking} or \emph{weak marking} 
of a punctured map by $\tau$.\footnote{The cited reference applies to
global types, but by \cite[Lem.~3.5]{ACGSII}, giving a realizable global
type is the same as giving a realizable type. Since all our types are
realizable here, we ignore the notion of global type.}
Roughly, a \emph{weak marking} of a punctured map $f:C^{\circ}/W \rightarrow X$
involves the following information.
(1) A marking of the underlying domain curve $\ul{C}$ by $G$. In other words, 
we have a pre-stable curve $\ul{C}_v$ for each $v\in V(G)$ of 
genus $\mathbf{g}(v)$, a marked point $p_L\in \ul{C}_v$ for each leg
$L\in L(G)$ adjacent to $v$, and marked points $q_{E,1}, q_{E,2}$
in $\ul{C}_{v_1},\ul{C}_{v_2}$ for each edge $E$ connecting $v_1$ to $v_2$.
Further, $\ul{C}$ is obtained as a marked curve by identifying 
pairs $q_{E,1},q_{E,2}$ for all $E\in E(G)$.
(2) For each subcurve or punctured or nodal
section $Z$ of $\ul{C}$, indexed by an element $x\in V(G)\cup L(G)\cup E(G)$, 
the
morphism $\ul{f}|_Z$ factors through the closed stratum $X_{\bsigma(x)}$
of $X$. (3) For any geometric point $\bar w\rightarrow W$
giving a curve of type $\tau_{\bar w}=(G_{\bar w}, \mathbf{g}_{\bar w},
\bsigma_{\bar w}, \mathbf{u}_{\bar w})$, 
the contraction morphism $G_{\bar w}\rightarrow G$
induced by the marking of the domain yields
a contraction morphism of types $\tau_{\bar w}\rightarrow \tau$.

A \emph{marking} of a punctured map $f:C^{\circ}/W \rightarrow X$
is a weak marking satisfying an additional requirement that a certain
natural monoid ideal on $\shM_W$ defines an idealized log structure on
$W$. See \cite[Def.~3.4]{ACGSII} for full details.

In either case, if further $\btau=(\tau,\mathbf{A})$ is a decoration
of $\tau$, then $f:C^\circ/W\rightarrow X$ is (weakly) $\btau$-marked
if in addition to being (weakly) $\tau$-marked, for each $v\in V(G)$,
the curve class associated the the stable map $\ul{f}$ restricted
to the subcurve indexed by $v$ is $\mathbf{A}(v)$.

In particular, this gives rise to the following moduli spaces:
\begin{enumerate}
\item $\scrM(X,\btau)$ (resp.\ $\scrM'(X,\btau)$) the moduli space
of (weakly) $\btau$-marked stable punctured maps.
\item $\fM(\shX,\tau)$ (resp.\ $\fM'(\shX,\tau)$) the moduli space
of (weakly) $\btau$-marked punctured maps to $\shX$.
\item $\fM(\shX,\btau)$ (resp.\ $\fM'(\shX,\btau)$) the moduli space
of (weakly) $\btau$-marked punctured maps to $\shX$. Note here
that while curve classes in $\shX$ are meaningless, the decoration
$\mathbf{A}$ on $\tau$ affects the notion of isomorphism in the
categories $\foM(\shX,\btau)$ or $\foM'(\shX,\btau)$, and there is
an \'etale morphism $\foM(\shX,\btau)\rightarrow \foM(\shX,\tau)$.
\end{enumerate}

In general, we are always working over $B$, but when we need to be more
precise, we write $\scrM(X/B,\btau)$. We say type $\tau$ is
\emph{realizable over $B$} if $\foM(\shX/B,\tau)$ has a geometric
point corresponding to a map of type $\tau$, see \cite[Def.~3.28]{ACGSII}.
By \cite[Prop.~3.29]{ACGSII}, this is equivalent to a combinatorial condition
requiring $\tau$ to be realisable by a family of tropical maps defined
over $\Sigma(B)$ with an additional condition which is easily checked.

By \cite[Thm.\ 3.10]{ACGSII}, all of the above moduli spaces
are algebraic stacks and $\scrM(X,\btau)$ and $\scrM'(X,\btau)$ 
are Deligne-Mumford. Further, there are natural morphisms
\begin{equation}
\label{def:varepsilon}
\begin{split}
\varepsilon:\scrM(X,\btau) & \rightarrow \foM(\shX,\btau)\\
\varepsilon:\scrM'(X,\btau) & \rightarrow \foM'(\shX,\btau)
\end{split}
\end{equation}
given by composing a punctured log map $C^{\circ}\rightarrow X$ with the
canonical map $X\rightarrow\shX$. \cite[\S4]{ACGSII} then gives a perfect
relative obstruction theory for $\varepsilon$. 

In general, it appears that the $\btau$-marked moduli spaces are more
important than the weakly $\btau$-marked moduli spaces. In particular,
\cite[Prop.~3.33]{ACGSII} shows that $\foM(\shX,\tau)$ is a closed
substack of $\foM'(\shX,\tau)$ defined by a nilpotent ideal. Further,
in the cases of greatest interest for this paper 
($\btau$ realizable, $B$ log smooth over $\Spec\kk$ or $B=\Spec\kk^{\dagger}$, 
the standard log point),
$\foM(\shX,\tau)$ is actually reduced and pure-dimensional, see
\cite[Prop.~3.30]{ACGSII}.
In fact, while $\foM(\shX,\tau)$ may be quite poorly behaved globally, 
it has a simple local structure coming from the fact that it is idealized log
smooth over $B$, see \cite[Thm.\ 3.25, Rem.\ 3.27]{ACGSII}. 
The main point for including the weakly marked moduli spaces is that
they naturally occur in the gluing formalism.

If there is a contraction morphism between decorated global types
$\phi:\btau\rightarrow \btau'$, we obtain a forgetful
map $\scrM(X,\btau)\rightarrow \scrM(X,\btau')$.
This gives rise to a stratified description of these moduli spaces,
see \cite[Rem.\ 3.31]{ACGSII}.

\medskip

If $\btau=(G,\bsigma,{\bf u},\mathbf{A})$ denotes a
choice of decorated type, and
$I\subseteq
E(G)\cup L(G)$ is a collection of edges and legs, then we write
\[
\foM^{\ev}(\shX,\btau)=\foM^{\ev(I)}(\shX,\btau)
:= \foM(\shX,\btau)\times_{\ul{\shX}^I} \ul{X}^I.
\]
Here $\ul{\shX}^I$ denotes the product of $|I|$ copies of $\ul{\shX}$
over $\ul{S}$, and similarly $\ul{X}^I$; the morphism
$\foM(\shX,\btau)\rightarrow \ul{\shX}^I$ is given by evaluation
at the nodes and punctured points indexed by elements of $I$,
and $\ul{X}^I\rightarrow \ul{\shX}^I$ is induced by the canonical smooth
map $\ul{X}\rightarrow\ul{\shX}$. The map $\varepsilon$ then factors
as
\begin{equation}
\label{eq:epsilon ev}
\xymatrix@C=30pt
{
\scrM(X,\btau)\ar[r]^{\varepsilon^{\ev}} & \foM^{\ev}(\shX,\btau)
\ar[r] & \foM(\shX,\btau).
}
\end{equation}
The second morphism is smooth, while $\varepsilon^{\ev}$ also possesses
a relative obstruction theory compatible with the morphism
$\varepsilon$ of \eqref{def:varepsilon}, see
\cite[\S4.2]{ACGSII}.

\medskip

We end by reviewing a basic result which encodes a generalisation of
the tropical balancing condition. The following is
\cite[Prop.~2.29]{ACGSII}.

\begin{proposition}
\label{prop:balancing}
Suppose given a punctured map $f:C^\circ/W\rightarrow X$ with $W$
a log point. Let $\tau=(G,\bsigma,\mathbf{u})$ be the corresponding
tropical type. If $v\in V(G)$, let $\ul{C}_v\subseteq \ul{C}$
be the corresponding irreducible component, and let 
$E_1,\ldots,E_n$ be the edges and
legs adjacent to $v$, oriented away from $v$. Let $s\in \Gamma(X,\overline
\shM_X^{\gp})$, and let $\shL_s$ be the corresponding line bundle,
i.e., the line bundle associated to the torsor given by the inverse image
of $s$ under the quotient map $\shM_X\rightarrow \overline\shM_X$.
For an integral tangent vector $v$ of $\sigma=\Hom(\overline\shM_{X,x},
\RR_{\ge 0})\in \Sigma(X)$, we write $\langle v,s\rangle$ for
the evaluation of $v$, as an element of $\Hom(\overline\shM_{X,x},\ZZ)$,
on the germ of $s$ at $x$. Then 
\[
\deg (\ul{f}^*\shL_s)|_{\ul{C}_v} = -\sum_{i=1}^n \langle \mathbf{u}(E_i),
s\rangle.
\]
\end{proposition}

In the case that $\ul{X}$ is non-singular and the log structure on $X$
comes from a simple normal crossings divisor $D=D_1+\cdots+D_s$ 
such that all intersections of irreducible components of $D$ are
connected, then this has a particularly nice interpretation as saying
that the contact orders determine intersection numbers of the curve
class with the divisors $D_i$. More precisely, in this case 
we may view $\Sigma(X)$ as in \cite[Ex.~1.4]{Assoc} as follows.
Let $\Div_D(X)=\bigoplus_i \ZZ D_i$ be the group of divisors supported
on $D$, and write $\Div_D(X)^*=\Hom(\Div_D(X),\ZZ)=\bigoplus_i \ZZ D_i^*$
for the dual lattice and $\Div_D(X)^*_{\RR}=\Div_D(X)^*\otimes_{\ZZ}\RR$.
Then
\[
\Sigma(X) = \left\{\sum_{i\in I}\RR_{\ge 0}D_i^*\Big| 
\hbox{$I\subseteq\{1,\ldots,s\}$ an index set with $\bigcap_{i\in I}D_i
\not=\emptyset$}\right\}.
\]
Thus a contact order $u$ may be written as $\sum_i a_i D_i^*$, with
$a_i\in\ZZ$, with $a_i$ denoting the order of tangency with $D_i$.
We then have the immediate corollary

\begin{corollary}
\label{cor:114}
Let $X,D$ be as above, and $f, \tau, E_1\ldots,E_n$ as in 
Proposition~\ref{prop:balancing}. Write $\mathbf{u}(E_i)=\sum_j a_{ij}D_j^*$.
Then 
\[
\deg (\ul{f}^*\cO_X(D_j))|_{\ul{C}_v}=\sum_{i=1}^n a_{ij}.
\]
\end{corollary}

\subsection{The gluing formalism}

We may now describe the key gluing formalism of \cite{ACGSII}.
We begin with the \emph{standard gluing situation}.

\begin{notation}[The standard gluing situation]
\label{not:standard gluing}
We fix a target $X\rightarrow B$, a
proper log smooth morphism. We further assume
that the log structure on $X$ is Zariski, and that $X$ satisfies
assumptions required to guarantee finite type moduli spaces
of punctured curves. Further we assume that $B$ is either
log smooth over $\Spec\kk$ or $B=\Spec\kk^{\dagger}$, the standard
log point. Fix a realizable
type $\tau=(G,\mathbf{g},\bsigma,\mathbf{u})$ of tropical
map to $\Sigma(X)/\Sigma(B)$. We select a set of splitting edges
$\mathbf{E}\subseteq E(G)$, and let $G_1,\ldots,G_r$ be the connected
components of the graph obtained by splitting $G$ at the edges
of $\mathbf{E}$, i.e., replacing each edge $E\in \mathbf{E}$ with
endpoints $v_1,v_2$ with two legs with endpoints $v_1,v_2$ respectively.
We write these two legs as flags $(E,v_1)$ and $(E,v_2)$.
We then let $\tau_1,\ldots,\tau_r$ be the induced
set of decorated types with underlying graphs $G_1,\ldots,G_r$.
Let $\mathbf{L}\subseteq \bigcup_{i=1}^r L(G_i)$ be the subset of
all legs obtained from splitting edges, and $\mathbf{L}_i=\mathbf{L}\cap
L(G_i)$.
For $v\in V(G)$, let $i(v)\in \{1,\ldots,r\}$ denote the connected
component $G_i$ containing $v$.
\end{notation}

For each $E\in\bE$ denote by $\fM'_E(\cX,\tau)$ the image of the nodal section
$s_E:\ul{\fM'}(\cX,\tau)\arr \ul{\fC'}^\circ(\cX,\tau)$ with the restriction
of the log structure on the universal domain $\fC'^\circ(\cX,\tau)$.
Denote further by $\tfM'(\cX,\tau)$ the fs fiber product
\begin{equation}
\label{Eqn: tilde log structure}
\tfM'(\cX,\tau) = \fM'_{E_1}(\cX,\tau)\times^\fs_{\fM'(\cX,\tau)} \cdots
\times^\fs_{\fM'(\cX,\tau)} \fM'_{E_r}(\cX,\tau),
\end{equation}
where $E_1,\ldots,E_r\in E(G)$ are the edges in $\bE$. With
this enlarged log structure, the pull-back $\tfC'^\circ(\cX,\tau)\arr
\tfM'(\cX,\tau)$ of the universal domain has sections $\tilde s_E$, $E\in
\bE$, in the category of log stacks. Moreover, $\ul\ev_\bE$ lifts to a
logarithmic evaluation morphism
\begin{equation}
\label{Eqn: ev_bE}
\textstyle
\ev_\bE: \tfM'(\cX,\tau) \arr \prod_{E\in \bE} \cX,
\end{equation}
with $E$-component equal to $\tilde f\circ \tilde s_E$ for $\tilde f:
\tfC'^\circ(\cX,\tau)\arr \cX$ the universal punctured morphism.

Similarly, for each of the types $\tau_i=(G_i,\bg_i,\bsigma_i,\bu_i)$
obtained by splitting and $L\in L(G_i)$, denote by $\fM'_L(\cX,\tau_i)$ the
image of the punctured section $s_L:\ul\fM'(\cX,\tau_i)\arr
\ul{\fC'}^\circ(\cX,\tau_0)$ defined by $L$, again endowed with the pull-back
of the log structure on ${\fC'}^\circ(\cX,\tau_0)$. With $L_1,\ldots,L_s$ the
legs of $\mathbf{L}_i$, define the stack
\[
\widetilde\fM'(\cX,\tau_i)=
\big(\fM'_{L_1}(\cX,\tau_i)\times^\mathrm{f}_{\fM'(\cX,\tau_i)}
\cdots \times^\mathrm{f}_{\fM'(\cX,\tau_i)} \fM'_{L_s}(\cX,\tau_i)\big)^{\mathrm{sat}},
\]
where $\mathrm{sat}$ denotes saturation and 
$\times^{\mathrm{f}}$ denotes fibre product in the category of fine
log stacks. This stack differs from $\fM'(\cX,\tau_i)$ by
adding the pull-back of the log structure of each puncture, so that the
pull-back $\tfC'^\circ(\cX,\tau_i) \arr
\widetilde\fM'\cX,\tau_i)$ of the universal curve now has punctured
sections in the category of log stacks. We define the evaluation morphism
\begin{equation}
\label{Eqn: ev_bL}
\textstyle
\ev_\bL: \prod_{i=1}^r\tfM'(\cX,\tau_i) \arr \prod_{E\in\bE} \cX\times \cX,
\end{equation}
by taking as $E$-component the evaluation at the corresponding two sections
$s_{E,v_1},s_{E,v_2}$, where $v_1,v_2$ are the endpoints of $E$.
Note that this involves a choice of ordering of the endpoints of $E$,
i.e., a choice of orientation of $E$.

It is worth noting the following (see \cite[Prop.~5.7]{ACGSII}).

\begin{proposition}
\label{prop:reduced stacks}
The canonical map 
$\widetilde\foM'(\shX,\tau)\rightarrow \foM'(\shX,\tau)$ induces
an isomorphism of underlying stacks, while the canonical maps
$\widetilde\foM'(\shX,\tau_i)\rightarrow \foM'(\shX,\tau_i)$ induces
an isomorphism on reductions.
\end{proposition}

There are evaluation space versions of this. We set
\[
\widetilde\foM'^{\ev}(\shX,\tau_i)=\foM'^{\ev(\mathbf{L}_i)}(\shX,
\tau_i)\times_{\foM'(\shX,\tau_i)} \widetilde\foM'(\shX,\tau_i)
\]
and
\[
\widetilde\foM'^{\ev}(\shX,\tau)=\foM'^{\ev(\mathbf{E})}(\shX,
\tau_i)\times_{\foM'(\shX,\tau_i)} \widetilde\foM'(\shX,\tau_i).
\]

There is a natural splitting map
\begin{equation}
\label{eq:bad splitting}
\delta_{\foM}:\foM'(\shX,\tau)\rightarrow \prod_{i=1}^r \foM'(\shX,\tau_i)
\end{equation}
as defined in \cite[Prop.~5.4]{ACGSII}, taking a $\tau$-marked map and splitting it at the
nodes marked by the edges in $\mathbf{E}$. This lifts to 
\begin{equation}
\label{Eqn: splitting morphism}
\delta_{\foM}:\widetilde\foM'^{\ev(\mathbf{E})}(\shX,\tau)\rightarrow 
\prod_{i=1}^r \widetilde\foM'^{\ev(\mathbf{L}_i)}(\shX,\tau_i),
\end{equation}
and is shown to be a finite morphism in \cite[Cor.~5.15]{ACGSII}. 
Passing to the $\ev$-spaces
is key here: the morphism in \eqref{eq:bad splitting} is rarely finite
or even proper. This also gives an upgrading of the evaluation morphisms:
\begin{align}
\label{eq:ev}
\begin{split}
\ev_\bE: & \tfM'^{\ev(\mathbf{E})}(\cX,\tau) \arr \prod_{E\in \bE} X,\\
\ev_\bL: & \prod_{i=1}^r\tfM'^{\ev(\mathbf{L}_i)}(\cX,\tau_i) \arr \prod_{E\in\bE} X\times X.
\end{split}
\end{align}

We review the key gluing results of \cite{ACGSII}. We first describe
the glued moduli space as an fs fibre product:

\begin{theorem}
\label{Thm: Gluing theorem}
Suppose given a gluing situation as in Notation \ref{not:standard gluing}. 
Then the commutative diagram
\[
\xymatrix{
\tfM'^{\ev(\mathbf{L})}(\shX,\tau) \ar[r]^(.43){\delta_\fM}\ar[d]_{\ev_\bE} &
\prod_{i=1}^r \tfM'^{\ev(\mathbf{L}_i)}(\cX,\tau_i)\ar[d]^{\ev_\bL}\\
\prod_{E\in \bE} X\ar[r]^(.43)\Delta& \prod_{E\in \bE} X\times X 
}
\]
with $\Delta$ the product of diagonal embeddings and the other arrows defined in
\eqref{Eqn: splitting morphism} and \eqref{eq:ev}, is
cartesian in the category of fs log stacks. We remind the reader that all
products in this square are taken over $B$.

An analogous statement holds for $\tau$ replaced by a decorated global type
$\btau=(\tau,\bA)$, or replacing $\widetilde\foM'^{\ev}(\shX,\tau)$,
$\widetilde\foM'^{\ev}(\shX,\tau_i)$ with the analogous moduli spaces
of stable maps to $X/B$, $\widetilde\scrM'(X,\tau)$, $\widetilde\scrM'(X,
\tau_i)$.
\end{theorem}

We remark that in this theorem, the fact that all
fibre products are over $B$ can make the actual calculation of fibre
products more difficult than they need to be. But it is often enough
to work over $\Spec\kk$, as the following proposition 
(see \cite[Prop.~5.13]{ACGSII}) shows:

\begin{proposition}
\label{prop:relative versus absolute}
Let $B$ be an affine log scheme equipped with a global chart
$P\rightarrow \cM_B$ inducing an isomorphism $P\cong \Gamma(B,
\overline{\cM}_B)$.
Let $\tau$ be a type of tropical map for $X/B$,
with underlying graph connected.
Then there are isomorphisms $\fM(\cX,\tau)\cong \fM(\cX/\Spec\kk,
\tau)$ and $\scrM(X,\tau)\cong \scrM(X/\Spec\kk,\tau)$.\footnote{We recall
our convention that $\scrM(X,\tau)=\scrM(X/B,\tau)$.}
\end{proposition}

Finally, to make contact with stable punctured maps to $X$, we have
\cite[Prop.~5.17]{ACGSII}:

\begin{theorem}
\label{Prop: Gluing via evaluation spaces}
In the situation of Theorem~\ref{Thm: Gluing theorem} there is a cartesian
diagram
\begin{equation}
\label{diag:evaluation-gluing}
\vcenter{\xymatrix@C=30pt
{
\scrM(X,\tau) \ar[r]^(.45){\delta}\ar[d]_{\varepsilon^{\ev}} &
\prod_{i=1}^r \scrM(X,\tau_i)\ar[d]^{\hat\varepsilon=\prod_i\varepsilon^{\ev}_i}\\
\fM^{\ev(\mathbf{E})}(\cX,\tau) \ar[r]_(.45){\delta^\ev} &
\prod_{i=1}^r \fM^{\ev(\mathbf{L}_i)}(\cX,\tau_i)
}}
\end{equation}
with horizontal arrows the canonical splitting maps 
(see \cite[Prop.~5.4]{ACGSII}), and the vertical arrows the canonical 
strict morphisms of \eqref{eq:epsilon ev}. 

Analogous statements hold for decorated and for weakly marked versions of the
moduli stacks, see \cite[Def.~3.8]{ACGSII}.
\end{theorem}

The evaluation spaces $\foM^{\ev}(\shX,\tau)$ etc.\ play a crucial role
here. First, if instead one used the spaces $\foM(\shX,\tau)$, there
would be no way to obtain a Cartesian diagram, as the splitting map
on the level of punctured maps to Artin fans has no way to impose a matching
condition at the schematic level. Using the evaluation spaces allows
$\delta^{\ev}$ to impose both schematic and logarithmic matching conditions.

Second, the splitting map at the level of the spaces of punctured maps
to Artin fans is very poorly behaved, being neither representable nor
proper. However, $\delta^{\ev}$, as this theorem states, is in fact
finite and representable. Hence we may use it to push-forward Chow classes,
and in particular, as \cite[Thm.~5.19]{ACGSII} points out,
the compatibility of obstruction theories then allows a calculation
\[
\delta_*([\scrM(X,\tau)]^{\virt})=\hat\varepsilon^!\delta^{\ev}_*
[\foM^{\ev}(\shX,\tau)]
\]
of the virtual fundamental class of $\scrM(X,\tau)$. 

Thus the main task is finding a useful expression for $\delta^{\ev}_*
[\foM^{\ev}(\shX,\btau)]$. While the Artin stacks of the type
$\foM^{\ev}(\shX,\tau)$ may seem very forbidding, in a certain sense
they are very well-behaved: they are idealized log smooth over $B$,
see \cite[Thm.~3.25]{ACGSII}. This means that there are local descriptions 
of these stacks
as unions of strata of toric varieties. Further, these local descriptions
can be determined very explicitly from the tropical description of the
types $\tau$ and $\tau_i$.

This effective description of these moduli spaces has led Yixian Wu 
\cite{Wu},
in the case that all gluing strata are toric, to give an effective
formula for the Chow class $\delta_*[\foM^{\ev}(\shX,\tau)]$ as
a weighted sum of strata of $\prod_{i=1}^r \foM^{\ev}(\shX,\tau_i)$.
This formula has already proved to be very useful, see e.g.,
\cite{Walls}.

Here we wish to develop a gluing formalism in a complementary direction,
which, although very far from general, is also useful.

\section{Four-point lemmas: fibre products of log points}
\label{sec:four point}

The key point is to understand the gluing fibre
diagram of Theorem \ref{Thm: Gluing theorem}
by studying fibre products of log points. We carry
this study out in this section; unfortunately, this is rather dry.
An fs fibre product of log points can be quite subtle. Even 
determining whether such a fibre product is non-empty is difficult, as
a ``four-point lemma'' does not hold widely in log geometry, see
\cite[III Prop.\ 2.2.3]{Ogus} for some results. 
Here we will generally be interested in the number of connected components
of a fibre product for application to specific gluing situations.

We start with a small lemma:

\begin{lemma}
\label{lem:point saturation}
Let $W:=\Spec(Q\rightarrow\kappa)$ be a log point with $Q$ a sharp fine 
monoid. Then $W^{\mathrm{sat}}$ is a disjoint union of 
possibly non-reduced points, and
the number of connected components of $W^{\mathrm{sat}}$
is $|(Q^{\gp})_{\tors}|$.
\end{lemma}

\begin{proof}
By construction of the saturation, \cite[III Prop.\ 2.1.5]{Ogus},
\[
W^{\mathrm{sat}}=W\times_{\Spec\kappa[Q]} \Spec\kappa[Q^{\mathrm{sat}}],
\]
where $Q^{\mathrm{sat}}$ is the saturation of $Q$ inside $Q^{\gp}$.
Note that the morphism $W\rightarrow \Spec\kappa[Q]$
identifies $W$ with the closed point of $\Spec\kappa[Q]$ corresponding to the
maximal monomial ideal $\mathfrak{m}
=\langle z^q\,|\, q\in Q\setminus\{0\}\rangle$.

The monomial ideal $I\subseteq \kappa[Q^{\mathrm{sat}}]$ generated by the 
image of $\mathfrak{m}$
then satisfies $\sqrt{I}=\mathfrak{m}'=\langle z^q\,|\, q\in Q^{\mathrm{sat}}
\setminus (Q^{\mathrm{sat}})^{\times}\rangle$. Indeed, for any element 
$q\in Q^{\mathrm{sat}}\setminus (Q^{\mathrm{sat}})^{\times}$, there
exists a positive integer $n$ such that $nq\in Q$. Further, $nq\not=0$
since otherwise $q$ is torsion in $Q^{\mathrm{sat}}$ and hence invertible.
Thus $z^{nq}$ is a generator of $\mathfrak{m}$, so
$z^{nq}\in I$. This shows that $\mathfrak{m}'\subseteq \sqrt{I}$.
The converse holds as $I\subseteq \mathfrak{m}'$ and $\mathfrak{m}'$
is a radical ideal, as $\kappa[Q^{\mathrm{sat}}]/\mathfrak{m}'
=\kappa[(Q^{\gp})_{\tors}]$ is reduced.

Thus we see that $(W^{\mathrm{sat}})_{\red}=\Spec \kappa[Q^{\mathrm{sat}}]
/\mathfrak{m}'=\Spec \kappa[(Q^{\gp})_{\tors}]$. However, the latter
consists of $|(Q^{\gp})_{\tors}|$ points.
\end{proof}

\begin{lemma}
\label{lem:reduced product}
Let $W_1,W_2,X$ be fs log schemes with morphisms $W_1,W_2\rightarrow X$.
Then there is a canonical isomorphism
\[
(W_{1,\red}\times^{\fs}_{X_{\red}} W_{2,\red})_{\red} \cong
(W_1\times_X^{\fs} W_2)_{\red}.
\]
\end{lemma} 

\begin{proof}
Considering the strict closed immersion $(W_1)_{\red}\rightarrow W_1$,
we obtain via base change a strict closed immersion 
$(W_1)_{\red}\times^{\fs}_X W_2\rightarrow W_1\times^{\fs}_X W_2$. 
Repeating with
$W_2$, we obtain a strict closed immersion 
$(W_1)_{\red}\times^{\fs}_X (W_2)_{\red} \rightarrow W_1\times^{\fs}_X W_2$. 
Since
the morphisms $(W_i)_{\red}\rightarrow X$ factor through $X_{\red}$,
the former log scheme is isomorphic to
$(W_1)_{\red}\times^{\fs}_{X_{\red}} (W_2)_{\red}$. 

Thus we obtain a strict closed immersion 
$(W_{1,\red}\times^{\fs}_{X_{\red}} W_{2,\red})_{\red} \rightarrow
(W_1\times_X^{\fs} W_2)_{\red}$, which we must now prove is an isomorphism.
We do this by showing this map induces a bijection on geometric points.
Indeed, the set of geometric points of
$W_1\times^{\fs}_X W_2$ and $(W_1\times^{\fs}_X W_2)_{\red}$ are
the same, and a strict geometric point 
point $\Spec(Q\rightarrow \kappa)\rightarrow W_1\times^{\fs}_X W_2$ clearly
induces a strict closed point 
$\Spec(Q\rightarrow\kappa)\rightarrow (W_1)_{\red}\times^{\fs}_{X_{\red}} 
(W_2)_{\red}$. This shows the claim.
\end{proof}

%
%

\begin{lemma}
\label{lem:point product}
Let $W_1, W_2, X$ be finite length connected fs log schemes over
$\Spec\kappa$ with $\kappa$ an algebraically closed field.
Write the ghost sheaf monoids as $Q_1, Q_2$ and $P$ respectively. Suppose
given morphisms $f_i:W_i\rightarrow X$ inducing $\theta_i=\bar f_i^{\flat}:
P\rightarrow Q_i$. Set
\[
\theta:=(\theta^{\gp}_1,-\theta^{\gp}_2):
P^{\gp}\rightarrow Q_1^{\gp}\oplus Q_2^{\gp}.
\]
Then
\begin{enumerate}
\item If the fs fibre product $W_1\times^{\fs}_X W_2$ is non-empty,
it has $|\coker(\theta)_{\mathrm{tors}}|$ connected components.
\item
$|\coker(\theta)_{\mathrm{tors}}|=|\coker(\theta^t)_{\mathrm{tors}}|,$
where $\theta^t:Q_1^*\oplus Q_2^*\rightarrow P^*$ denotes the transpose
map to $\theta$.
\end{enumerate}
\end{lemma}

\begin{proof}
(1) By Lemma \ref{lem:reduced product}, we may assume $W_i$ and $X$ are (reduced) log points.
According to the construction of the fs fibre product
in \cite[III, \S2.1]{Ogus}, we proceed
in a couple of steps. Let $W, W^{\mathrm{int}}$ and $W^{\mathrm{fs}}$ denote
the fibre product $W_1\times_X W_2$ in the category of log schemes, 
fine log schemes, and fs log schemes. Then $W^{\mathrm{int}}$ is the
integralization of $W$ and $W^{\mathrm{fs}}$ is the saturation of
$W^{\mathrm{int}}$.

First, $\ul{W}$ agrees with the fibre product $\ul{W}_1\times_{\ul{X}}\ul{W}_2
=\Spec\kappa$. Integralization
involves passing to a closed subscheme of $\Spec\kappa$.
Thus either $\ul{W}^{\mathrm{int}}=\Spec\kappa$ or is the empty scheme.
We rule out the latter as we have assumed that the fs fibre product
is non-empty. In this case, 
$W^{\mathrm{int}}=\Spec (Q\rightarrow \kappa)$,
where
\[
Q=Q_1\oplus^{\mathrm{fine}}_P Q_2
\]
and $\oplus^{\mathrm{fine}}$ denotes push-out in the category of fine
monoids. This is constructed (see \cite[I, Prop.\ 1.3.4]{Ogus} and
its proof)
as the fine submonoid of
\[
\coker\theta=Q_1^{\gp}\oplus_{P^\gp} Q_2^{\gp}=Q^{\gp}
\]
generated by the 
images of $Q_1$ and $Q_2$.

As $W^{\mathrm{fs}}$ is the saturation of $W^{\mathrm{int}}$,
(1) now follows from Lemma \ref{lem:point saturation}.

(2) is easy homological algebra:
$0\rightarrow \ker\theta \rightarrow P^{\gp}\rightarrow Q_1^{\gp}
\oplus Q_2^{\gp}\rightarrow \coker(\theta)\rightarrow 0$
is a free resolution of $\coker(\theta)$, and $\ext^1(\coker(\theta),
\ZZ)$ is isomorphic to the torsion part of $\coker(\theta)$. However, this
$\ext$ group is calculated as the middle cohomology of the complex
$Q_1^*\oplus Q_2^* \rightarrow P^*\rightarrow \ker(\theta)^*$,
and as the kernel of $P^*\rightarrow \ker(\theta)^*$ is a saturated
sublattice of $P^*$, the torsion part of $\coker(\theta^t)$ agrees
with $\ext^1(\coker(\theta),\ZZ)$.
\end{proof}

The following is a somewhat technically complicated criterion
for non-emptiness for $W_1\times^{\fs}_X W_2$, which is tailored for
our gluing needs.

\begin{lemma}
\label{lem:point product nonempty}
Let $W_i, X$ be as in Lemma \ref{lem:point product}, and suppose
also given log points $W_i', X'$ over $\Spec\kappa$ similarly with
maps $f_i':W_i'\rightarrow X'$. Write $Q_i', P'$ for the corresponding
monoids, $\theta_i':P'\rightarrow Q_i'$ for the induced maps.
Suppose further given a commutative
diagram
\begin{equation}
\label{eq:W1W2X}
\xymatrix@C=30pt
{
W_1\ar[d]_{g_1}\ar[r]^{f_1}&X\ar[d]_{g}&W_2\ar[l]_{f_2}\ar[d]^{g_2}\\
W_1'\ar[r]_{f_1'}&X'&W_2'\ar[l]^{f_2'}
}
\end{equation}
Suppose that:
\begin{enumerate}
\item $W_1'\times^{\fs}_{X'} W_2'$ is non-empty.
\item 
The projections
$Q_1^\vee\times_{P^{\vee}} Q_2^{\vee}\rightarrow Q_i^{\vee}$ have image
intersecting the interior of $Q_i^{\vee}$.
\item The maps $\bar g^{\flat}:P'\rightarrow P$ and
$\bar g^{\flat}_i:Q_i'\rightarrow Q_i$ are all injective,
and the map induced by $\theta$,
\begin{equation}
\label{eq:theta induced}
P^{\gp}/\bar g^{\flat}((P')^{\gp})
\rightarrow 
Q_1^{\gp}/\bar g_1^{\flat}((Q_1')^{\gp})
\oplus
Q_2^{\gp}/\bar g_2^{\flat}((Q_2')^{\gp})
\end{equation}
is injective.
\end{enumerate}
Then $W_1\times^{\fs}_X W_2$ is non-empty.
\end{lemma}

\begin{proof}
Given the hypotheses,
we will construct a morphism
$\Spec \kappa^{\dagger}\rightarrow W_1\times^{\fs}_X W_2$ from the standard
log point. 
Recall that giving a morphism of log points
$g:\Spec (Q\rightarrow \kappa) \rightarrow \Spec (P\rightarrow \kappa)$
is equivalent to giving $g^{\flat}:P\times \kappa^{\times}
\rightarrow Q\times \kappa^{\times}$ written as
\[
g^{\flat}(p,t)=(\bar g^{\flat}(p), \chi_g(p)\cdot t)
\]
for $\bar g^{\flat}:P\rightarrow Q$ a local homomorphism and 
$\chi_g:P\rightarrow \kappa^{\times}$ an arbitrary homomorphism.
Note also that $\theta:P\rightarrow Q$ being a local homomorphism can be 
characterized dually by the statement that $\theta^t(Q^{\vee})$ intersects
the interior of $P^{\vee}$.

First, let $Q=(Q_1\oplus^{\mathrm{fs}}_{P} Q_2)/\mathrm{tors}$.
Note that $Q$ would be the stalk of the ghost sheaf of any point of
$W_1\times^{\fs}_X W_2$ if this log scheme were non-empty. Then
$Q^{\vee}=Q_1^{\vee}\times_{P^{\vee}} Q_2^{\vee}$ 
(\cite[Proposition 6.3.5]{ACGSI}), 
with a similar expression for $(Q')^{\vee}$. 

Choose an
element $q\in \Int(Q^{\vee})$. Then necessarily the image $q_i$ of $q$ 
in $Q_i^{\vee}$ lies in $\Int(Q_i^{\vee})$ by condition (2). Further,
$q_1$ and $q_2$ have the same image in $P^{\vee}$.
Let $q_i'=(\bar g_i^{\flat})^t(q_i)$. Necessarily $q_i'\in \Int((Q_i')^{\vee})$
as $\bar g_i^{\flat}:Q_i'\rightarrow Q_i$ is a local homomorphism.
Further, $q_1'$ and $q_2'$ have the same image in $(P')^{\vee}$ 
because of commutativity of \eqref{eq:W1W2X}.
This gives an element $q':=
(q_1',q_2')$ of $(Q')^{\vee}=(Q_1')^{\vee}\times_{(P')^{\vee}} (Q_2')^{\vee}$.
In particular $q'$ lies in the interior of $(Q')^{\vee}$. 

Hence we obtain a commutative diagram
\[
\xymatrix@C=30pt
{
Q\ar[r]^q& \NN\\
Q'\ar[ru]_{q'}\ar[u]&
}
\]
with $q,q'$ local homomorphisms. Note that $Q'$ is the stalk of the 
ghost sheaf
at any point of $W_1'\times^{\fs}_{X'} W_2'$. By condition (1) this latter log
scheme is non-empty, and so there is a morphism
$\Spec\kappa^{\dagger}\rightarrow W_1'
\times^{\fs}_{X'} W_2'$ which induces the map $q':Q'\rightarrow \NN$
on stalks of ghost sheaves. Indeed,
it is sufficient to construct a morphism $\Spec\kappa^{\dagger}
\rightarrow \Spec(Q'\rightarrow\kappa)$, which is equivalent to giving
a local homomorphism $Q'\rightarrow\NN$ and a homomorphism
$Q'\rightarrow \kappa^{\times}$. We take the local homomorphism $Q'\rightarrow
\NN$ to be given by $q'$.

The chosen morphism can also be viewed as arising in a commutative diagram
\begin{equation}
\label{eq:standard point diagram}
\xymatrix@C=30pt
{
\Spec\kappa^{\dagger}\ar[r]\ar[d]&W'_2\ar[d]\\
W'_1\ar[r]&X'
}
\end{equation}
At the level of ghost sheaves, this diagram is given by
\[
\xymatrix@C=30pt
{
\NN& Q_2'\ar[l]_{q_2'}\\
Q_1'\ar[u]^{q_1'}&P'\ar[u]_{\theta_2'}\ar[l]^{\theta_1'}
}
\]
and the morphisms $\Spec\kappa^{\dagger}\rightarrow W'_i$ are then
determined by additional data of maps $\psi_i':Q_i'\rightarrow\kappa^{\times}$.
Commutativity of
\eqref{eq:standard point diagram} then comes down to the equality
\begin{equation}
\label{eq:psi12 commutativity}
(\psi_1'\circ \theta_1')\cdot \chi_{f_1'}
=
(\psi_2'\circ \theta_2')\cdot \chi_{f_2'}
\end{equation}
in $\Hom(P',\kappa^{\times})$.

We now wish to construct an analogous commutative diagram
\begin{equation}
\label{eq:standard point diagram2}
\xymatrix@C=30pt
{
\Spec\kappa^{\dagger}\ar[r]\ar[d]&W_2\ar[d]\\
W_1\ar[r]&X
}
\end{equation}
given at the level of ghost sheaves by the commutative diagram
\[
\xymatrix@C=30pt
{
\NN& Q_2\ar[l]_{q_2}\\
Q_1\ar[u]^{q_1}&P\ar[u]_{\theta_2}\ar[l]^{\theta_1}
}
\]
As all homomorphisms in this diagram are local, it is enough
to construct analogously $\psi_i:Q_i\rightarrow\kappa^{\times}$
such that
\begin{equation}
\label{eq:psi12 nonprimed commutativity}
(\psi_1\circ \theta_1)\cdot \chi_{f_1}
=
(\psi_2\circ \theta_2)\cdot \chi_{f_2}.
\end{equation}

By commutativity of \eqref{eq:W1W2X}, we have
\begin{equation}
\label{eq:chi commutativity}
\chi_g\cdot (\chi_{f_i}\circ \bar g^{\flat})=
\chi_{f_i'}\cdot (\chi_{g_i}\circ \theta_i')
\end{equation}
in $\Hom(P',\kappa^{\times})$. 

Using the assumed injectivity of $\bar g_i^{\flat}$ of condition (3),
we choose a lift $\psi_i:Q_i^{\gp}\rightarrow \kappa^{\times}$ of
$\chi_{g_i}^{-1} \cdot \psi_i':\bar g_i^{\flat}
((Q_i')^{\gp})
\rightarrow \kappa^{\times}$.
As $\kappa$ is algebraically closed, this can always be done even if
$\bar g_i^{\flat}((Q_i')^{\gp})$ is not saturated in $Q_i^{\gp}$.
Then for $p\in P' $, we have, with the third line by the
definition of $\psi_i$ and \eqref{eq:chi commutativity},
\begin{align*}
\big((\psi_i\circ \theta_i)\cdot\chi_{f_i}\big)(\bar g^{\flat}(p))
= {} &
\psi_i(\theta_i(\bar g^{\flat}(p))) \cdot \chi_{f_i}(\bar g^{\flat}(p))\\
= {} & 
\psi_i(\bar g_i^{\flat}(\theta_i'(p))) \cdot 
\chi_{f_i}(\bar g^{\flat}(p))\\
= {} & \left[ \psi_i'(\theta_i'(p))\cdot \chi_{g_i}(\theta_i'(p))^{-1}\right]
\cdot \left[\chi_g(p)^{-1}\cdot \chi_{f_i'}(p)\cdot \chi_{g_i}(\theta_i'(p))
\right]\\
= {} & \left[\psi_i'(\theta'_i(p)) \cdot\chi_{f_i'}(p)\right] \cdot
\chi_g(p)^{-1}. 
\end{align*}
By \eqref{eq:psi12 commutativity}, this is independent of $i$.

Now consider 
\[
\left[(\psi_1\circ \theta_1)\cdot\chi_{f_1}\right]
\cdot \left[(\psi_2\circ\theta_2)\cdot \chi_{f_2}\right]^{-1}
\in \Hom(P^{\gp},\kappa^\times),
\]
and note this homomorphism is the identity on $\bar g^{\flat}((P')^{\gp})$
by the previous paragraph. Thus it induces an element of
$\Hom(P^{\gp}/\bar g^{\flat}((P')^{\gp}),\kappa^{\times})$. 
As $\kappa$ is algebraically closed, $\kappa^{\times}$ is a divisible group
and hence by the injectivity of \eqref{eq:theta induced}, we obtain
a surjective map
\[
\Hom(Q_1^{\gp}/\bar g_1^{\flat}((Q_1')^{\gp}),\kappa^{\times})
\times
\Hom(Q_2^{\gp}/\bar g_2^{\flat}((Q_2')^{\gp}), \kappa^{\times})
\rightarrow
\Hom(P^{\gp}/\bar g^{\flat}((P')^{\gp}),\kappa^{\times}).
\]
Thus we may find $\varphi_i\in \Hom(Q_i^{\gp}/\bar g_i^{\flat}((Q_i')^{\gp}),
\kappa^{\times})$ such that if we replace $\psi_i$ with $\varphi_i
\cdot\psi_i$, \eqref{eq:psi12 nonprimed commutativity} holds.
\end{proof}

\section{Gluing one curve}
\label{sec:gluing one curve}
We give a first application of the material of the previous subsection.
We consider our standard gluing situation as in Notation 
\ref{not:standard gluing}.

Assume given basic punctured maps $f_i:C_i^\circ/W_i
\rightarrow X$ of type $\btau_i$ for $1\le i \le r$, with the
$W_i=\Spec(Q_i\rightarrow\kappa)$ being logarithmic points, $Q_i$
the basic monoid associated to $f_i$, and $\kappa$ algebraically closed. 
Suppose further that whenever
$E\in \mathbf{E}$ with vertices $v_1,v_2$, with corresponding
punctured points $p_{E,v_i} \in C_{i(v_i)}$,
we have
\begin{equation}
\label{eq:matching}
f_{i(v_1)}(p_{E,v_1})=f_{i(v_2)}(p_{E,v_2}),
\end{equation}
i.e., the maps $f_i$ will glue schematically. We may then ask how many
gluings exist at the logarithmic level. More precisely, we would like
to understand the scheme $W$ defined as follows:

\begin{definition} 
The \emph{gluing} $f:C^{\circ}/W\rightarrow X$ of  the punctured maps $f_i$
is defined by
\[
W := \scrM'(X,\btau)\times_{\prod_{i=1}^r \scrM'(X,\btau_i)} \prod_{i=1}^r
W_i,
\]
and $f:C^{\circ}/W\rightarrow X$ the pull-back of the universal map
over $\scrM'(X,\btau)$ to $W$.
\end{definition}

In this situation, we introduce the following notation. 
For any punctured point $p_{E,v}$ of 
$C^{\circ}_i$ indexed by a flag $v\in E\in \mathbf{E}$, let $P_{E,v}$ be the 
stalk of 
$\overline\shM_X$ at $f_{i(v)}(p_{E,v})$. 
By \eqref{eq:matching}, we have $P_{E,v_1}=P_{E,v_2}$ if $v_1,v_2$ are the
vertices of $E$, and write both as $P_E$. For any irreducible component
of $C_i$ with generic point $\eta$ corresponding to a vertex $v$ of $G_i$, write
$P_v$ for the stalk of the ghost sheaf at $f_i(\eta)$.

For each $i$, we have a family of tropical maps 
$h_i:\Gamma(G_i,\ell_i)\rightarrow
\Sigma(X)$ defined over $\tau_i$. If $\omega_v\in \Gamma(G_i,\ell_i)$
is the cone corresponding to a vertex $v\in V(G_i)$, then
we obtain by restriction a map
\begin{equation}
\label{eq:eval def}
\ev_v:\omega_{v}\rightarrow \Sigma(X)
\end{equation}
mapping into the cone $P^{\vee}_{v,\RR}\in \Sigma(X)$.
Explicitly this is defined as the transpose of 
\[
\xymatrix@C=15pt
{
P_v \ar[r]^{\bar f_i^{\flat}}\ar[r] & Q_i,
}
\]
where here $Q_i$ is identified with $\overline\shM_{C_i,\bar\eta}$.
Hence at the level of groups we also obtain a map
\[
\ev_{v}:Q_i^*\rightarrow P_v^*.
\]

\begin{definition}
\label{def:tropical gluing map}
With the notation as above, choose an orientation on each edge $E\in E(G)$
so that $E$ has vertices $v_E, v_E'$ and is oriented from $v_E$ to $v'_E$. 
We define the \emph{tropical gluing map}
\[
\Psi:\prod_{i=1}^r
Q_i^*\times \prod_{E\in \mathbf{E}}\ZZ\rightarrow \prod_{E\in \mathbf{E}} 
P_E^*
\]
by
\[
\Psi\big((q_1,\ldots,q_r), (\ell_E)_{E\in \mathbf{E}}\big)
=\big(\ev_{v_E}(q_{i(v_E)})+\ell_E \mathbf{u}(E) - 
\ev_{v'_E}(q_{i(v_E')})\big)_{E\in \mathbf{E}}.
\]
We define the \emph{tropical multiplicity} of the gluing situation
to be 
\[
\mu=\mu(\tau,\mathbf{E}):= |(\coker\Psi)_{\tors}|.
\]
\end{definition}

\begin{remark}
\label{rem:tropical gluing map}

The map $\Psi$ is called the tropical gluing map for the following reason.
Suppose given $s=\big((q_i), (\ell_E)_{E\in E(G)}\big)
\in \ker\Psi$ such that $q_i\in Q_i^{\vee}$ for each $v$ and
$\ell_E>0$ for each $E$. 
Then we may construct a tropical map
$h_s:G\rightarrow\Sigma(X)$ as follows. First, for each $i$,
let $G'_i$ be the subgraph of $G_i$ obtained by removing legs of
the form $(E,v)$ for $E\in \mathbf{E}$. Then $G_i'$ is naturally
identified with a subgraph of $G$, and we may define $h_s|_{G_i'}$
to agree with $(h_i)_{q_i}|_{G_i'}$. On the other hand, if we give
each $E\in \mathbf{E}$ the length $\ell_E$, then $s\in \ker\Psi$
guarantees we can extend $h_s$ across all edges $E\in \mathbf{E}$.

Thus it is reasonable to think of $\ker\Psi$ as the integral tangent space
to the family of glued tropical curves.
\end{remark}

\begin{theorem}
\label{thm:gluing count}
Suppose we are in the above situation.
If the gluing $W$ is non-empty, then $W$ has 
$\mu(\tau,\mathbf{E})$ connected components.
\end{theorem}

\begin{proof}
By Theorem \ref{Thm: Gluing theorem}, we have an fs Cartesian diagram
\begin{equation}
\label{eq:point diagram}
\xymatrix@C=30pt
{
\widetilde W\ar[r] \ar[d] & \prod_{i=1}^r\widetilde W_i
\ar[d]^{\ev_{\mathbf{L}}}\\
\prod_{E\in \mathbf{E}} X \ar[r]_{\Delta}&\prod_{E\in\mathbf{E}} X\times
X
}
\end{equation}
Here 
\[
\widetilde W_i:=W_i\times_{\scrM'(X,\btau_i)}\widetilde\scrM'(X,\btau_i)
\]
and
\[
\widetilde W:=W\times_{\scrM'(X,\btau)}\widetilde\scrM'(X,\btau).
\]
A priori all fibre products are over $B$, but by the assumption
\eqref{eq:matching}, the composition of $f_i$ with the structure
map $X\rightarrow B$ are all constant with the same image, so we
may replace $B$ by a suitable affine neighbourhood of this image and 
apply Proposition \ref{prop:relative versus absolute}. Thus we may
replace $B$ with $\Spec\kk$ in the above discussion and thus assume
that all products in \eqref{eq:point diagram} are defined over
$\Spec\kk$.

By Proposition \ref{prop:reduced stacks},
the underlying schemes of $\widetilde W$ and $W$ agree.
Thus we need to calculate the number of connected components
of $\widetilde W$. Further, by \eqref{eq:matching}, for any edge
$E\in \mathbf{E}$ with endpoints $v_1, v_2$, the evaluation maps
$\widetilde W_{i(v)}\rightarrow X$, $\widetilde W_{i(v')}\rightarrow X$
both factor through the strict closed point $f_{i(v_1)}(p_{E,v_1})
=f_{i(v_2)}(p_{E,v_2})$. Thus we may replace, for each edge $E\in\mathbf{E}$,
the target $X$ with the corresponding log point, and hence obtain an
fs Cartesian diagram
\begin{equation}
\label{eq:point gluing product}
\xymatrix@C=30pt
{
\widetilde W\ar[r] \ar[d] & \prod_{i=1}^r\widetilde W_i
\ar[d]^{\ev_{\mathbf{L}}}\\
\prod_{E\in \mathbf{E}} \Spec(P_E\rightarrow \kk) 
\ar[r]_{\Delta}&\prod_{E\in\mathbf{E}} \Spec(P_E\rightarrow\kk)^2
}
\end{equation}
Further, by Lemma \ref{lem:reduced product}, we may replace $\widetilde W_i$
by its reduction without changing the number of connected components of
$\widetilde W$. Again by Proposition~\ref{prop:reduced stacks},  
the reduction of the underlying scheme of 
$\widetilde W_i$ agrees with the underlying scheme of $W_i$ (being a
point). Thus now $\widetilde W$ is a fibre product of log points.

Note that for $L_j=(E,v)\in \mathbf{L}_i$, 
$W_i\times_{\foM'(X,\btau_i)}\foM'_{L_j}(X,\btau_i)$
has ghost sheaf $Q_{i}^{L_j}\subseteq Q_i\oplus\ZZ$, with an equality on the
level of groups, and with the induced evaluation map $f_i\circ p_{E,v}:W_i
\rightarrow X$ yielding a map at the tropical level
of
\[
((f_i\circ p_{E,v})^{\flat})^t: (Q_i^{L_j})^{\vee}\rightarrow P_E^{\vee}
\]
taking the value on $(s,\ell)\in (Q_i^{L_j})^{\vee}\subseteq Q_i^{\vee}\oplus\ZZ$
given by
\[
\ev_v(s)+\ell \mathbf{u}(L_i).
\]
Further, from the fibre product description 
\eqref{Eqn: tilde log structure} we see that
\[
\overline\shM_{\widetilde W_i}\subseteq Q_i\oplus\bigoplus_{(E,v)\in L(G_i)}
\ZZ,
\] 
again with an equality on groups.

Of course $\Delta$ induces a map $P_E^*\rightarrow P_E^*\times
P_E^*$ given by the diagonal. Putting this together, the homomorphism
$\theta^t$ of Lemma \ref{lem:point product} then takes the form
\begin{equation}
\label{eq:theta t}
\theta^t:\prod_E P_E^* \times
\prod_{i=1}^r Q_i^* \times \prod_{v\in E\in \mathbf{E}} \ZZ
\rightarrow \prod_{v\in E\in \mathbf{E}} P_E^*
\end{equation}
given by
\[
\theta^t\big((n_E)_{E\in E(G)}, (s_i)_{1\le i\le r}, (\ell_{E,v})_{v\in E\in 
\mathbf{E}}
\big)=
(\ev_v(s_{i(v)})+\ell_{E,v}\mathbf{u}(E,v)-n_E)_{v\in E\in \mathbf{E}}.
\]
If $\widetilde W$ is non-empty,
then by Lemma \ref{lem:point product} the number of its 
connected components is the order of the torsion part of $\coker\theta^t$. 
We next compare this with the order of the torsion part of $\coker\Psi$.

We have a diagram 
\begin{equation}
\label{eq:big N diagram}
\xymatrix@C=30pt
{
\prod_{E\in \mathbf{E}}\ZZ\times \prod_{E\in\mathbf{E}} P_E^*\ar[r]^>>>>{\alpha}
\ar[d]_{\pi}&
\prod_{E\in\mathbf{E}} P_E^*\times\prod_i Q_i^*\times \prod_{v\in E\in
\mathbf{E}}\ZZ
\ar[d]_{\theta^t}\ar[r]^>>>>{\gamma}&\prod_i Q_i^*\times\prod_{E\in
\mathbf{E}}\ZZ
\ar[d]_{\Psi}\\
\prod_{E\in\mathbf{E}} P_E^*\ar[r]^{\beta}&
\prod_{v\in E\in\mathbf{E}} P_E^*\ar[d]\ar[r]^{\delta}&
\prod_{E\in\mathbf{E}} P_E^*\ar[d]\\
&\coker\theta^t\ar[r]_{\cong}& \coker\Psi
}
\end{equation}
Here $\pi$ is the surjective map given by
\[
\pi\big((\ell_E)_{E\in \mathbf{E}},(n_E)_{E\in \mathbf{E}}\big)
=\big(\ell_E \mathbf{u}(E,v_E)-n_E)_{E\in \mathbf{E}},
\]
where we use the chosen vertices $v_E,v'_E$ of Definition 
\ref{def:tropical gluing map}. Further, $\alpha,\beta$ are injections defined by
\begin{align*}
\alpha\big((\ell_E)_{E\in \mathbf{E}}, (n_E)_{E\in \mathbf{E}}\big)
= {} & \big((n_E)_{E\in \mathbf{E}},(0)_{1\le i\le r}, ((-1)^{\delta_{v,v'_E}}
\ell_E)_{v\in E\in \mathbf{E}}
\big),\\
\beta\big((n_E)_{E\in \mathbf{E}}\big)= {} & \big((n_E)_{v\in E
\in \mathbf{E}}\big).
\end{align*}
while $\gamma$ and $\delta$ are surjections defined by
\begin{align*}
\gamma\big((n_E)_{E\in \mathbf{E}}, (s_i)_{1\le i\le r}, 
(\ell_{E,v})_{v\in E\in 
\mathbf{E}} \big)= {} & 
\big((s_i)_{1\le i\le r},(\ell_{E,v_E}+\ell_{E,v'_E})_{E\in \mathbf{E}}\big),\\
\delta\big( (n_{E,v})_{v\in E\in \mathbf{E}}\big) = {} &
\big(n_{E,v_E}-n_{E,v'_E}\big)_{E\in \mathbf{E}}.
\end{align*}
Finally, $\Psi$ is the tropical gluing map of Definition 
\ref{def:tropical gluing map}.
One checks that the diagram is commutative with the top two rows
exact, and hence the
snake lemma implies that $\coker\theta^t\cong\coker\Psi$, giving
the result.
\end{proof}

\begin{example}
It is very important to note that the gluing parameter space $W$
need not be reduced, something which is quite different from
gluing ordinary stable maps. This arises via saturation in the
fs fibre product. For example, consider a target space $X$
with $\Sigma(X)=\RR_{\ge 0}^2$, and a gluing situation where
$\tau$ is a graph with two vertices, $v_1$ and $v_2$, and two edges,
$E_1$ and $E_2$, each connecting $v_1$ with $v_2$. We split at the
two edges, getting types $\tau_1,\tau_2$. We assume $\mathbf{u}(E_1,v_1)
=-\mathbf{u}(E_1,v_2) = (-w_1,w_1)$ and $\mathbf{u}(E_2,v_1)=
-\mathbf{u}(E_2,v_2) = (-w_2,w_2)$ with $\gcd(w_1,w_2)=1$ and $w_1 < w_2$.
Finally, we assume $\bsigma(v_1)=\RR_{\ge 0} \times 0$ and
$\bsigma(v_2)= 0 \times \RR_{\ge 0}$. 

A slightly tedious calculation shows that $\ul{W}=\Spec \kk[t]/(t^{w_1})$
in this case. Morally, one can think of this as follows. If the glued
curve smooths, there are smoothing parameters $u_1, u_2$ for the two nodes,
in the sense that the local structure at each node is of the form $xy=u_1$
or $xy=u_2$. The relation $u_1^{w_2}=u_2^{w_1}$ is then forced by the
logarithmic geometry of the situation. Saturation normalizes this curve, 
and the inverse image of the point with ideal $(u_1,u_2)$ under this 
normalization is the non-reduced gluing.
\end{example}

\begin{example}
As soon as the target space $X$ has strata of codimension at least three,
it is easy to find examples where the gluing does not exist. Here is
a standard local example. Let $X$ be the three-fold ordinary double
point  $\Spec\kk[x,y,z,w]/(xy-zw)$,
with toric log structure. Let $\ul{C}_1=V(x,z,y-\lambda_1w)\subseteq X$,
$\ul{C}_2=V(y,w,\lambda_2 x-z)$, for $\lambda_1,\lambda_2\in\kk^{\times}$.
In this case, $\Sigma(X)$ is the
cone in $\RR^3$ generated by $v_{x,z}=(0,0,1)$,
$v_{x,w}=(1,0,1)$, $v_{y,z}=(0,1,1)$ and $v_{y,w}=(1,1,1)$. Here each
generator generates a ray corresponding to a toric divisor of $X$, and
the subscripts indicate which coordinates are zero on that divisor.
It is possible to give the domain curves $\ul{C}_i$ the structure of
punctured log curves (ignoring the fact these curves are not proper)
so that the inclusions $\ul{C}_i\rightarrow \ul{X}$
become punctured log maps, with punctures $p_i\in C_i$ mapping
to $0\in X$ and contact orders $u_1=(1,1,0)$ and $u_2=(-1,-1,0)$.
The basic monoids $Q_1$, $Q_2$ are both $\NN$. One checks easily
that this gluing situation is not tropically transverse.

In this case, a careful analysis of the fibre product shows the fibre
product is only non-empty if $\lambda_1=\lambda_2$, i.e., $\ul{C}_1$
and $\ul{C}_2$ have the same slope. Morally, this can be seen
by blowing up the divisor $x=z=0$ in $X$ to obtain 
a small resolution $\widetilde X\rightarrow X$. The strict transforms
of $\ul{C}_1$ and $\ul{C}_2$ are then disjoint unless $\lambda_1=\lambda_2$.
Before the blow-up, the slopes are encoded in the evaluation maps
$\widetilde W_i\rightarrow X$; after the blow-up, the slopes are
already visible in the schematic evaluation maps.
\end{example}

\begin{remark}
\label{rem:ev space gluing}
Suppose instead we are given a gluing situation of maps to $\cX$ 
equipped with evaluation maps at the punctured points, i.e., 
the maps are determined by morphisms $W_i\rightarrow 
\foM'^{\ev(\mathbf{L}_i)}(\shX,
\btau_i)$ rather than $W_i\rightarrow \scrM'(X,\btau_i)$, still
with $\ul{W}_i=\Spec\kappa$.
In other words, we
are given pre-stable punctured maps $f_v:C_v^{\circ}/W_v\rightarrow \cX$ 
along with
compatible morphisms $\ul{W}_i\rightarrow \ul{X}$ for each punctured
point $p_{E,v}$. We assume further, as before, that the images of
$p_{E,v_1}$, $p_{E,v_2}$ under these maps to $\ul{X}$ agree
for each edge $E$. Then
similarly, the number of connected components of the glued space $W$
is $\mu(\tau,\mathbf{E})$. Indeed, the gluing is controlled by
precisely the same Cartesian diagram as in the situation of
Theorem \ref{thm:gluing count}.
\end{remark}

Note that Theorem \ref{thm:gluing count} says nothing about whether $W$ is 
non-empty. Here is one often useful criterion:

\begin{definition}
We say a gluing situation is \emph{tropically transverse} if the map
$\Psi$ of Definition \ref{def:tropical gluing map} has finite cokernel.
\end{definition}

\begin{theorem}
\label{thm:tropically transverse case}
In the situation of Theorem \ref{thm:gluing count}, $W$ is non-empty if the
gluing situation is tropically transverse.
\end{theorem}

\begin{proof}
We apply the non-emptiness criterion of Lemma \ref{lem:point product nonempty}
to the product of \eqref{eq:point gluing product}. In that lemma,
we take $W_i'=X'=\Spec\kk$ so that the first condition of the lemma
is trivially satisfied, and the third condition is equivalent to the
injectivity of the map $\theta$ whose transpose $\theta^t$ is given in
\eqref{eq:theta t}. However, $\theta^t$ having a finite cokernel
implies $\theta$ is injective, and \eqref{eq:big N diagram}
implies $\coker \theta^t\cong \coker\Psi$, and hence tropical 
transversality implies the third condition. Finally, $\tau$ realizable 
implies the second condition.
\end{proof}

\section{Gluing moduli spaces}
\label{sec:degeneration gluing}

\subsection{The general situation}
We continue with a standard gluing situation for $X/B$, as in
Notation \ref{not:standard gluing}. 

\begin{theorem}
\label{thm:gluing factorization}
There is a diagram
\[
\xymatrix@C=30pt
{
\scrM(X/B,\btau)\ar[d]_{\varepsilon^{\ev}}\ar[r]^{\phi'}&
\scrM^{\sch}(X/B,\btau)\ar[r]\ar[d]_{\varepsilon^{\sch}}&
\prod_{i=1}^r \scrM(X/B,\btau_i)\ar[d]^{\hat\varepsilon}\\
\fM^{\ev(\mathbf{L})}(\cX/B,\tau)\ar[r]_{\phi}&
\fM^{\sch,\ev(\mathbf{L})}(\cX/B,\tau)\ar[r]_{\Delta'}\ar[d]_{\ul{\ev}'}&
\prod_{i=1}^r \fM^{\ev(\mathbf{L}_i)}(\cX/B,\tau_i)\ar[d]^{\ul{\ev}}\\
& \prod_{E\in \mathbf{E}} \ul{X}_{\bsigma(E)}\ar[r]_{\Delta}&
\prod_{v\in E\in \mathbf{E}} \ul{X}_{\bsigma(E)}
}
\]
with all squares Cartesian in all categories, defining
the moduli spaces $\scrM^{\sch}(X/B,\btau)$
and $\fM^{\sch,\ev(\mathbf{L})}(\cX/B,\tau)$.
Further, $\phi$ is a finite representable morphism.
\end{theorem}

\begin{proof}
We define the morphism $\prod_{i=1}^r\fM^{\ev(\mathbf{L}_i)}(\cX/B,\tau_i)\rightarrow
\prod_{v\in E\in \mathbf{E}} \ul{X}_{\bsigma(E)}$ as follows.
For flag $v\in E\in \mathbf{E}$, consider the composition
\[
\prod_{i=1}^r\fM^{\ev(\mathbf{L}_i)}(\cX/B,\tau_i) \rightarrow 
\fM^{\ev(\mathbf{L}_{i(v)})}(\cX/B,\tau_{i(v)})=\fM(\cX/B,\tau_i)\times_{\prod\ul{\cX}}
\prod\ul{X}
\rightarrow \ul{X}
\]
where the first arrow is projection and the second is further
projection onto the factor $\ul{X}$ indexed by the leg $(E,v)\in\mathbf{L}_i$.
Necessarily, this morphism factors through $\ul{X}_{\bsigma(E)}$ by the
definition of a $\tau_i$-marked curve, see \cite[Def.~3.4,(1)]{ACGSII}.
This gives a morphism
\[
\ul{\ev}_{(E,v)}:
\prod_{i=1}^r\fM^{\ev(\mathbf{L}_i)}
(\cX/B,\tau_i) \rightarrow \ul{X}_{\bsigma(E)},
\]
and we define
\[
\ul{\ev}:=\prod_{v\in E\in \mathbf{E}} \ul{\ev}_{(E,v)}.
\]

The morphism $\Delta$ is the product of diagonals. It is then clear 
that the Cartesian 
diagram of Theorem \ref{Prop: Gluing via evaluation spaces}
factors as stated, as $\fM^{\sch,\ev(\mathbf{L})}(\cX/B,\tau)$
captures those curves which glue schematically. Further, $\phi$
is a finite and representable morphism as $\psi=\Delta'\circ\phi$ is
finite and representable by Theorem \ref{Prop: Gluing via evaluation spaces}.
\end{proof}

\begin{remark}
\label{rem:this is the best in general}
As stated, this is not a significant improvement over
Theorem \ref{Prop: Gluing via evaluation spaces}: it merely separates
the gluing into two steps, the first step being schematic gluing and the
second step taking into account only the gluing at the logarithmic level.

For the first step, we are in luck if (1) $\Delta$ is lci, so that 
the Gysin pull-back $\Delta^!$ exists, and (2) $\ul{\ev}$ is flat, so that 
$\Delta^!$ and $(\Delta')^!$ induce the same map
\[
\Delta^!=(\Delta')^!:A_*\left(\prod_i \scrM(X/B,\btau_i)\right)\rightarrow 
A_*\left(\scrM^{\sch}(X/B,\btau)\right).
\]
In any event, each $\fM^{\ev(\mathbf{L}_i)}(\shX/B,\tau_i)$ is pure-dimensional
by \cite[Prop.~3.30]{ACGSII},
so each $\scrM(X/B,\btau_i)$ carries a virtual
fundamental class. Further, the pull-back of the relative obstruction
theory for $\hat\varepsilon$ (the product over $i$ of the relative obstruction
theories for $\scrM(X/B,\btau_i)\rightarrow \fM^{\ev(\mathbf{L}_i)}(\shX/B,\tau_i)$)
yields a relative
obstruction theory for $\varepsilon_{\sch}$. 
If $\ul{\ev}$ is flat, then $\fM^{\sch,\ev(\mathbf{L})}(\cX/B,\tau)$
is also pure-dimensional, yielding a virtual fundamental
class $[\scrM^{\sch}(X/B,\btau)]^{\virt}$,
and if $\Delta$ is lci, we have
\begin{equation}
\label{eq:Gysin pullback}
[\scrM^{\sch}(X/B,\btau)]^{\virt}
=
\Delta^!\left(\prod_i[\scrM(X/B,\btau_i)]^{\virt}\right).
\end{equation}

Note that $\Delta$ is lci if and only if the
strata $\ul{X}_{\bsigma(E)}$ are non-singular. 
However, a deepest stratum of $X$ is always non-singular, as is
a stratum of dimension one more than a deepest stratum of $X$.
If instead the log structure on $X$ arises from an snc divisor, all
strata are non-singular.

On the other hand, flatness of $\ul{\ev}$ is only automatic when the
strata $X_{\bsigma(E)}$ are always deepest strata. Otherwise,
more care needs to be taken. Again, when the gluing strata are non-singular,
there is a tropical characterization of flatness.

\begin{theorem}
\label{thm:flatness}
Let $\tau=(G,\mathbf{g},\bsigma,\mathbf{u})$ be a realizable type of 
punctured map over $B$, $\mathbf{L}\subseteq L(G)$ a subset of legs, and let
\[
\ul{\ev}:\foM(\shX/B,\tau)\rightarrow\prod_{L\in \mathbf{L}}
\ul{\shX}_{\bsigma(L)}
\]
be the schematic evaulation map at the punctured points indexed
by $\mathbf{L}$. Suppose
the strata $\ul{\shX}_{\bsigma(L)}$ for $L\in \mathbf{L}$ are non-singular.
Then $\ul{\ev}$ is flat if and only if, for every realizable type $\tau'
=(G',\mathbf{g}',\bsigma',\mathbf{u}')$
with a contraction morphism $\phi:\tau'\rightarrow\tau$, the inequality
\[
\dim\tau' \ge \dim \tau +\sum_{L\in\mathbf{L}}(\dim \bsigma'(L)
-\dim \bsigma(L))
\]
holds. 
\end{theorem}

\begin{proof}
Since $\foM^{\ev(\mathbf{L})}(\shX/B,\tau)$ is pure-dimensional by 
\cite[Prop.~3.30]{ACGSII},
if $\ul{\ev}$ is flat then all stack-theoretic fibres $\ul{\ev}^{-1}(\bar x)$ 
of $\ul{\ev}$ are 
pure-dimensional with dimension independent of the choice of
a geometric point $\bar x\rightarrow \prod_{L\in\mathbf{L}} \ul{\shX}_{\bsigma(L)}$. Conversely, it follows from \cite[Thm.~3.25]{ACGSII} that
$\foM(\shX/B,\tau)$ is idealized log smooth over $\Spec\kk$
(as $\mathbf{M}(G,\mathbf{g})\times B$ appearing in that theorem is
idealized log smooth over $\Spec\kk$). However, from the specific
description of the idealized structure on $\foM(\shX/B,\tau)$ in 
\cite[Prop.~3.24]{ACGSII}, it follows from \cite[Prop.~B.4]{ACGSII} that
$\foM(\shX/B,\tau)$ smooth locally is a stratum of a toric variety,
and hence in particular smooth locally is toric, hence Cohen-Macaulay.
Thus, by \cite[Thm.~23.1]{Ma89}, 
if the fibres of $\ul{\ev}$ are all pure-dimensional of dimension
$\dim \foM(\shX/B,\tau) - \dim \prod_{L\in\mathbf{L}} \ul{\shX}_{\bsigma(L)}$,
then $\ul{\ev}$ is flat. Thus the question reduces
to calculating the fibre dimension.

Note that because we assumed $\tau$ was realizable over $B$, 
$\foM(\shX/B,\tau)$
contains a dense open stratum consisting of punctured maps to
$\shX$ of type $\tau$. Further, we have by \cite[Prop.~3.30]{ACGSII} that
\[
\dim \foM(\shX/B,\tau)=3|\mathbf{g}|-3+|L(G)|-\dim\tau+\delta_B
\]
where
\[
\delta_B:=\begin{cases}
0 & B=\Spec\kk\\
1 & B=\Spec\kk^{\dagger}
\end{cases}
\]
and $|\mathbf{g}|=b_1(G)+
\sum_{v\in V(G)} \mathbf{g}(v)$.

Let $\bar x$ be a geometric point of the open stratum
of $\prod_{L\in\mathbf{L}} \ul{\shX}_{\bsigma(L)}$. Note
this open stratum is isomorphic, as an algebraic stack, to
$B\GG_m^N$ where 
\[
N=-\dim \prod_{L\in\mathbf{L}} \ul\shX_{\bsigma(L)}=
\sum_{L\in \mathbf{L}}(\dim\bsigma(L)-\delta_B),
\]
and the dimension of the fibre of $\ul{\ev}$ over $\bar x$ is then
\begin{equation}
\label{eq:general fibre dim}
\dim\foM(\shX/B,\tau)+N=
3|\mathbf{g}|-3+|L(G)|-\dim\tau+\delta_B + \sum_{L\in\bsigma(L)}(\dim\bsigma(L)
-\delta_B).
\end{equation}
Thus we need to show all fibres, when non-empty, have this dimension. So
now choose some other geometric point 
$\bar x=(\bar x_L)\rightarrow \prod_{L\in\mathbf{L}}\ul{\shX}_{\bsigma(L)}$, 
and suppose
$\bar x_L$ lies in the stratum of $\shX$ indexed by a cone $\sigma_L$. Let
$\tau'$ be a type of punctured map which appears in $\ul{\ev}^{-1}(\bar x)$,
so that there is a contraction map $\phi:\tau'\rightarrow\tau$. Note
that $\sigma_L=\bsigma'(L)$ for $L\in\mathbf{L}$.
Then the dimension of the stratum of the fibre over $x$ with
of punctured maps with this type is similarly calculated as
\begin{align}
\label{eq:special fibre dim}
\begin{split}
&3|\mathbf{g}'|-3+|L(G')|-\dim\tau'+ \delta_B + \sum_{L\in\bsigma(L)}
(\dim\sigma_L-\delta_B)\\
= {} & 3|\mathbf{g}| - 3 +| L(G)| - \dim\tau' + \delta_B +
\sum_{L\in \bsigma(L)}(\dim\bsigma'(L)-\delta_B).
\end{split}
\end{align}
Now since fibre dimension is upper semi-continuous, it is sufficient
to show that the quantity of \eqref{eq:general fibre dim} is
greater than or equal to the quantity of \eqref{eq:special fibre dim}.
But this is the inequality of the theorem.
\end{proof}

\begin{remark}
The criterion of Theorem \ref{thm:flatness} may seem imposing to check
as it in theory involves an arbitrary number of type $\tau'$. But
we may always replace $\foM(\shX/B,\tau)$ with an open subset
obtained by deleting those strata with type $\tau'$ (equipped with
a contraction $\tau'\rightarrow\tau$) such that there does not exist
a punctured map in $\scrM(X/B,\btau)$ of type $\tau''$ such that the
induced contraction map $\tau''\rightarrow \tau$ factors through
$\tau'\rightarrow\tau$. Thus for any application, it is sufficient
to restrict attention to tropical types which are contractions of
types which actually occur.
\end{remark}

The next step is to understand $\phi$. In general, we don't yet know how
to say much about it, save for the next theorem, which gives us the
degree of $\phi$ onto its image. When the image is a proper closed
substack, this tells us there are logarithmic obstructions to gluing,
and it is still not understood how to deal with these obstructions
in general. However, in the tropically transverse case,
$\phi$ is surjective.
\end{remark}

\begin{theorem}
\label{thm:general gluing degree}
The degree of $\phi$ onto its image is the tropical multiplicity 
$\mu(\tau,\mathbf{E})$ defined in 
Definition~\ref{def:tropical gluing map}. If the gluing situation
is tropically transverse and $\ul{\ev}$ is flat, then $\phi$ is dominant.
\end{theorem}

\begin{proof}
We note that as $\tau$ is realizable over $B$, 
$\foM^{\ev(\mathbf{L})}(\shX/B,\tau)$
has a non-empty dense open stratum $U$, whose geometric points
are precisely those geometric points corresponding to a punctured 
map whose tropicalization
is a family of tropical maps of type $\tau$. Similarly,
$\foM^{\ev(\mathbf{L}_i)}(\shX/B,\tau_i)$ has a non-empty dense stratum $U_i$,
with each geometric point in this stratum corresponding to a punctured
map whose tropicalization is a family of curves of type $\tau_i$.
Certainly then $\Delta'\circ\phi(U)\subseteq\prod_i U_i$.
Further, $(\phi\circ\Delta)^{-1}(\prod_i U_i) = U$ as the type
of punctured map obtained by gluing together punctured maps of type $\tau_i$
is necessarily $\tau$. It is thus sufficient to determine the
degree of $\phi|_U$ onto its image in $(\Delta')^{-1}(\prod_i U_i)$.
To this end, pick a strict geometric point $W'\rightarrow
\foM^{\sch,\ev(\mathbf{L}_i)}(\shX/B,\tau)$ in the image of $\phi|_U$. Because
$\Delta'$ is strict, we can decompose $W'=\prod_i W_i$ as a product
over $B$, with
$W_i\rightarrow \foM^{\ev(\mathbf{L}_i)}(\shX/B,\tau_i)$ strict geometric points.
Thus we may view this as a situation of \S\ref{sec:gluing one curve}. 
In particular, the corresponding gluing is
$W'\times_{\prod_i \foM'^{\ev(\mathbf{L}_i)}(\shX/B,\tau_i)} \foM'^{\ev(\mathbf{L})}(\shX/B,\tau)$.
By Lemma \ref{lem:reduced product} and \cite[Prop.~3.33]{ACGSII}, 
the reduction of this gluing agrees with the reduction of
\[
W'\times_{\prod_i\foM^{\ev(\mathbf{L}_i)}(\shX,\tau_i)}
\foM^{\ev(\mathbf{L})}(\shX,\tau)
\cong
W'\times_{\foM^{\sch,\ev(\mathbf{L})}(\shX/B,\tau)} \foM^{\ev(\mathbf{L})}(\shX/B,\tau).
\]
By Remark \ref{rem:ev space gluing}, this fibre product has
$\mu(\tau,\mathbf{E})$ connected components. 
Since this number is independent of the choice of geometric point in the image
of $\phi|_U$ and $\foM^{\ev}(\shX/B,\tau)$ is reduced, it follows
this fibre must be always reduced. This shows the degree.

If $\ul{\ev}$ is flat, then so is $\ul{\ev}'$, and hence every irreducible
component of $\foM^{\mathrm{sch},\ev(\mathbf{L})}(\shX/B,\tau)$
dominates $\prod_{E\in \mathbf{E}} \ul{X}_{\bsigma(E)}$. The domination
of $\phi$ then similarly follows from
Theorem \ref{thm:tropically transverse case}.
\end{proof}

\subsection{Gluing from rigid tropical curves}
\label{subsec:rigid}
One of the standard situations for gluing is the degeneration
situation studied in \cite{ACGSI}. Here, one considers
a base $B$ a smooth curve or spectrum of a DVR, with log structure
the divisorial log structure induced by a closed point $b_0\in B$.
As usual, we assume $X\rightarrow B$ is projective and log smooth.
We denote by $X_0$ the fibre over $b_0$, and we view $b_0$ with
its induced log structure, i.e., $b_0$ is a standard log point.

In this case, the morphism $X\rightarrow B$ tropicalizes to a morphism
$\delta:\Sigma(X)\rightarrow \Sigma(B)=\RR_{\ge 0}$. This gives rise to
a polyhedral cell complex $\Delta(X)=\delta^{-1}(1)$. Given a class 
$\beta$
of log curve for $X/B$, we have the logarithmic decomposition formula,
see \cite[Thm.~1.2]{ACGSI}:
\[
[\scrM(X_0/b_0,\beta)]^{\virt} = \sum_{\btau=(\tau,{\bf A})}
{m_{\tau}\over |\Aut(\btau)|} j_{\tau*}[\scrM(X_0/b_0,\btau)]^{\virt}.
\]
Here $\btau$ runs over isomorphism classes of
decorated \emph{rigid tropical types}.
These are realizable tropical types $\btau$ such that the moduli space
of tropical maps of type $\btau$ is one-dimensional, with a unique
member of this one-dimensional family factoring through $\Delta(X)$,
hence the term rigid. The quantity $m_{\tau}\in\NN$ is the
smallest integer such that scaling $\Delta(X)$ by $m_{\tau}$ leads to a 
tropical map with integral vertices and edge lengths; put another way,
\[
m_{\tau} = |\coker(N_{\tau}\rightarrow N_{\Sigma(B)}=\ZZ)|.
\]

We now explore how Theorem \ref{thm:gluing factorization} may be
applied in this situation. Unfortunately,
the morphisms $\phi,\phi'$ of Theorem \ref{thm:gluing factorization}
need not  be surjective, but here we
determine the length of the inverse image of a point 
in the image of $\phi$.

We first introduce the following notation. 
For each $\sigma\in\Sigma(X)$, the morphism $\delta:\Sigma(X)\rightarrow 
\Sigma(B)$
induces homomorphisms $\delta_*:N_{\sigma}\rightarrow \ZZ$, and we define
$\overline{N}_{\sigma}$ to be the kernel of this homomorphism. 
Provided $\delta|_\sigma$ surjects onto $\Sigma(B)=\RR_{\ge 0}$, this kernel can
be identified with the space of integral tangent vectors of the corresponding
polyhedron in $\Delta(X)$.

\begin{definition}
\label{def: tropically transverse}
Let $\tau$ be a type of rigid tropical curve in $\Delta(X)$. 
Define
\[
\overline\Psi:\bigoplus_{v\in V(G)} \overline{N}_{\bsigma(v)} \oplus 
\bigoplus_{E\in E(G)} \ZZ
\rightarrow \bigoplus_{E\in E(G)} \overline{N}_{\bsigma(E)}
\]
to be the homomorphism
\[
\big((n_v)_{v\in V(G)},(\ell_E)_{E\in E(G)}\big)
\mapsto
\big(n_{v_E}+\ell_E \mathbf{u}(E)-n_{v'_E}\big)_{E\in E(G)}
\]
where for each $E\in E(G)$, we orient $E$ from a vertex $v_E$
to a vertex $v_E'$ to determine the sign of $\mathbf{u}(E)$, which necessarily
lies in $\overline{N}_{\bsigma(E)}$.
We define the \emph{tropical multiplicity} of $\tau$ to be 
\[
\mu(\tau):=|(\coker\overline\Psi)_{\tors}|.
\]
\end{definition}

\begin{theorem}
\label{thm: rigid degree} 
Let $X\rightarrow B$ be as above.
Fix a rigid tropical type $\btau=(G,\bsigma,{\bf u})$,
and let $\{\btau_v\,|\,v \in V(G)\}$ be the decorated tropical types obtained
by splitting $\btau$ at all edges. This gives a standard gluing
situation, and hence a diagram as in Theorem \ref{thm:gluing factorization}.
With notation as in that theorem,
$\phi$ is degree $\mu(\tau)/m_{\tau}$ onto its image.
\end{theorem}

\begin{proof} 
By Theorem \ref{thm:general gluing degree},
it is sufficient to compare the multiplicity $|\coker(\Psi)_{\tors}|$
as defined in Definition \ref{def:tropical gluing map} and the multiplicity
$|\coker(\overline\Psi)_{\tors}|$.
Note that as each split type $\tau_v$ consists of a single vertex
with a number of adjacent edges, the only moduli of tropical
maps of type $\tau_v$ is given by the location of $v$ in $\bsigma(v)$.
Thus one sees that the basic monoid associated to the type $\tau_v$ is
\begin{equation}
\label{eq:QvPv}
Q_v=P_{\bsigma(v)}.
\end{equation}
Thus $Q_v^*=P_{\bsigma(v)}^*=N_{\bsigma(v)}$ and $P_E^*=N_{\bsigma(E)}$
in Definition \ref{def:tropical gluing map}.
The map $\Psi$ of that definition now becomes
\[
\Psi:\bigoplus_{v\in V(G)} N_{\bsigma(v)} \oplus 
\bigoplus_{E\in E(G)} \ZZ
\rightarrow \bigoplus_{E\in E(G)} N_{\bsigma(E)},
\]
and the degree of $\phi$ onto its image is $|(\coker\Psi)_{\tors}|$.

We now have a commutative diagram of exact sequences
\[
\xymatrix@C=30pt
{
0\ar[r]&\prod_v \overline{N}_{\bsigma(v)}\times\prod_E \ZZ
\ar[d]_{\overline\Psi}\ar[r]&\prod_v N_{\bsigma(v)}\times \prod_E\ZZ
\ar[d]_{\Psi}\ar[r]^>>>>>{\delta_*} & \prod_v \ZZ\ar[r]\ar[d]^{\partial}&0\\
0\ar[r]&\prod_E \overline{N}_{\bsigma(E)}\ar[r]&\prod_E N_{\bsigma(E)}
\ar[r]_>>>>>{\delta_*}&\prod_E\ZZ\ar[r]&0
}
\]
Here the two maps labelled $\delta_*$ are induced by $\delta_*:N_{\bsigma(v)}
\rightarrow\ZZ$ and $\delta_*:N_{\bsigma(E)}\rightarrow\ZZ$ for
each vertex $v$ and edge $E$. 
The map $\overline\Psi$ 
is as defined in Definition \ref{def: tropically transverse}.
The map $\partial$ is defined by
\[
\partial\big((n_v)_{v\in V(G)}\big)=\big(n_{v_E}-n_{v'_E}\big)_{E\in E(G)}.
\]
In particular, $\partial:\prod_v \ZZ\rightarrow \prod_E\ZZ$
is the complex calculating the simplicial cohomology of $G$, 
and thus $\ker\partial=H^0(G,\ZZ)=\ZZ$, as $G$ is assumed connected.
Note $\ker\partial$ is generated by $(1)_{v\in V(G)}$.
Also $\coker\partial = H^1(G,\ZZ)=\ZZ^{b_1(G)}$.
Thus the snake lemma gives a long exact sequence
\[
0\rightarrow \ker\overline\Psi\rightarrow \ker\Psi\rightarrow
H^0(G,\ZZ)\rightarrow \coker\overline\Psi\rightarrow\coker\Psi\rightarrow
H^1(G,\ZZ)\rightarrow 0.
\]

Note that $\ker\overline\Psi$ and $\ker\Psi$ can be interpreted as the space of
integral tangent vectors to the moduli space of maps of
type $\tau$ in $\Delta(X)$ and $\Sigma(X)$ respectively.
By the assumption of rigidity of $\tau$, we thus have
$\ker\overline\Psi=0$ and $\ker\Psi=\ZZ$. Further, the map
$\ZZ\cong\ker\Psi\rightarrow H^0(G,\ZZ)\cong \ZZ$ is multiplication
by $m_{\tau}$ by definition of the latter number. This gives an exact
sequence
\[
0\rightarrow \ZZ/m_{\tau}\ZZ\rightarrow\coker\overline\Psi \rightarrow
\coker\Psi\rightarrow H^1(G,\ZZ)\rightarrow 0.
\]
Since $H^1(G,\ZZ)$ is torsion-free and $\ZZ/m_{\tau}\ZZ$ is torsion, 
we easily obtain a short exact sequence
\[
0\rightarrow \ZZ/m_{\tau}\ZZ\rightarrow
(\coker\overline\Psi)_{\tors}\rightarrow
(\coker\Psi)_{\tors}\rightarrow 0.
\]
This shows that the multiplicity $\mu(\tau,\mathbf{E})$ defined in
Definition \ref{def:tropical gluing map} agrees with $\mu(\tau)/m_{\tau}$, and
the result follows.
\end{proof}

Again, we get better behaviour in the tropically transverse case,
where now we have a slightly weaker definition for tropical
transversality.

\begin{definition}
\label{def: tropically transverse2}
Let $\tau$ be the combinatorial type of a rigid tropical curve in
$\Delta(X)$. We say that $\tau$ is \emph{tropically transverse} if the image 
of the map $\overline\Psi$ of 
Definition \ref{def: tropically transverse} has finite index.
\end{definition}

\begin{theorem}
\label{thm:tropically transverse}
In the situation of Theorem \ref{thm: rigid degree}, suppose
that the rigid decorated type $\btau$ is 
tropically transverse, $\ul{\ev}$ is flat, and the central fibre
$\ul{X}_0$ is reduced. Then $\phi$ is a finite dominant morphism
of degree $\mu(\tau)/m_{\tau}$. Further, suppose
that each $\ul{X}_{\bsigma(E)}$
is non-singular (so that $\Delta$ is an lci morphism and the Gysin
map $\Delta^!$ exists). Then 
\[
\phi'_*[\scrM(X_0/b_0,\btau)]^{\virt} = {\mu(\btau)
\over m_{\btau}}
\Delta^!\left(\prod_{v\in V(G)} [\scrM(X_0/b_0,\btau_v)]^{\virt}\right).
\]
\end{theorem}

\begin{proof}
The statement concerning the degree of $\phi$
follows from Theorem \ref{thm: rigid degree} provided that we know
$\phi$ is dominant. To show this, we follow the argument of Theorem
\ref{thm:general gluing degree} and modify the argument of Theorem 
\ref{thm:gluing count}. In particular, we may assume given a strict
geometric point $\bar w:W'\rightarrow \foM^{\sch,\ev}(\cX_0/b_0,\tau)$.
Further, the log structure on $W'$ can be described as a product of log points
$W'=\prod_v W_v$ over $\ul{W}'$ 
with $W_v\rightarrow \fM^{\ev}(\cX_0/b_0,\tau_v)$ 
strict geometric points having image in the dense open strata
of the latter moduli spaces. We need to show that the gluing
$W'\times_{\fM^{\sch,\ev}(\cX_0/b_0,\tau)}\fM^{\ev}(\cX_0/b_0,\tau)$ 
is non-empty for general choice of $\bar w$.
To show this non-emptiness, we return to the setup and notation of the proof of 
Theorems \ref{thm:gluing count} and use the non-emptiness criterion
of Lemma \ref{lem:point product nonempty},
applied to the commutative diagram obtained from the fact that all spaces
involved are defined over $b_0$:\footnote{We note here we work with 
fibre products over $\Spec\kk$ rather than $B=b_0$, as we did in Theorem 
\ref{thm:gluing count}. Here we
apply Proposition \ref{prop:relative versus absolute} to see that we
may just as well work with the moduli spaces over $\Spec\kk$.}
\begin{equation}
\label{eq:Xb diagram}
\xymatrix@C=30pt
{
\prod_{E\in E(G_)}X_{\bsigma(E)}\ar[d]\ar[r]&
\prod_{v\in E\in E(G)} X_{\bsigma(E)}\ar[d]&
\prod_{v\in V(G)} \widetilde W_v\ar[l]\ar[d]\\
\prod_{E\in E(G)} b_0\ar[r]&\prod_{v\in E\in E(G)} b_0&
\prod_{v\in V(G)} b_0\ar[l]
}
\end{equation}
Now $\prod_E b_0\times_{\prod_{v\in E} b_0} \prod_{v} b_0$
is easily seen to be non-empty: indeed, there is a morphism
from $b_0$ to this fibre product induced by the diagonal morphisms
$b_0\rightarrow \prod_E b_0$ and $b_0\rightarrow \prod_v b_0$. (In fact,
the induced morphism is an isomorphism, so the fibre product is $b_0$,
but we don't need this.) This gives condition (1) of
the hypotheses of Lemma \ref{lem:point product nonempty}.

Condition (2) follows from realizability the tropical type
$\tau$. Indeed, at the level of dual of monoids, the fibre product
involving the top row of \eqref{eq:Xb diagram} gives a Cartesian
diagram of monoids
\[
\xymatrix@C=30pt
{
\widetilde Q^{\vee}\ar[r]\ar[d]& \prod_{v\in V(G)} \widetilde Q_v^{\vee}
\ar[d]_{\Sigma(\ev)}\\
\prod_{E\in E(G)} P_{\bsigma(E)}^{\vee}
\ar[r]_{\Sigma(\Delta)}&\prod_{v\in E\in E(G)} P_{\bsigma(E)}^{\vee}
}
\]
Here $\widetilde Q_v$ is the stalk of the ghost sheaf of $\widetilde W_v$,
with $\widetilde Q_v\subseteq Q_v\oplus\bigoplus_{v\in E\in E(G)} \ZZ$.
Recall \eqref{eq:QvPv} that $Q_v=P_{\bsigma(v)}$. We need to show that the image
of $\widetilde Q^{\vee}$ in $\prod_E P_{\bsigma(E)}^{\vee}$ or
$\prod_v \widetilde Q_v^{\vee}$ intersects the interior of that
monoid.

First we describe $\Sigma(\Delta)$ and $\Sigma(\ev)$. Indeed,
$\Sigma(\Delta)$ is just the diagonal, while
\[
\Sigma(\ev)\left((n_v)_{v\in V(G)},(\ell_{v,E})_{v\in E\in E(G)}\right)
=(n_v+\ell_{v,E}\mathbf{u}_v(E))_{v\in E\in E(G)}
\]
where $\mathbf{u}_{v}(E)$ is the contact order of the leg $(v,E)$ of
$G_v$. 
Thus the above fibre product diagram allows us to view $\widetilde Q^{\vee}$
as a submonoid of $\prod_{v\in v(G)} \widetilde Q_v^{\vee}$ consisting
of those tuples $\left((n_v)_{v\in V(G)},(\ell_{v,E})_{v\in E\in E(G)}\right)$
such that, for every edge $E\in E(G)$ with vertices $v_1,v_2$, we have
\begin{equation}
\label{eq:edge midpoint}
n_{v_1}+\ell_{v_1,E}\mathbf{u}_{v_1}(E)=n_{v_2}+\ell_{v_2,E}\mathbf{u}_{v_2}(E).
\end{equation}
We may then interpret a point $s\in \widetilde Q^{\vee}$ as giving the data of:
\begin{enumerate}
\item
The choice of a tropical map $h_s:G\rightarrow\Sigma(X)$ 
of type $\tau$. This tropical map is determined by 
$h_s(v)=n_v \in P^{\vee}_{\bsigma(v)}$,
and the condition \eqref{eq:edge midpoint} then guarantees that
the edge $E$ is still mapped to an edge in $\bsigma(E)$
with tangent vector $\mathbf{u}(E)$.
\item
For each edge $E\in E(G)$, a choice of decomposition of the length 
of $E$ as a sum $\ell_{v_1,E}+\ell_{v_2,E}$. This is given by
\eqref{eq:edge midpoint}: since $\mathbf{u}_{v_2}(E)=-\mathbf{u}_{v_1}(E)$,
\eqref{eq:edge midpoint} can be rewritten as $n_{v_2}-n_{v_1}=(\ell_{v_1,E}
+\ell_{v_2,E})\mathbf{u}_{v_1}(E)$, and thus $\ell_E=\ell_{v_1,E}
+\ell_{v_2,E}$.
\end{enumerate}
However, since $\tau$ is a type of rigid curve, there is a unique
tropical curve $h_s:\Gamma\rightarrow\Sigma(X)$ of type $\tau$, up to scaling.
By scaling this curve sufficiently so it is integrally defined and all
edge lengths are at least $2$, we may also choose non-trivial integral
decompositions of the edge lengths. This yields
an interior point of $\widetilde Q^{\vee}$, which necessarily maps
to the interior of $\prod_v \widetilde Q_v^{\vee}$ and to 
an interior point of $\prod_E P^{\vee}_{\bsigma(E)}$.
This shows condition (2) of Lemma \ref{lem:point product nonempty}.

Abusing notation, we denote by $\delta$ the generator of the ghost sheaf 
$\NN$ of $b_0$, and also denote by $\delta$ its image in any of the 
monoids $P_{\bsigma(E)}$ under the structure map $X_0\rightarrow b_0$.
Note this is largely compatible with our previous use of the notation
$\delta:\Sigma(X)\rightarrow\Sigma(B)$. To verify
condition (3) of the lemma, using the above description
of the monoids involved in the fibre product, we need to show injectivity
of
\[
\bigoplus_{v\in E\in E(G)} P_{\bsigma(E)}^{\gp}/\ZZ\delta
\rightarrow \left[\bigoplus_{E\in E(G)} P_{\bsigma(E)}^{\gp}/\ZZ\delta\right]
\times\left[\bigoplus_{v\in V(G)} P_{\bsigma(v)}^{\gp}/\ZZ\delta
\times\bigoplus_{v\in E\in E(G)}\ZZ\right]
\]
Because of the hypothesis that the central fibre $X_0$ is reduced,
all groups $P^{\gp}_{\bsigma(E)}/\ZZ\delta$ are torsion-free, and hence
injectivity is equivalent to the transpose map having finite cokernel.
However, we have a variant of diagram
\eqref{eq:big N diagram} given by replacing each $N_{\bsigma(E)},
N_{\bsigma(v)}$ with $\overline{N}_{\bsigma(E)}, \overline{N}_{\bsigma(v)}$,
which shows that the cokernel of the transpose of the above map
is finite if and only if the cokernel of $\overline\Psi$ is finite,
which is the tropically transverse condition.

For the last statement, the extra condition  implies $\Delta$ is lci, hence
$\Delta^!$ makes sense. By flatness of $\ul{\ev}$,
the Gysin pull-backs $\Delta^!$ and $(\Delta')^!$ agree. 
The result then follows from Theorem \ref{thm: rigid degree}
and properties of virtual pull-backs \cite[Thm.\ 4.1]{Mano}.
\end{proof}

\section{Punctured versus relative}
\label{sec:punctured versus relative}
In a gluing situation arising from a degeneration situation $X\rightarrow B$
as reviewed in \S\ref{subsec:rigid},
one has the moduli spaces $\scrM(X,\btau_v)$, classifying punctured
maps marked by $\btau_v$. In particular, such maps will factor through
a stratum $X_{\bsigma(v)}$ with its induced log structure. On the
other hand, $X_{\bsigma(v)}$ comes with a divisorial log structure
induced by $\partial X_{\bsigma(v)}\subseteq X_{\bsigma(v)}$, the
union of lower dimensional strata of $X$ contained in $X_{\bsigma(v)}$.
We write this different log scheme as $\overline X_{\bsigma(v)}$.
It is then natural to compare $\scrM(X,\btau_v)$ with a moduli space
of punctured log maps to $\overline X_{\bsigma(v)}$. In general, this
still a non-trivial question, and is related to double ramification
cycles with target, see \cite{JPPZ20},\cite{Lu}.  Here, however, we deal with
a special case, namely where $X_{\bsigma(v)}$ is an irreducible component of
$X_{b_0}$. We deal with somewhat more general types, however.

To set this up, let $\sigma\in\Sigma(X)$ be a non-zero cone, corresponding
to a stratum $X_{\sigma}$, and assume $X_{\sigma}\subseteq X_{b_0}$.
If $\dim\sigma=1$, then $X_{\sigma}$ is an irreducible 
component of $X_{b_0}$. As
above, we have the log scheme $\overline X_{\sigma}$, and there is a 
canonical morphism of log schemes
\[
\psi:X_{\sigma}\rightarrow \overline X_{\sigma}
\]
induced by the natural inclusion $\shM_{\overline X_{\sigma}}
\subseteq \shM_{X_{\sigma}}=\shM_X|_{X_{\sigma}}$. Indeed, smooth
locally at a geometric point $\bar x\in X_{\sigma}$, $X$ is given
as $\Spec \kk[P]$ for $P=\overline{\shM}_{X,\bar x}$, and the monoid
$P$ has a codimension $\dim\sigma$ face $F\subseteq P$ such that $X_{\sigma}$
is smooth locally $\Spec\kk[F]$ and $F=\overline\shM_{\overline X_{\sigma},\bar x}$. 
The log structure on $X_{\sigma}$ at $\bar x$ is given by a neat chart 
$P\rightarrow\cO_{X_{\sigma}}$, and the log structure 
on $\overline{X}_{\sigma}$
is given by restricting this chart to $F$. Hence we obtain the inclusion
of log structures.

Note the inclusion of faces $F\subseteq P$ dualizes to a generization
\[
P^{\vee}_{\RR} \rightarrow F^{\vee}_{\RR}=
(P^{\vee}_{\RR}+\RR\sigma)/\RR\sigma.
\]
This yields a description of the induced map of cone complexes
\[
\Sigma(\psi):\Sigma(X_{\sigma}) \rightarrow \Sigma(\overline X_{\sigma})
\]
at the level of cones of $\Sigma(X_{\sigma})$.
In particular, a type $\tau=(G,\mathbf{g},\bsigma,\bu)$ 
for tropical map to $\Sigma(X_{\sigma})$
induces a type $\bar\tau=(\bar G,\bar{\mathbf{g}},\bar\bsigma,\bar\bu)$ 
for tropical map to $\Sigma(\overline{X}_{\sigma})$.
Indeed, $\bar G$ will coincide with $G$, except some bounded
legs of $G$ may become unbounded. Of course
$\bar{\mathbf{g}}=\mathbf{g}$. For $v\in V(G)$,
we may take $\bar\bsigma(v)$ to be the minimal cone of 
$\Sigma(\overline{X}_{\sigma})$ containing $\Sigma(\psi)(\bsigma(v))$.
For an edge $E$ of $G$, we define $\bar\bu(E)= \Sigma(\psi)_*(\bu(E))$.

We remark that the strict inclusion $X_{\sigma}\hookrightarrow X$
induces a map $\Sigma(X_{\sigma})\rightarrow \Sigma(X)$. Even though
$X$ is assumed to be Zariski, this need not be an inclusion of cone
complexes. For example, if $X_0$ is a union of two $\PP^1$'s meeting
at two points and $X_\sigma$ is one of these $\PP^1$'s, we have
$\Sigma(X)$ a union of two quadrants glued along their boundary, but
$\Sigma(X_{\sigma})$ consists of two quadrants glued together along
one boundary ray. However, the map $|\Sigma(X_{\sigma})|\rightarrow
|\Sigma(X)|$ will always be injective in a neighbourhood of 
$\sigma$. 

Thus, if $\tau$ is a type of punctured map to $X/B$
such that $\sigma\subseteq \bsigma(x)$ for all $x\in V(G)\cup E(G)\cup L(G)$,
we may view $\tau$ as a type of punctured map to $X_{\sigma}/B$, and hence
as above we obtain a type $\bar\tau$ of punctured 
map to $\overline{X}_{\sigma}$.

\begin{example}
Suppose given a type $\tau$ of punctured maps to $X/B$,
with underlying graph $G$ having precisely one vertex $v$ and adjacent
legs $L_1,\ldots,L_n$. Suppose further that $\bsigma(v)=\sigma$. 
Then as above, we obtain a type $\bar\tau$ of
punctured map to $\overline X_{\sigma}$. However, in fact all
contact orders are positive. Indeed, because $\bsigma(v)=\sigma$,
any $L_i$ must have contact order $\mathbf{u}(L_i)$ lying in the
tangent wedge of $\bsigma(L_i)$ along the face $\sigma$ of $\bsigma(L_i)$.
Hence the image of $\mathbf{u}(L_i)$ in the tangent space to
$(\bsigma(L_i)+\RR\sigma)/\RR\sigma$ in fact lies in this cone.
Therefore we may view the type $\bar\tau$ as a type of logarithmic
map to $\overline{X}_{\sigma}$. Note that the $\btau_v$ occuring in the
first paragraph of this section fits into this example.
\end{example}

Still with $\tau$ a type of punctured map to $X/B$ with $\sigma\subseteq
\bsigma(x)$ for all $x\in V(G)\cup E(g)\cup L(G)$,
we have a commutative diagram
\begin{equation}
\label{eq:punctured to relative diagram}
\xymatrix@C=30pt
{
\scrM(X/B,\btau) \ar[r]\ar[d] & \scrM(\overline{X}_{\sigma},\bar\btau)\ar[d]\\
\foM(\shX/B,\tau) \ar[r] & \foM(\overline \shX_{\sigma},\bar\tau)
}
\end{equation}
where $\overline{\shX}_{\sigma}$ is defined in the same way as
$\overline{X}_{\sigma}$.
While the vertical arrows are the standard ones, the horizontal ones
are as follows. By the definition of a $\btau$-marked curve, we have
$\scrM(X/B,\btau)\cong \scrM(X_{\sigma}/B,\btau)$, i.e., any $\btau$-marked
curve factors through $X_{\sigma}$, by \cite[Def.~3.4,(1)]{ACGSII}.
Given a stable punctured map $f:C^{\circ}/W\rightarrow X_{\sigma}$,
we have $\psi\circ f: C^{\circ}/W\rightarrow\overline{X}_{\sigma}$.
In general, this is neither a pre-stable nor basic punctured map, but
we may pre-stabilize by \cite[Prop.~2.5]{ACGSII} and replace with a
basic punctured map by \cite[Prop.~2.35]{ACGSII}.
This gives the upper horizontal arrow, and a similar discussion
gives the lower horizontal arrow.

\begin{theorem}
\label{thm:punctured to relative}
Let $X\rightarrow B$, $\btau$, $\sigma$ be as in the above discussion,
and suppose further that $\dim\sigma=1$. Then the horizontal arrows
of the diagram \eqref{eq:punctured to relative diagram} are isomorphisms
at the level of underlying stacks (but not at the level of logarithmic
stacks). Furthermore these isomorphisms induce an isomorphism of
obstruction theories.
\end{theorem}

\begin{proof}
We work with the top horizontal arrow of 
\eqref{eq:punctured to relative diagram}, as the Artin fan case is identical.
We need to construct the inverse map. Explicitly, given a basic
$\bar\btau$-marked stable punctured
map $\bar f:\overline{C}^{\circ}/\overline{W}\rightarrow 
\overline{X}_{\sigma}$, we need to construct a stable punctured $\btau$-marked
map $f:C^{\circ}/W\rightarrow X_{\sigma}$ defined over $B$, or
equivalently, over $b_0$. 
The maps $f,\bar f$ will be the same on schemes,
but we need to modify the log structures on $\overline{W}$, $\overline{C}$.

To this end, let $\overline Q_{\bar w}$ be the stalk of the ghost 
sheaf at a geometric point $\bar w \in \overline W$, so that
$\bar\tau_{\bar w}= \overline Q^{\vee}_{\bar w,\RR}$ 
parametrizes a universal family of tropical
maps $\bar h:\Gamma(\bar G_{\bar w},\bar\ell)\rightarrow \Sigma(\overline{X}_{\sigma})$,
inducing for each $s\in \bar\tau$ a tropical map $\bar h_s:\bar G_{\bar w}
\rightarrow \Sigma(\overline{X}_{\sigma})$.

For $r\in\RR_{\ge 0}$, write $\Delta_r:= \delta^{-1}(r)$, 
where $\delta:\Sigma(X_{\sigma})\rightarrow
\Sigma(B)$ is the tropicalization of $X_{\sigma}\rightarrow B$.
Of course $\Delta_1$ agrees with the underlying topological space of
$\Delta(X_{\sigma})$.
Note that we may restrict the map $\Sigma(\psi)$ to $\Delta_r$
to obtain an inclusion of topological spaces $\Delta_r
\hookrightarrow |\Sigma(\overline X_{\sigma})|$. Indeed, it is sufficient
to check that for each $\omega\in\Sigma(X_{\sigma})$, 
$\Sigma(\psi)$ induces an inclusion on $\omega_r:=\omega\cap\delta^{-1}(r)$.
But this follows immediately from the fact that $\Sigma(\psi)|_{\omega}$
is given by localizing along the ray $\sigma\subseteq\omega$ and
$\delta|_{\sigma}$ surjects onto $\Sigma(B)$. We thus view $\Delta_r$
in this way as a subspace of $\Sigma(\overline{X}_{\sigma})$.

One also easily sees that the spaces
$\Delta_r$ exhaust $|\Sigma(\overline X_{\sigma})|$, i.e., for
$y \in |\Sigma(\overline X_{\sigma})|$, there exists an $r_0\ge 0$
such that $y\in \Delta_r$ for all $r\ge r_0$.

Let $\bar G'_{\bar w}$ be the subgraph of $\bar G_{\bar w}$ 
obtained by deleting all legs of $\bar G_{\bar w}$. We define a function
$\alpha:\bar\tau_{\bar w}\rightarrow \RR_{\ge 0}$ by 
\[
\alpha(s):=\inf\{ r \in \RR_{\ge 0}\,|\, \bar h_s(\bar G'_{\bar w})
\subseteq \Delta_r\}.
\]
The set on the right-hand side is non-empty, and hence this makes sense.

\begin{claim}
$\alpha$ is an upper convex piecewise linear function on 
$\bar\tau_{\bar w}$ with rational slopes.
\end{claim}

\begin{proof}
Suppose $\omega\in\Sigma(X_{\sigma})$, with $\sigma\subseteq\omega$.
Let $F^1,\ldots,F^p\subseteq\omega$ be the codimension one faces of
$\omega$ not containing $\sigma$ with $\delta|_{F^i}$ surjective.
Let $H=\delta^{-1}(0)\cap\omega$. Note $H$ does not contain
$\sigma$.
We then have the Minkowski decomposition
\[
\omega_r = \mathrm{Conv}\{F^1_r,\ldots,F^p_r,\sigma_r\}+H
\]
where $\mathrm{Conv}$ denotes convex hull.

Let $n_{F^i}\in \omega^{\vee}$ be a rational generator of the dual 
one-dimensional face to $F^i$. Since $\sigma\not\subseteq F^i$, 
$n_{F^i}$ is positive on $\sigma\setminus\{0\}$.
So we can normalize $n_{F^i}$ so that $n_{F^i}|_{\sigma}$ agrees with 
$\delta|_{\sigma}$ by rescaling by a rational number. Thus $\delta-n_{F^i}$
vanishes on $\sigma$ and hence descends to a linear function on
$\bar\omega=(\omega+\RR\sigma)/\RR\sigma$. 
Further, $\delta-n_{F^i}$ takes the value $r$ on $F^i_r$.

We now show that
\begin{equation}
\label{eq:omegar image}
\Sigma(\psi)(\omega_r) = \{ m \in \bar\omega\,|\, 
\langle \delta- n_{F^j} ,m\rangle \le r, 1\le j\le p\}.
\end{equation}
To see the forward inclusion, we may write an element $w$
of $\omega_r$ as 
\[
\sum_{i=1}^p a_i f_i + a_{p+1} s + h
\]
with $f_i\in F^i_r$, $s\in \sigma_r$, $\sum_{i=1}^{p+1}a_i=1$, 
$a_i\ge 0$, and $h\in H$. Then 
\[
(\delta-n_{F^j})(w) = \sum_{i=1}^p a_i r -
\sum_{i=1}^p a_i \langle n_{F^j},f_i\rangle - \langle n_{F^j},h\rangle
\le r.
\]
Conversely, if $m$ lies in the right-hand side of \eqref{eq:omegar image},
choose a lift $m'$ of $m$ to $\omega$. Write $\sigma_1=\{s\}$, and 
set $m''=m'+s (r-\delta(m'))$. Then $m''$ is another lift of $m$ to
$\omega^{\gp}$, but may not lie in $\omega$. However,
$\delta(m'')=r$, and we would like to show
$m''\in\omega$. First, $\langle \delta - n_{F^j}, m''\rangle 
=\langle \delta - n_{F^j},m\rangle \le r$, so 
$\langle n_{F^j}, m''\rangle \ge \delta(m'')-r = 0$. If $K\subseteq
\omega$ is a codimension one face of $\omega$ containing $\sigma$, and
$n_{K}$ is a generator of the dual ray, then $\langle n_{K}, 
m''\rangle
=\langle n_{K}, m'\rangle \ge 0$. Finally, if $H=\delta^{-1}(0)
\cap\omega$ is a codimension one
face of $\omega$, then $\delta$ is a generator of the dual ray and
$\delta(m'')=r>0$. Thus we see that $m''$ is positive on all generators
of rays of $\omega^{\vee}$, and hence lies in $\omega$, hence
in $\omega_r$. So $m$ lies in $\Sigma(\psi)(\omega_r)$. We have
now shown \eqref{eq:omegar image}.

Now for each vertex $v\in V(\bar G_{\bar w})$, we have the
evaluation map $\ev_v:\bar\tau_{\bar w}\rightarrow
\bar\bsigma(v)$ at $v$ given by \eqref{eq:eval def}. 
Note that $\bar h_s(\bar G'_{\bar w})
\subseteq \Delta_r$ if and only if $\bar h_s(v)\in \Sigma(\psi)(\bsigma(v)_r)$ 
for all $v\in V(\bar G_{\bar w})$, where $\bsigma(v)$ is the cone of
$\Sigma(X_{\sigma})$ mapping to $\bar\bsigma(v)$ under $\Sigma(\psi)$.
In particular, 
for a given $v$, $\bar h_s(v)\in \Sigma(\psi)(\bsigma(v)_r)$
if and only if
$\langle \delta - n_{v,F^j},\ev_v(s)\rangle \le r$ for $n_{v,F^j}$
running over
the set of normal vectors to codimension one faces of $\bsigma(v)$ 
not containing $\sigma$ or contained in $\delta^{-1}(0)$ as above. 
Thus $\alpha$ is given as 
\[
\alpha(s)= \sup \{\langle \delta-n_{v,F^j},\ev_v(s)\rangle\},
\]
with the supremum running over all $v\in V(\bar G_{\bar w})$ and normal
vectors $n_{v,F^j}$, i.e., $\alpha$ agrees with the supremum
of the linear functionals $\ev_v^*(\delta-n_{v,F^j})$. This
makes $\alpha$ piecewise linear and (upper) convex.
\end{proof}

We may now define 
\[
\tau_{\bar w}:=\{(s,r)\in \bar\tau_{\bar w}\times\RR_{\ge 0}
\,|\,r\ge \alpha(s)\}.
\]
This is a rational polyhedral cone by the claim. By construction,
$\tau_{\bar w}$ parameterizes lifts of the tropical maps $\bar h_s$ to
$\Sigma(X_{\sigma})$. Indeed, if $(s,r)\in \tau_{\bar w}$,
$\bar h_s$ maps $\bar G'_{\bar w}$ into 
$\Delta_r$, which we now view as a subset of $|\Sigma(X_{\sigma})|$.
We may then extend this map as far as possible along each leg $L_i$ to
define a domain graph $G_{\bar w}$, which coincides with $\bar G_{\bar w}$ 
except that $\bar G_{\bar w}$ may have some unbounded legs which
are replaced with line segments. We then get
$h_{s,r}:G_{\bar w}\rightarrow \Sigma(X_{\sigma})$. In particular, 
an unbounded leg $L_i$
turns into a punctured leg, i.e., a line segment, if only a portion of 
$\bar h_s(L_i)$ is contained in $\Delta_r$.

Note here we may view $\tau_{\bar w}$ as type of tropical map
to $\Sigma(X_{\sigma})$, with $\bsigma$ defined in the obvious way
from $\bar\bsigma$ and $\bu(E)$ a tangent vector of $\bsigma(E)_1$ mapping
under $\Sigma(\psi)_*$ to $\bar\bu(E)$.
In particular, we obtain a universal family of tropical maps of type
$\tau_{\bar w}$,
\begin{equation}
\label{eq:lifted family}
h_{\bar w}:\Gamma(G,\ell)\rightarrow \Sigma(X_{\sigma}).
\end{equation}

We then have the corresponding basic monoid
\[
Q_{\bar w} := \tau_{\bar w}^{\vee}\cap (N_{\bar\tau_{\bar w}}^*\oplus\ZZ).
\]
Note that 
\[
Q_{\bar w} \subseteq N_{\bar\tau_{\bar w}}^*\oplus\NN
\]
since $(0,1)\in \tau_{\bar w}$.
The monoid $Q_{\bar w}$ will be the basic monoid for the point 
$\bar w$ for the punctured log map to $X_{\sigma}$ we are going to construct.
Note the projection $\tau_{\bar w}\rightarrow \bar\tau_{\bar w}$
dualizes to an inclusion $\overline{Q}_{\bar w}\hookrightarrow
Q_{\bar w}$ which identifies $\overline{Q}_{\bar w}$ with
the facet $\{(q,0)\in Q_{\bar w}\}$ of $Q_{\bar w}$. 

It is easy to check that if
$\bar w'$ is a generization of $\bar w$, the above construction of $Q_{\bar w}$ 
is compatible with generization maps, i.e., the generization map
$\overline Q_{\bar w}\rightarrow \overline Q_{\bar w'}$ 
induces a generization map $Q_{\bar w}\rightarrow Q_{\bar w'}$. Indeed,
dually,
we have an inclusion of faces $\bar\tau_{\bar w'}\subseteq \bar\tau_{\bar w}$, 
and $\alpha:\bar\tau_{\bar w}\rightarrow\RR_{\ge 0}$ 
restricts to the corresponding
map for $\bar\tau_{\bar w'}$. Hence we obtain an inclusion of faces
$\tau_{\bar w'}\subseteq\tau_{\bar w}$. From this, we see that
the monoids $Q_{\bar w}$ for
various $\bar w$ define a fine subsheaf $\overline{\cM}_W$ 
of $\overline{\cM}_{\overline{W}}^{\gp}\oplus\NN$ containing
$\overline{\cM}_{\overline{W}}\oplus\NN$.

Let $W^{\dagger}=\overline W\times b_0$, so that $\overline{\cM}_{W^{\dagger}}
= \overline{\cM}_{\overline W}\oplus\NN$, and set 
\[
\cM_W:=\overline{\cM}_W\times_{\overline{\cM}_{W^{\dagger}}^{\gp}}
\cM_{W^{\dagger}}^{\gp}.
\]
We may then define a structure morphism $\alpha_W:\cM_W\rightarrow \cO_W$
by taking $\alpha_W=\alpha_{\overline{W}}$ on $\cM_{\overline W}\oplus 0
\subseteq \cM_W$ and $\alpha_W$ taking the value $0$ on 
$\cM_W\setminus (\cM_{\overline W}\oplus 0)$. Thus we obtain an fs log scheme $W$.

We have morphisms $W\rightarrow W^{\dagger} \rightarrow \overline W$ 
by construction,
and the log smooth curve $\overline C\rightarrow \overline W$ 
pulls back to give log smooth curves
$C^{\dagger}\rightarrow W^{\dagger}$
and $C\rightarrow W$. In addition, the punctured curve
$\overline C^{\circ}\rightarrow\overline W$ pulls back to give a
punctured curve 
\[
(C^{\dagger})^{\circ}=\overline C^\circ\times b_0
\rightarrow W^{\dagger}.
\]
The morphism 
$\bar f:\overline C^{\circ}\rightarrow \overline{X}_{\sigma}$ 
induces a morphism $f^{\dagger}:(C^{\dagger})^{\circ}\rightarrow
\overline X_{\sigma}\times b_0$ defined over $b_0$.
We wish to define a punctured
structure $C^{\circ}$ on $C$ yielding a commutative diagram
\[
\xymatrix@C=30pt
{
C^{\circ}\ar[r]^f\ar[d]& X_{\sigma}\ar[d]\\
(C^{\dagger})^{\circ}\ar[r]_{f^{\dagger}}&\overline X_{\sigma}\times b_0
}
\]
As $\cM_{W^{\dagger}}^{\gp}=\cM_{W}^{\gp}$, we would have $\cM^{\gp}_{C^{\circ}}
=\cM^{\gp}_{C^{\dagger}}$ if $C^{\circ}$ exists. Also, 
$\cM_{X_{\sigma}}^{\gp}=\cM_{\overline{X}_\sigma\times b_0}^{\gp}$, so
$f^{\flat}$ and $(f^{\dagger})^{\flat}$ would have to agree at the level
of groups.
It is thus sufficient to construct a pre-stable punctured log structure 
$C^{\circ}$ on $C$ so that
\begin{equation}
\label{eq:needed inclusion}
\overline{(f^{\dagger})}^{\flat}(\overline{\cM}_{X_{\sigma}})
\subseteq \overline{\cM}_{C^{\circ}}.
\end{equation}
Note that the saturation of $\overline{\cM}_{C^{\circ}}$ over $\bar w
\in W$, as well as the map $\overline{(f^\dagger)}^{\flat}:f^{-1}
\overline\shM_{X_{\sigma}}\rightarrow\overline\shM_{C^\circ}^{\mathrm{sat}}$,
are completely determined locally near $\bar w$ by
the data of the tropical family of maps $h_{\bar w}$
of \eqref{eq:lifted family}. Thus we only need to consider
the punctured points. 

So let $p\in C_{\bar w}$ be a marked point corresponding to a leg $L$
of $\bar G_{\bar w}$ (or the corresponding leg of $G_{\bar w}$). 
We have the monoids $\overline P_{\bar\bsigma(L)}$ 
(resp.\ $P_{\bsigma(L)}$) which are the stalks
of $\overline{\cM}_{\overline{X}_{\sigma}}$ 
(resp.\ $\overline{\cM}_{X_{\sigma}}$)
at the generic point of the
strata of $\overline{X}_{\sigma}$ (resp.\ $X_{\sigma}$)
corresponding to $\bar\bsigma(L)$ (resp.\ $\bsigma(L)$).
Then $\overline P_{\bar\bsigma(L)}$ is a codimension one face
of $P_{\bsigma(L)}$, and 
we have a commutative diagram with horizontal arrows induced by
$\bar f$, $f^{\dagger}$ and the family of tropical maps parameterized by
$\tau_{\bar w}$:
\[
\xymatrix@C=30pt
{
\overline{P}_{\bar\bsigma(L)}\ar[r]\ar[d]&
\overline{\cM}_{\overline{C}^{\circ},p}
\subseteq\overline Q_{\bar w}\oplus\ZZ\ar[d]\\
\overline{P}_{\bar\bsigma(L)}\oplus\NN\ar[r]\ar[d]&
\overline{\cM}_{(C^{\dagger})^{\circ},p}\subseteq
(\overline Q_{\bar w}\oplus\NN)\oplus\ZZ
\ar[d]\\
P_{\bsigma(L)} \ar[r] & \overline{\cM}_{C^{\circ},p}\subseteq 
Q_{\bar w}\oplus \ZZ
}
\]
Here the inclusion $\overline{P}_{\bar\bsigma(L)}\oplus\NN\rightarrow 
P_{\bsigma(L)}$ identifies $P_{\bsigma(L)}$ with a submonoid of
$\overline{P}^{\gp}_{\bar\bsigma(L)}\oplus\NN$ with the facet
$\overline{P}_{\bar \bsigma(L)}$ of $P_{\bsigma(L)}$ identified with
$\overline{P}_{\bar \bsigma(L)}\oplus 0$.
To obtain a pre-stable puncturing, we take the stalk of the ghost sheaf
of $\overline{\cM}_{C^{\circ}}$ at $p$ to be the submonoid of 
$Q_{\bar w}\oplus\ZZ$
generated by $Q_{\bar w}\oplus\NN$ and the image of $P_{\bsigma(L)}$
under $\overline{(f^{\dagger})}^{\flat}$.
We just need to make sure that this defines a puncturing.

Let $\varphi_L:\overline{P}_{\bar\bsigma(L)}
\oplus\NN\rightarrow \overline Q_{\bar w}\oplus \NN$ be the composition
of $\overline{(f^{\dagger})}^{\flat}$
with the projection to $Q_{\bar w}\oplus\NN$. By construction,
$\varphi_L$ preserves the second component, i.e., $\varphi_L(p,r)=(q,r)$
for some $q$. Also by construction, $\varphi_L^{\gp}$ coincides
with the similarly defined map on groups
$P^{\gp}_{\bsigma(L)}\rightarrow \overline{\shM}_{C^{\circ},p}^{\gp}$
coming from $f^\flat$, and $\varphi_L^{\gp}
(P_{\bsigma(L)}\setminus \overline{P}_{\bar\bsigma(L)})\subseteq
Q_{\bar w}\setminus \overline{Q}_{\bar w}$.
Suppose $(\bar m_1,\bar m_2)\in Q_{\bar w}\oplus\ZZ$ is
in the image of $P_{\bsigma(L)}$ and $\bar m_2<0$. 
To show that $\shM_{C^{\circ}}$ is a puncturing, we need to show that if
$m_1$ is a lift of $\bar m_1$ to a local
section of $\shM_W$, then $\alpha_W(m_1)=0$. Now
$\bar m_1 = \varphi_L^{\gp}(\bar n_1)$
for some $\bar n_1 \in P_{\bsigma(L)}$. We have two cases.
If $\bar n_1\in
P_{\bsigma(L)}\setminus \overline{P}_{\bar\bsigma(L)}$, then
$\varphi_L^{\gp}(\bar n_1)\in Q_{\bar w}\setminus \overline{Q}_{\bar w}$.
Hence, by definition of the structure map $\alpha_W$, we have
$\alpha_W(m_1)=0$, as desired. On the other hand, if
$\bar n\in \overline P_{\bsigma(L)}$, then $\alpha_W(m_1)=\alpha_{\overline{W}}
(m_1)=0$, as $\overline{C}^{\circ}$ is a punctured curve.

We have thus now constructed a basic punctured map $f:C^{\circ}/W
\rightarrow X_{\sigma}$. This construction
is functorial and hence defines a morphism
\[
\underline{\scrM(\overline{X}_{\sigma},\bar\btau)}
\rightarrow \underline{\scrM(X_{\sigma}/B,\btau)}.
\]

This morphism is checked to be the inverse to
the morphism 
$\ul{\scrM(X_{\sigma}/B,\btau)}\rightarrow 
\underline{\scrM(\overline X_{\sigma},\bar\btau)}$
previously constructed. 
Hence these stacks are isomorphic. Note they carry different log structures.

To compare obstruction theories, note that given the universal
punctured log maps $f:C^{\circ}\rightarrow X_{\sigma}$ over
$\scrM(X_{\sigma}/B,\btau)$ and  $\overline f:\overline C^{\circ}
\rightarrow \overline X_{\sigma}$ over $\scrM(\overline X_{\sigma},\btau)$,
the obstruction theories are given by $R\pi_* f^*\Theta_{X/B}$ and
$R\bar\pi_* f^*\Theta_{\overline X_{\sigma}}$ respectively, 
where $\pi,\bar\pi$ are 
the maps from the universal curves down to their moduli spaces.
However, $\Theta_{X/B}|_{X_{\sigma}}=\Theta_{\overline X_{\sigma}}$.
Indeed, working dually with sheaves of differentials, 
$\Omega_{X/B}|_{X_{\sigma}}\cong \Omega_{X_{\sigma}/B}$ by 
\cite[IV,Cor.~2.3.3]{Ogus}. On the other hand, there is a natural
morphism $\Omega^1_{\overline{X}_{\sigma}}\rightarrow
\Omega_{X_{\sigma}/B}$ by \cite[IV,Prop.~1.2.15]{Ogus}. This is checked
to be an isomorphism explicitly on charts using \cite[IV,Prop.~1.1.4]{Ogus}.
\end{proof}

\section{The classical degeneration formula}
\label{sec:classical}

The classical degeneration situation, originally considered by Li
and Ruan in \cite{LR} and developed in the algebro-geometric context by
Jun Li in \cite{Li} is a special case of the degeneration situation of
the previous section. We have already discussed this case in the
context of the decomposition formula in \cite[\S6.1]{ACGSI}. We consider
$X\rightarrow B$ a simple normal crossings degeneration with 
$X_0=Y_1\cup Y_2$ a reduced union of two irreducible components,
with $Y_1\cap Y_2=D$ a smooth divisor in both $Y_1$ and $Y_2$.
In this case $\Sigma(X)=(\RR_{\ge 0})^2$ with map $\Sigma(X)\rightarrow
\Sigma(B)$ given by $(x,y)\mapsto x+y$, so that $\Delta(X)$ is a unit
interval. It was shown in \cite[Prop.~6.1.1]{ACGSI} 
that if $f:\Gamma\rightarrow
\Delta(X)$ is a rigid tropical map, then all vertices of $\Gamma$
map to endpoints of $\Delta(X)$ and all edges of $\Gamma$ surject
onto $\Delta(X)$. In this case, giving a rigid tropical map with target 
$\Delta(X)$ is the same information as an admissible triple of \cite{Li}.

Using the setup of \S\ref{sec:degeneration gluing}, we immdiately obtain
the logarithmic stable map version of the main result of \cite{Li},
as proved in \cite{KLR}:

\begin{theorem}
In the situation described above, let $\btau$ be a decorated
type of rigid tropical map in $\Delta(X)$. There is a diagram
of (non-logarithmic) stacks
\[
\xymatrix@C=30pt
{
\ul{\scrM(X_0/b_0,\btau)}\ar[r]^{\phi'}&
\ul{\scrM^{\mathrm{sch}}(X_0/b_0,\btau)}\ar[r]
\ar[d]& \prod_{v\in V(G)}
\ul{\scrM(\overline X_{\bsigma(v)},\bar\btau_v)}\ar[d]\\
&\prod_{E\in E(G)}\ul{D}\ar[r]_{\Delta}&
\prod_{v\in E\in E(G)} \ul{D}
}
\]
with the square Cartesian and defining the
space $\ul{\scrM^{\mathrm{sch}}(X_0/b_0,\btau)}$. 
Further, $\phi'$ is finite and
\[
\phi'_*[\scrM(X_0/b_0,\btau)]^{\virt}
= m_{\btau}^{-1}
\left(\prod_{E\in E(G)} w(E)\right) \Delta^!\left(\prod_{v\in V(G)} 
[\scrM(\overline X_{\bsigma(v)},\bar\btau_v)]^{\virt}\right),
\]
where $w(E)$ is the index (degree of divisibility) of $\mathbf{u}(E)$.
\end{theorem}

\begin{proof}
Using Theorem \ref{thm:punctured to relative}, 
the given diagram is a part of the diagram of Theorem 
\ref{thm:gluing factorization}. The result will follow from Theorem
\ref{thm:tropically transverse}. Thus we first verify the tropical
transversality condition, and calculate $\mu(\btau)$.

Note that $\overline{N}_{\bsigma(v)}=0$, while
each $\overline{N}_{\bsigma(E)}$ can be identified with $\ZZ\cong \ZZ (1,-1)
\subseteq \ZZ^2$ for any vertex $v$, edge $E$. Thus the morphism
$\Psi$ of Definition \ref{def: tropically transverse} takes the form
\[
\Psi:\bigoplus_{E\in E(G)}\ZZ \rightarrow \bigoplus_{E\in E(G)}\ZZ
\]
given by $\Psi\big((\ell_E)_{E\in E(G)}\big)=\big(\ell_E \mathbf{u}(E)\big)$.
In particular, the image has finite index, and this index is
$\prod_{E} w(E)$.

By Theorem \ref{thm:tropically transverse}, it is thus sufficient to
show that $\ul{\ev}$ is flat. However, this follows immediately
from Theorem \ref{thm:flatness}, as $\bsigma'(L)=\bsigma(L)$ always in
this situation.
\end{proof}

\section{Applications to wall structures for type III degenerations
of K3 surfaces}

In this section, we will work with a specific kind of degeneration
$g:X\rightarrow S$. Here, we will use $S$ for the base
log scheme rather than $B$ as is done in the previous sections
and in \cite{ACGSII}, as $B$ will notationally play a different role.
We make the following assumptions:
\begin{assumptions}
\label{ass:K3}
\begin{enumerate}
\item
$\ul{S}$ is a one-dimensional non-singular scheme with closed
point $0\in \ul{S}$ inducing the divisorial log structure on $S$.
\item $\ul{X}\rightarrow \ul{S}$
is a simple normal crossings degeneration of K3 surfaces.
\item The fibre over $0$, $X_0$, is the only singular fibre, 
necessarily a simple
normal crossings divisor, and the log structure on $X$ is the divisorial
log structure induced by $X_0$. 
\item With $D:=(X_0)_{\red}$ the reduction,
we have $K_{X/S}+D=0$. 
\item
Any intersection of irreducible components of $D$ is
connected. 
\item $D$ has a zero-dimensional stratum. 
\end{enumerate}
\end{assumptions}

Here we will give a useful inductive description for the so-called
\emph{canonical wall structure} defined in \cite{Walls}. We make use
of an observation of Ranganathan in \cite[\S6.5.2]{Ra}, following
Parker \cite{Parker}, that gluing
remains fairly easy in the case the normal crossings divisor
$D$ has at worst triple points and the domain curve is genus $0$.
Rather than give the general description of this approach in
our language, we just carry out the procedure in our particular
application.

\subsection{Review of wall types and balancing}

We begin by recalling certain concepts from \cite{Walls}. 
By the assumptions on $X\rightarrow S$ made above, all irreducible
components of $D$ are good in the sense of \cite[\S1.1]{Walls}.
Further, Assumptions 1.1 and 1.2 of \cite{Walls} hold. Thus,
in the notation of that paper, we may take $B=|\Sigma(X)|$ 
and $\P$ the set of cones of $\Sigma(X)$, so that $(B,\P)$
is a pseudo-manifold as explained in \cite[Prop.~1.3]{Walls}. We set
\[
\Delta := \bigcup_{\sigma\in\Sigma(X)\atop \codim\sigma \ge 2}
\sigma,
\]
and $B_0:=B\setminus \Delta$. Then \cite[\S1.3]{Walls} gives the structure of
integral affine manifold to $B_0$. Further, by \cite[Prop.~1.15]{Walls},
the tropicalization of $g$ induces a map
$g_{\trop}:B\rightarrow \Sigma(S)=\RR_{\ge 0}$ which is an affine
submersion. We set $B'=g_{\trop}^{-1}(1)$, and
\[
\P'=\{\sigma\cap g_{\trop}^{-1}(1)\,|\,\sigma\in\P\},
\]
a polyhedral decomposition of $B'$. In the notation of the previous
sections, $\P'$ is the set of polyhedra in $\Delta(X)$, but we wish
to avoid this notation now to avoid conflict with $\Delta$ as the
discriminant locus. Here we also set 
\[
\Delta':=B'\cap \Delta,\quad\quad B_0':=B'\setminus \Delta.
\]

We use the convention that for $\sigma\in \P$ and 
$\sigma'=\sigma\cap B'\in \P'$, we write either $X_{\sigma}$
or $X_{\sigma'}$ for the stratum of $X$ corresponding to $\sigma$.
Under this convention, irreducible components of $D$ correspond to 
vertices of $\P'$, and a vertex $v$ has integral coordinates
if and only if the corresponding irreducible component $X_v$ of
$X_0$ appears with multiplicity $1$ in $X_0$. Thus the cells of
$\P'$ are not in general lattice polytopes. In case $X_0$ is
reduced, then \cite[Prop.~1.16]{Walls} applies, but we do not wish
to make this restriction. In any event, $B'$ is still an affine manifold
with singularities.

We recall from \cite[Lem.~2.1]{Walls}:

\begin{lemma}
\label{lem:balancing}
Let $f:C^{\circ}/W\rightarrow X$ be a stable punctured map to $X$,
with $W=\Spec(Q\rightarrow\kappa)$ a geometric log point. For
$s\in \Int(Q^{\vee}_{\RR})$, let $h_s:G\rightarrow B$ be
the corresponding tropical map. If $v\in V(G)$ satisfies
$h_s(v)\in B_0$, then $h_s$ satisfies the balancing condition
at $v$. More precisely, if $E_1,\ldots,E_n$ are the legs
and edges adjacent to $v$, oriented away from $v$,
then the contact orders $\mathbf{u}(E_i)$ may be interpreted
as elements of $\Lambda_{h_s(v)}$, the stalk at $h_s(v)$ of the
local system $\Lambda$ of integral tangent vectors.
In this group, the balancing condition
\begin{equation}
\label{eq:balancing}
\sum_{i=1}^m \mathbf{u}(E_i)=0
\end{equation}
is satisfied.
\end{lemma}

As a consequence, we say a realizable tropical type $\tau$ of tropical map
to $\Sigma(X)/\Sigma(S)$ is \emph{balanced} if, for any $s\in\Int(\tau)$,
the corresponding tropical map $h_s:G\rightarrow B$ satisfies the
balancing condition of this lemma.

We then define:

\begin{definition}
\label{def:wall type}
A \emph{wall type} is a type $\tau=(G,\bsigma,{\mathbf u})$ 
of tropical map to $\Sigma(X)$ defined over $\Sigma(S)$ such that:
\begin{enumerate}
\item
$G$ is a genus zero graph with $L(G)=\{L_{\out}\}$ and
$u_{\tau}:={\bf u}(L_{\out})\not=0$.
\item $\tau$ is realizable and balanced.
\item Let $h:\Gamma(G,\ell)\rightarrow \Sigma(X)$ be the corresponding
universal family of tropical maps, and $\tau_{\out}\in \Gamma(G,\ell)$
the cone corresponding to $L_{\out}$. Then $\dim\tau=1$ and
$\dim h(\tau_{\out})=2$.
\end{enumerate}
A \emph{decorated wall type} is a decorated type $\btau=(\tau,{\bf A})$
with $\tau$ a wall type and 
\[
{\bf A}:V(G)\rightarrow \coprod_{\sigma\in\Sigma(X)} H_2(X_{\sigma})
\]
is a \emph{refined decoration}, 
i.e., ${\bf A}(v)\in H_2(X_{\bsigma(v)})$ for $v\in V(G)$.
\end{definition}

\begin{remarks}
(1) The notion of refined decoration, in which curve classes 
lie in the group of curve classes of the relevant stratum rather
than of $X$, gives more control over moduli spaces. In particular,
the moduli space $\scrM(X/S,\btau)$ for $\btau$ carrying a 
refined decoration is simply a union of connected components of
the moduli space where the decoration is obtained by pushing forward
all curve classes $\mathbf{A}(v)$ to $H_2(X)$.

(2) We note that, other than the issue of refined decorations,
this definition is slightly simpler than the one
given in \cite[Def.~3.6]{Walls} in that (1) we have specialized to the
case that $\dim X=3$ and all irreducible components of $D$ are
good; (2) $B$ does not have a boundary in our case; and
(3) here we insist that the
type be defined over $\Sigma(S)$, but by \cite[Prop.~3.7]{Walls},
this is implied by \cite[Def.~3.6]{Walls}. As a consequence, the two
notions of wall type coincide in this particular case.

(3) We note that because of our restriction to $\dim X=3$, given a wall
type $\tau$, there is a unique tropical map $h'_{\tau}:G\rightarrow B$
of type $\tau$ factoring through $B'$. We will write this as
$h'_{\tau}:G\rightarrow B'$, or $h'$ when clear from context. Such a 
tropical map is rigid.
\end{remarks}

It is shown in \cite[Const.~3.13]{Walls} that if $\btau$ is a 
decorated wall type, then $\scrM(X/\kk,\btau)\cong \scrM(X/S,\btau)$, and
this stack is proper over $\Spec\kk$
and carries a zero-dimensional virtual fundamental class.
As a consequence, from now on we will work over $\Spec\kk$ rather than
$S$, writing $\scrM(X,\btau)$, $\foM(\shX,\tau)$ for the relevant moduli
spaces.

We then define
\[
W_{\btau}:={\deg [\scrM(X,\btau)]^{\virt}\over |\Aut(\btau)| }.
\]
We also define the number $k_{\tau}$ as follows. The map
$h:\Gamma(G,\ell)\rightarrow \Sigma(X)$ induces a homomorphism 
\[
h_*:N_{\tau_{\out}}\rightarrow N_{\bsigma(L_{\out})}.
\]
We then define
\[
k_{\btau}:=|\coker(h_*)_{\tors}|.
\]

\begin{example}
Let $v\in \P'$ be a vertex with corresponding irreducible component
$X_v$. Let $D_v\subseteq X_v$ be the union of lower-dimensional strata
of $D$ contained in $X_v$. Suppose $X_v$ contains a rational curve
$E$ with self-intersection $-1$. Suppose further that
$E$ meets $D_v$ transversally
at one point. This point is contained in a one-dimensional
stratum $X_{\rho}$ for $\rho\in\P'$ an
edge. Necessarily, $v$ is one endpoint of $\rho$.
Let $v'$ be the other endpoint of $\rho$, and $\mu,\mu'$ be the
multiplicities of the irreducible components $X_v, X_{v'}$ in the
fibre $X_0$. Let $k$ be such that $\mu | k \mu'$.

Consider a type $\tau$ where $G$ has one vertex $w$, one leg
$L$, and no edges. We take $\bsigma(w)$ to be the ray of $\P$ corresponding to 
$v$ and $\bsigma(L)$ to be the cone of $\P$ corresponding to $\rho$.
Before specifying $\mathbf{u}(L)$, we decorate $\tau$ by taking
$\mathbf{A}(w)=k[E] \in H_2(X_v)$. Then $\mathbf{A}(w)\cdot X_{v'}=k$.
Since $E\subseteq X_v\cup X_{v'}$ and $E\cdot X_0=0$, necessarily
$\mathbf{A}(w)\cdot (\mu X_v + \mu' X_{v'})=0$. From this we conclude that
$\mathbf{A}(w)\cdot X_v = -k\mu'/\mu$, which is an integer by assumption.
Thus there is a unique choice of $\mathbf{u}(L)$ compatible with
these intersections, by Corollary \ref{cor:114}.

In this case, we may apply Theorem \ref{thm:punctured to relative} to
see that $\scrM(X/S,\btau)\cong \scrM(\overline X_v, \bar\btau)$ in
the notation of \S\ref{sec:punctured versus relative}. Here
$\overline X_v$ will be the log scheme structure on $\ul{X}_v$ induced
by the divisor $D_v\subseteq \ul{X}_v$. Note that any stable log map 
of curve class $kE$ is a multiple cover of $E$. Via the comparison
of logarithmic and relative invariants of \cite{AMW}, we may use
the calculation of \cite[Prop.~5.2]{GPS} to obtain that
$\deg [\scrM(X/S,\btau)]^{\virt} = (-1)^{k+1}/k^2$. 

On the other hand, a simple calculation shows that $k_{\tau}=k$.
\end{example}

Our goal here is to give an inductive method for computing 
$k_{\tau}W_{\tau}$, the quantity which plays a key role in
the construction of the canonical wall structure of \cite{Walls}. 
The methods here are particular to relative dimension two, and the structure
of these invariants in higher dimensions is more subtle and will
be developed in \cite{LocalGlobal}.

\subsection{An invariant of Looijenga pairs}
\label{subsect:invariant LP}

Recall a \emph{Looijenga pair} is a pair $(X,D)$ where $X$ is
a non-singular projective
rational surface and $D$ is a reduced nodal anti-canonical
divisor with at least one node. In what follows, assume that $D$
has at least three irreducible components. In this case, $D$ will be a
cycle of rational curves, and we write $D=D_1+\cdots+D_n$.
We are again in the situation of \cite{Walls},
and may again take $B=|\Sigma(X)|$ and $\P$ the set of cones
of $\Sigma(X)$. Then $B_0:=B\setminus \{0\}$ carries the structure
of an integral affine manifold, see \cite[\S1.3]{Walls} again, but
the original construction was given in \cite[\S1.2]{GHK}. Then
Lemma~\ref{lem:balancing} still holds in this case. 

We will also allow the more general possibility that $(X,D)$ is
a log smooth surface over $\Spec \kk$, so that $(X,D)$ is a
toroidal crossings pair. In this case, we require that there should be
a birational proper log \'etale morphism
$(\widetilde X, \widetilde D)\rightarrow (X,D)$ such that
$(\widetilde X,\widetilde D)$
is a Looijenga pair in the
sense of the previous paragraph.
Note the singularities of $X$ only occur
at the nodes of $D$. In this case, we
similarly obtain $B=|\Sigma(X)|$, $\P$ the set of cones of $\Sigma(X)$,
now no longer necessarily standard cones. As is standard, such a morphism
$\pi$ is described by specifying $\Sigma(\widetilde X)$ as being a 
refinement of $\Sigma(X)$. Thus
$B_0:=B\setminus\{0\}$ still carries the structure of an integral
affine manifold via the identification of supports
$|\Sigma(X)|=|\Sigma(\widetilde X)|$.

In the sequel, Looijenga pair will refer to this more general case, and we will
refer to a non-singular Looijenga pair when we impose that $X$ be non-singular.

\begin{example}
Suppose we are in the situation of Assumptions \ref{ass:K3}, with 
$v\in \P'$ a vertex and $X_v$ the corresponding irreducible component
of $D$. With $D_v$ the union of lower-dimensional strata of $D$
contained in $X_v$, it follows from the assumption that $K_{X/S}+D=0$
and adjunction that $K_{X_v}+D_v=0$. Note that $D_v$ is nodal,
and then it is easy to see, e.g., by the classification of surfaces, that
$(X_v,D_v)$ is a Looijenga pair.
\end{example}

Fix a class $\beta$ of logarithmic map of genus $0$ to a Looijenga pair $(X,D)$,
with $q+1$ marked points, and contact orders $u_1,\ldots,u_q,u_{\out}$.
Here, we assume that $u_k=w_k \nu_{i_k}$, $u_{\out}= w_{\out} \nu_{\out}$,
where $\nu_{i_k}$ is a primitive
generator of a ray $\rho_{i_k}$ of $\P$ corresponding to the 
irreducible component $D_{i_k}$ of $D$, and $w_k$ is a positive integer, 
and similarly for $w_{\out}$ and $\nu_{\out}$. We do not assume the
irreducible components $D_{i_k}$ are distinct.

We then have a schematic evaluation map
\begin{equation}
\label{eq:evaluation}
\ul{\ev}:\scrM(X,\beta)\rightarrow \prod_{k=1}^q D_{i_k}
\end{equation}
given by evaluating at the marked points with contact orders
$u_1,\ldots,u_q$. It follows from a standard virtual dimension
calculation via Riemann-Roch that the virtual dimension of
$\scrM(X,\beta)$ is $q$. We define $N_{\beta}$ by the identity
\begin{equation}
\label{eq:Nbeta def}
\ul{\ev}_*[\scrM(X,\beta)]^{\virt} = N_{\beta}\prod_{k=1}^q [D_{i_k}].
\end{equation}

As recalled above, subdivisions of $(B,\P)$ correspond to 
proper log \'etale
birational morphisms $(\widetilde X,\widetilde D)\rightarrow 
(X,D)$.
As the contact orders in $\beta$ for $X$ are integral points of
$B$, they also determine a set of contact orders for $\widetilde X$.
This set of contact orders along with a curve class $\widetilde A\in
H_2(\widetilde X)$ determines a class $\tilde\beta$ of logarithmic
map to $\widetilde X$. We have:

\begin{lemma}
\label{lem:subdivisionind}
Let $(X,D)$ be a Looijenga pair and $\beta$ a class of logarithmic
map to $(X,D)$ as above, with underlying curve class $A$.
Let $\pi:\widetilde X\rightarrow X$ be a log \'etale birational morphism with
$\widetilde X$ non-singular. If $\scrM(X,\beta)$ is non-empty, then
there exists a unique curve class $\widetilde A\in H_2(\widetilde X)$
with $\pi_*\widetilde A = A$
such that $\scrM(\widetilde X,\tilde\beta)$ is non-empty, where
$\tilde\beta$ is as above.
Further, with this choice of curve class, $N_{\beta}=N_{\tilde\beta}$.
\end{lemma}

\begin{proof}
We recall from \cite[Eq.~(1)]{AW} that there is a Cartesian diagram
(in all categories) with strict vertical arrows
\[
\xymatrix@C=30pt
{
\scrM(\widetilde X)\ar[r]^{\scrM(\pi)}\ar[d]&\scrM(X)\ar[d]\\
\foM(\widetilde\cX\rightarrow\cX)\ar[r]_>>>>>{\foM(\pi)}&\foM(\cX)
}
\]
(where the moduli spaces have 
no restriction on type). Further, \cite[Prop.~5.2.1]{AW}
shows that $\foM(\pi)$ is of pure degree $1$, while 
\cite[Lem.~4.2]{HW} shows that $\foM(\pi)$ is proper. Taken together,
this implies $\foM(\pi)$ is surjective, and hence
the union of connected components of $\scrM(\widetilde X)$ lying over
$\scrM(X,\beta)$ is non-empty. Given a class $\tilde\beta$ of
such a stable log map lying over $\scrM(X,\beta)$,
the contact orders of $\tilde\beta$ are determined by the
contact orders of $\beta$, and hence the only unknown is
the curve class $\widetilde A\in H_2(\widetilde X)$, which
necessarily satisfies $\pi_*\widetilde A= A$ by construction
of the map $\scrM(\pi)$. On the other hand, for any irreducible
component $\widetilde D_{\rho}$ of $\widetilde D$, the intersection
number $\widetilde D_{\rho}\cdot \widetilde A$ is completely determined
by the contact orders of $\tilde\beta$, by Corollary~\ref{cor:114}.
Since the intersection matrix of the set of exceptional curves of
$\pi$ is negative definite, and any two choices of lifts of $A$
differ by a linear combination of exceptional curves, it follows
that $\widetilde A$ is uniquely determined.

Identifying $\widetilde D_i$ with $D_i$ using $\pi$, we now have
a commutative diagram
\[
\xymatrix@C=30pt
{
\scrM(\widetilde X, \tilde\beta)\ar[rd]^{\widetilde{\ul{\ev}'}}\ar[d]_{\scrM(\pi)}&\\
\scrM(X,\beta)\ar[r]_{\ul{\ev}}&\prod_{k=1}^q D_{i_k}
}
\]
By \cite[Thm.~1.1.1]{AW}, 
$\scrM(\pi)_*[\scrM(\widetilde X,\tilde\beta)]^{\virt}
=[\scrM(X,\beta)]^{\virt}$. The result follows. 
\end{proof}

Alternatively, we may define $N_{\beta}$ as follows. Choose
points $x_k\in D_{i_k}$, $1\le k\le q$, and write $\mathbf{x}:=(x_1,\ldots,
x_q)$. Define
\begin{equation}
\label{eq:point constraint}
\scrM(X,\beta,\mathbf{x}):=\scrM(X,\beta)\times_{\prod_k D_{i_k}} \mathbf{x}.
\end{equation}

\begin{proposition}
\label{prop:alt def}
Let $(X,D)$ be a Looijenga pair, $\beta$ as above. Then
$\scrM(X,\beta,\mathbf{x})$ carries a virtual fundamental class
of virtual dimension zero, and $\deg [\scrM(X,\beta,\mathbf{x})]^{\virt}
=N_{\beta}$.
\end{proposition}

\begin{proof}
Let $D_i^{\circ}\subseteq D_i$ be the open subset obtained by deleting
the two double points of $D$ contained in $D_i$. We then have a diagram:
\[
\xymatrix@C=30pt
{
\scrM(X,\beta,\mathbf{x})\ar[r]\ar[d]_{\varepsilon_{\mathbf{x}}}&
\scrM(X,\beta)^{\circ}\ar[r]^{j''}\ar[d]_{\varepsilon^{\circ}}&
\scrM(X,\beta)\ar[d]^{\varepsilon}\\
\foM^{\ev}(\shX,\beta,\mathbf{x})\ar[d]_{\ul{\ev}_{\mathbf{x}}}\ar[r]_{\iota'}&
\foM^{\ev}(\shX,\beta)^{\circ}\ar[d]_{\ul{\ev}^{\circ}}\ar[r]_{j'}&
\foM^{\ev}(\shX,\beta)\ar[d]^{\ul{\ev}}\\
\mathbf{x}\ar[r]_{\iota}&\prod_{k=1}^q D_{i_k}^{\circ}\ar[r]_j&
\prod_{k=1}^q D_{i_k}
}
\]
Here the evaluation space $\foM^{\ev}(\shX,\beta)$ is given by
evaluation at the marked points with contact orders $u_1,\ldots,u_q$,
and all other stacks are defined by the requirement that all squares
are cartesian in the category of ordinary stacks.
The relative obstruction theory for $\varepsilon$ then pulls back to
give relative obstruction theories for $\varepsilon^{\circ}$ and
$\varepsilon_{\mathbf{x}}$.

By compatibility of flat pullback with proper pushfoward, and
compatibility of virtual pullback with flat pullback, we have
\begin{align*}
N_{\beta}\prod_k [D_{i_k}^{\circ}]=j^*\left(N_{\beta}\prod_k[D_{i_k}]\right)
= {} & j^*(\ul\ev\circ\varepsilon)_*[\scrM(X,\beta)]^{\virt}\\
= {} & (\ul{\ev}^{\circ}\circ\varepsilon^{\circ})_*
(j'')^*[\scrM(X,\beta)]^{\virt}\\
= {} & (\ul{\ev}^{\circ}\circ\varepsilon^{\circ})_*
[\scrM(X,\beta)^{\circ}]^{\virt}.
\end{align*}
Further, $\ul\ev^{\circ}$ is flat, being a base-change of an evaluation
map $\foM(\shX,\beta)^{\circ}\rightarrow \prod_k [D_{i_k}^{\circ}/\GG_m^2]$.
Hence, the Gysin maps $\iota^!$ and $(\iota')^!$ agree, and by compatibility
of Gysin maps with virtual pullback, we have
\begin{align*}
\deg [\scrM(X,\beta,\mathbf{x})]^{\virt} = 
\deg (\ul{\ev}_{\mathbf x}\circ\varepsilon_{\mathbf{x}})_*
[\scrM(X,\beta,\mathbf{x})]^{\virt} = {} & 
\deg (\ul{\ev}_{\mathbf x}\circ\varepsilon_{\mathbf{x}})_*
\iota^![\scrM(X,\beta)^{\circ}]^{\virt}\\
= {} & \deg \iota^! (\ul{\ev}^\circ\circ\varepsilon^{\circ})_*
[\scrM(X,\beta)^{\circ}]^{\virt}\\
= {} & N_{\beta},
\end{align*}
as desired.
\end{proof}

While we will only need the following result in \cite{K3}, its method
of proof will be used in \S\ref{subsec:blowups}.
The argument is now standard, going back to \cite{GPS}
and \cite{Bousseau}.

\begin{lemma}
\label{lem:vanishing}
Let $(X,D)$ be a Looijenga pair, $\beta$ as above. Suppose
given a stable log map $f:C\rightarrow X$ in $\scrM(X,\beta)$ lying
in $\ul\ev^{-1}(x_1,\ldots,x_q)$ where none of the $x_i$ are double
points of $D$. Then $f(C)\cap D$ is a finite set.
\end{lemma}

\begin{proof}
Denote by $p_1,\ldots,p_q,p_{\out}
\in C$ the marked points with contact orders $u_1,\ldots,u_q,u_{\out}$
respectively, so that
$f(p_k)=x_k$. However, we do not know the value of $f(p_{\out})$.

Let $f$ be of tropical type $\tau=(G,\bsigma,\mathbf{u})$, and 
fix $s\in\Int(\tau)$,
yielding a tropical map $h_s:G\rightarrow B$. We write the legs of $G$
corresponding to the marked points as
$L_1,\ldots,L_q,L_{\out}$. Of course $G$ may have many vertices, but
$h_s$ is balanced at those vertices
which don't map to the origin of $B$. Also, as
by assumption the marked points $p_1,\ldots,p_q$ of $C$ do not map
to a double point of $D$, the image under $h_s$ of each leg $L_k$
of $G$ lies in a ray of $\P$. Now note that if $f(C)\cap D$ is not
finite, then $f(C)$ necessarily contains a double point of $D$, and then
the image of $h_s$ must intersect the interior of a two-dimensional cone
of $\P$. Thus it is now enough to
show that the image of $h_s$ is contained in the one-skeleton of
$\P$.

Suppose this is not the case.
First suppose there is a $v\in V(G)$ such that
$h_s(v)\in\Int(\sigma)$ for a two-dimensional cone $\sigma\in\P$.
As such a vertex corresponds to a contracted irreducible component
of $C$, by stability, there must be at least three edges or legs
adjacent to $v$. Further, the only possible leg adjacent to $v$ is
$L_{\out}$. However, necessarily $h_s(L_{\out})$ is parallel to the
ray $\bsigma(L_{\out})$ of $\P$,
and in particular, is not contained in the ray $\RR_{\ge 0}h_s(v)$.
Then it follows easily from the balancing condition that one of the
two possibilities occur: (1)
There are at least two edges or legs $E_1,E_2$ adjcaent
to $v$ with $h_s(E_i)$ not contained in the ray $\RR_{\ge 0}h_s(v)$,
with $h_s(E_1)$ mapping into one side of this ray and $h_s(E_2)$ mapping
into the other. (2) No leg is adjacent to $v$ and all edges $E_i$ 
adjacent to $v$ satisfy
$h_s(E_i)\subseteq\RR_{\ge 0}h_s(v)$. In the latter case, again by balancing,
there must be another vertex $v'$ with $h_s(v')\in \RR_{>0}h_s(v)$ with
$h_s(v)$ lying between $0$ and $h_s(v')$. Repeating this argument,
eventually we will come to a vertex where case (1) holds.

Thus we may assume case (1) holds at the vertex $v$.
Choose one of these two edges which is not a leg,
say $E_1$. It has another
vertex $v'$. Note that $h_s(v')$ is not the origin in $B$, so the
balancing condition holds for $h$ at $v'$. Thus, again there must be
another edge or leg adjacent to $v'$ which maps to the other side of the
ray $\RR_{\ge 0} h_s(v')$. Continuing in this fashion, there
are three possibilities: (1) We
arrive at $L_{\out}$, in which case we made the wrong
initial choice of edge $E_i$. Choose the other edge instead.
(2) We get
an infinite sequence of edges, of course impossible. (3) We obtain a loop,
in which case the genus of $C$ is positive. But $C$ is assumed to be
genus zero. Thus we arrive at a contradiction.

If there are no vertices of $G$ mapping into the interior of a two-dimensional
cone,
but there is an edge $E$ of $G$ with $h_s(E)$ intersecting the interior
of a two-dimensional cone, then the endpoints of $E$ must map to
different rays of $\P$. We can then repeat the same argument
as above, starting at either endpoint of $E$.

So in either case we obtain a contradiction, showing the result.
\end{proof}

We observe that in the case $(X,D)$ is a toric pair, the invariant
$N_{\beta}$ has already been encountered in \cite{GPS}. In \cite[\S3]{GPS},
a number $N^{\mathrm{hol}}_{\mathbf{m}}(\mathbf{w})$ is defined.
Here, we start with a toric surface $X$ with fan $\Sigma$ in a two-dimensional
vector space $N_{\RR}$; of course, the tropicalization of the toric
pair $(X,D)$ is $(N_{\RR},\Sigma)$. Let $m_1,\ldots,m_n$ be primitive
generators of distinct rays $\rho_1,\ldots,\rho_n$ in $\Sigma$. Write
$\mathbf{m}=(m_1,\ldots,m_n)$. Further, write $\mathbf{w}=(\mathbf{w}_1,\ldots,
\mathbf{w}_n)$ with $\mathbf{w}_i = (w_{i1},\ldots,w_{in_i})$.
Suppose we may write 
\[
w_{\out}m_{\out}= -\sum_{i,j} w_{ij} m_i
\]
with $m_{\out}$ a primitive generator of a ray $\rho_{\out}$ of $\Sigma$.
Let $D_1,\ldots,D_n,D_{\out}$ be the toric divisors of $X$ corresponding
to $\rho_1,\ldots,\rho_n,\rho_{\out}$. Then it is standard (see e.g., 
\cite[Lem.~1.13]{GHK}) that this data determines a unique curve
class $A$ such that for each prime toric divisor $D'$, we have
\begin{equation}
\label{eq:int nums}
A\cdot D' = \begin{cases} \sum_j w_{ij} & \hbox{$D'=D_i$ and $D'\not =D_{\out}$;}\\
w_{\out}+\sum_j w_{ij} & D'=D_i = D_{\out};\\
w_{\out} & \hbox{$D'=D_{\out}$ and $D'\not=D_i$ for any $i$;}\\
0 & \hbox{otherwise}.
\end{cases}
\end{equation}
Thus the above data determines a class of log curve $\beta$ 
with underlying curve class $A$ and marked points with set of contact
orders $\{w_{ij}m_{ij}\,|\,1\le i \le n, 1\le j\le n_i\}
\cup\{w_{\out}m_{\out}\}$.

\begin{lemma} 
\label{lem:Nbeta GPS compare}
In the above situation, $N_{\beta} = N^{\mathrm{hol}}_{\mathbf{m}}(\mathbf{w})$.
\end{lemma}

\begin{proof}
This follows from the description of $N_{\beta}$ from Proposition
\ref{prop:alt def}, the definition of $N_{\mathbf{m}}^{\mathrm{rel}}
(\mathbf{w})$
of \cite[(4.7)]{GPS} and \cite[Thm.~4.4]{GPS}, as well as the
comparison between logarithmic and relative moduli spaces of \cite{AMW}.
In more modern language, one proceeds as follows. By adapting the argument of
\cite[Prop.~4.3]{GPS} to stable log maps, one shows
that in fact for general
choice of $\mathbf{x}=(x_{ij})_{1\le i \le n, 1\le j\le n_i}$, 
the moduli space $\scrM(X,\beta,\mathbf{x})$ is particularly
simple. It consists of a finite
number of reduced points, each corresponding to a log map from 
an irreducible domain curve $C$ with $f^{-1}(D)$ consisting only
of the marked points on $C$. The number $N^{\mathrm{hol}}_{\mathbf{m}}
(\mathbf{w})$ is defined precisely to be the count of such maps.
\end{proof}

We give one further result required for \cite{K3}. Suppose given an
inclusion $\widetilde N\subseteq N$ of rank two lattices and a fan
$\Sigma$ in $N$ determining a toric surface $X$. This also determines a fan 
$\widetilde\Sigma$ for $\widetilde N$, with corresponding
toric surface $\widetilde X$. The inclusion $\widetilde N\subseteq N$ 
determines a covering $p:\widetilde X\rightarrow X$ of degree 
$\mu$ the lattice index $[\widetilde N:N]$.
Suppose given a class $\tilde\beta$ 
of log curve with underlying curve class $\widetilde A$ for the toric pair 
$(\widetilde X,\widetilde D)$
with contact orders $\tilde u_1,\ldots, \tilde u_n, \tilde u_{\out}\in 
\widetilde N$ contained
in rays of $\widetilde\Sigma$, with $\tilde u_i=\tilde w_i\tilde \nu_i$
for $\tilde\nu_i$ primitive. Suppose $\tilde u_{\out}+\sum_i \tilde u_i=0$, 
and $\widetilde A$ is the
unique curve class determined by this data as in \eqref{eq:int nums}.
Then via the inclusion $\widetilde N\subseteq N$, the contact orders
$\tilde u_i$, $\tilde u_{\out}$ may be viewed as contact orders $u_i$ for
the toric pair $(X,D)$. (Note that under the
inclusion $\widetilde N\subseteq N$, we have $\tilde u_i=u_i$.)
This determines similarly
a class $\beta$ of log curve for the pair $(X,D)$. 
Write $u_i =  w_i\nu_i$ as before, with $\nu_i$
primitive in $N$. Of course, $\tilde w_i | w_i$.

\begin{lemma}
In the above situation,
\[
\mu N_{\beta} = N_{\tilde\beta}\prod_{i=1}^n {\mu \tilde w_i\over w_i}.
\]
\end{lemma}

\begin{proof}
One first notes via elementary toric geometry that the degree of
$p_i=p|_{\widetilde D_i}:\widetilde D_i\rightarrow D_i$ is the lattice index
$[\widetilde N/\tilde \nu_i\ZZ:  N/\nu_i\ZZ]$, which is 
$\mu \tilde w_i/w_i$.
Further, in a neighbourhood of a point of $\widetilde D_i$, the map $p$ has ramification
of degree $w_i/\tilde w_i$ along $\widetilde D_i$.

Composing stable log maps to $\widetilde X$ with $p$ gives a map
$q:\scrM(\widetilde X,\tilde\beta)\rightarrow \scrM(X,\beta)$. Note that
since the targets are toric pairs, and the log tangent bundle of a toric
variety is trivial, these moduli spaces are unobstructed. Hence their
virtual fundamental class and fundamental class coincide. Further, 
the degree of the map $q$ is $\mu$. Indeed, choose a general closed
point in $\scrM(X,\beta)$, corresponding to a stable
log map $f:C\rightarrow X$. Because it is general, then
as in the proof of Lemma \ref{lem:Nbeta GPS compare}, $C$ is an
irreducible curve with $f^{-1}(D)$ consisting only of
the marked points of $C$. Then 
$C':=C\times_X \widetilde X$ is
a (in general) non-normal curve. In a local coordinate $t$ near the marked point
$p_i$ with contact order $u_i$ and local coordinates near $f(p_i)$, 
$\widetilde X\rightarrow X$ takes the form $(x,y)\mapsto (x^{w_i/\tilde w_i},y)$
(with $\widetilde D_i$ given by $x=0$) and $f$ is given by 
$t \mapsto (t^{w_i}
\varphi(t), \psi(t))$ , where $\varphi(t)$ is invertible and $\psi(t)$
is arbitrary. Then the local equation for $C'$ in $\AA^2$ with coordinates
$x,t$ is
\[
x^{w_i/\tilde w_i}-t^{w_i}\varphi(t)=0.
\]
Note that \'etale locally we may find a function $\bar\varphi(t)$ with
with $\bar\varphi^{w_i/\tilde w_i} = \varphi$, so that the left-hand side
of the above equation factors as $\prod_{i=1}^{w_i/\tilde w_i}
(x-\zeta^i t^{\tilde w_i}\bar\varphi(t))$, with $\zeta$ a primitive
root of unity. Thus, after normalizing $C'$ to obtain a curve $\widetilde C$, 
we see that $\widetilde C\rightarrow C$ is an \'etale map, necessarily
of degree $\mu$. Since $C\cong \PP^1$, $\widetilde C$ must split into
$\mu$ connected components, each isomorphic to $\PP^1$. In addition,
the composition $\widetilde C\rightarrow C'\rightarrow \widetilde X$
of the normalization with the projection to $\widetilde X$ then induces a stable
map on each irreducible component. From the equations above, one sees
that each such stable map has contact order $\tilde w_i$ with $\widetilde D_i$ 
and hence gives an element of the fibre of $q$. Conversely, by the universal
property of the fibre product, any map in the fibre of $q$ must arise
in this way. Thus the degree of $q$ is $\mu$.

We now consider the commutative diagram
\[
\xymatrix@C=30pt
{
\scrM(\widetilde X,\tilde\beta)\ar[r]^q \ar[d]_{\widetilde{\ul{\ev}}}&\scrM(X,\beta)
\ar[d]^{\ul{\ev}}\\
\prod_{i=1}^n \widetilde D_i \ar[r]_{\prod p_i} &\prod_{i=1}^n D_i
}
\]
We then have
\[
\left(\prod p_i\right)_* \widetilde{\ul{\ev}}_* [\scrM(\widetilde X,
\tilde\beta)] =N_{\tilde\beta} \deg\left(\prod p_i\right) \prod [D_i]
\]
while
\[
\ul\ev_* q_* [\scrM(\widetilde X,\tilde\beta)]
= \mu N_{\beta} \prod [D_i].
\]
Since $\deg\left(\prod p_i\right) = \prod (\mu \tilde w_i/w_i)$,
the result follows.
\end{proof}

\subsection{Gluing via Parker and Ranganathan's triple point method}
We now return to the situation of Assumptions \ref{ass:K3}.
To state the main result of this section, 
fix a decorated wall type $\btau=(G,\bsigma,
\mathbf{u},\mathbf{A})$,
and let $v_{\out}\in V(G)$ be the vertex adjacent to $L_{\out}$. 
Let $h:\Gamma(G,\ell)\rightarrow\Sigma(X)$ be the universal family
of tropical maps of type $\btau$.
Let $E_1,\ldots,E_q\in E(G)$ be the edges adjacent to $v_{\out}$
(with $q\ge 0$). If we split $G$ at the edges $E_1,\ldots,E_q$,
we obtain connected components $G_1,\ldots,G_q$ and $G_{\out}$,
where $v_{\out}$ is the unique vertex of $G_{\out}$ and
$E_i$ becomes the unique leg of $G_i$, which we will denote as
$L_{i,\out}$. This in turn gives rise to decorated types
$\btau_1,\ldots,\btau_q,\btau_{v_{\out}}$.

\begin{lemma}
$\btau_i$ is a decorated wall type.
\end{lemma}

\begin{proof}
Recall we write $h':G\rightarrow B'$ for the unique tropical map
to $B'$ of type $\tau$. If $u_{i,\out}:=
\mathbf{u}(L_{i,\out})=0$,
then as this coincides with $\mathbf{u}(E_i)$, $h'$ contracts
the edge $E_i$. However, the length of the
edge $E_i$ is a parameter in the moduli space of tropical maps of type
$\tau$, and hence this edge length may vary freely.
This contradicts the rigidity of $h'$. Thus
Definition \ref{def:wall type},(1) holds. Item (2) of that definition
holds because it holds for $\btau$. 

For item (3) of Definition~\ref{def:wall type}, 
to show $\dim \tau_i=1$, it is sufficient to show that $h'_i:=
h'|_{G_i}$
is rigid. If it deformed as a tropical map to $B'$ in a way so that the image of
$L_{i,\out}$ also deforms to a line segment not passing through
$h'(v_{\out})$, then this violates \cite[Lem.~2.5,(1)]{Walls}.
Otherwise, if $h'_i$ deforms so that the image of the leg $L_{i,\out}$
continues to pass through $h'(v_{\out})$, then
this deformation can be glued to produce
a deformation of $h'$, again a contradiction. Thus $\dim\tau_i=1$
(remembering that $\tau_i$ parameterizes maps to $B$ rather than $B'$)
and, as $u_{i,\out}\not=0$, one sees that $\dim h_i(\tau_{i,\out})=2$.

Thus $\btau_i$ is a decorated wall type.
\end{proof}

\begin{construction}
\label{const:the construction}
Continuing with a fixed decorated wall type $\btau$,
our next goal is to associate an enumerative invariant of the form
defined in the previous subsection
to the vertex $v_{\out}$ of $G$. Let $x:=h'(v_{\out})\in B'$.
We will first construct: (1) a pair $(\overline B_x,\widetilde\Sigma_x)$
of integral affine surface with singularity $\overline B_x$
and a polyhedral cone decomposition $\widetilde\Sigma_x$;
(2) a Looijenga pair $(\widetilde X_x,
\widetilde D_x)$ which tropicalizes to $(\overline B_x,\widetilde\Sigma_x)$;
(3) a class $\btau_{\out}$ of logarithmic map to $(\widetilde X_x,
\widetilde D_x)$.

In what follows, let $\sigma'_x\in \P'$ be the minimal cell containing $x$,
$\sigma_x\in\P$ the minimal cone of $\P$ containing $x$. We write $X_x$
for the stratum $X_{\sigma_x}$ of $X$.

\medskip

{\bf Construction of a pair $(\overline B_x,\Sigma_x)$}.
First, we define $\overline B_x$ an affine manifold with singularities
along with a decomposition $\Sigma_x$
into not-necessarily strictly convex polyhedral cones. If $x\in B'_0$,
let $v_x\in\Lambda_{B,x}$ be a primitive integral tangent vector
to the ray $\RR_{\ge 0}x\subseteq B_0$.
Then $\overline B_x$ is identified with $\Lambda_{B,\RR,x}/\RR v_x$,
with lattice structure coming from $\Lambda_{B,x}/\ZZ v_x$.
In this case, we define 
\begin{equation}
\label{eq:sigmaxdef}
\Sigma_x:=\{(\sigma+\RR v_x)/\RR v_x\,|\,\sigma_x\subseteq\sigma
\in\P\}.
\end{equation}
Specifically, if $\dim\sigma'_x=2$,
then $\Sigma_x$ just consists of one cone, namely all of $\overline B_x$.
If $\dim\sigma'_x=1$, then $\Sigma_x$ consists of two half-spaces
and their common face, the image of the tangent space to $\sigma'_x$. 

If $x\in \Delta'$, i.e., $x$ is a vertex of $\P'$, we obtain
$(\overline B_x,\Sigma_x)$ as the tropicalization of the 
corresponding irreducible
component $X_x$ of $D$. Here $X_x$ carries the divisorial log structure coming
from the divisor $D_x:=\partial X_x$, as in \S\ref{sec:punctured versus relative}. 
Note that by the discussion in \S\ref{sec:punctured versus relative},
in this case, \eqref{eq:sigmaxdef} is still the correct description
of $\Sigma_x$. However, further, as in \S\ref{subsect:invariant LP},
$\overline B_x$ also carries an integral affine
structure with a singularity at the origin.

\medskip

{\bf Refinement $\widetilde\Sigma_x$ of $\Sigma_x$.}
Note that by construction and the assumption that $\btau$ is 
a wall-type, for any edge or leg $E$ adjacent to $v_{\out}
\in V(G)$, $h(\tau_E)$ is two-dimensional and
 $\rho_E:=(h(\tau_E)+\RR v_x)/\RR v_x$ is a ray in $\overline B_x$. 

We now let $\widetilde\Sigma_x$ be a refinement of $\Sigma_x$ chosen
so that (1) every cone in $\widetilde\Sigma_x$ is a strictly
convex rational polyhedral cone,
and (2) the rays 
$\rho_{E_i}$, $1\le i \le q$ and $\rho_{L_{\out}}$ are one-dimensional
cones of $\widetilde\Sigma_x$. In what follows, the precise choice of
$\widetilde\Sigma_x$ will be unimportant, see Lemma \ref{lem:subdivisionind}.

\medskip

{\bf The type $\tau_{\out}$.}
Note that the adjacent edges to $v_{\out}\in V(G)$ determine
a type $\tau_{\out}=(G_{\out},\bsigma_{\out},\mathbf{u}_{\out})$
of tropical map to
$(\overline B_x,\widetilde\Sigma_x)$, as follows.
First, $G_{\out}$ is the graph underlying $\btau_{v_{\out}}$ as previously
described, with only one vertex $v_{\out}$ and legs $E_1,\ldots,E_q,L_{\out}$.
We set $\bsigma_{\out}(v_{\out})=\{0\}
\in\widetilde\Sigma_x$, $\bsigma_{\out}(E_i)=\rho_{E_i}$ and 
$\bsigma_{\out}(L_{\out})
=\rho_{L_{\out}}$. Finally, the images of
$\mathbf{u}(E_i)$ (with the edges $E_i$ oriented away from $v_{\out}$)
and $\mathbf{u}(L_{\out})$ under the quotient map given by dividing out
by $\RR v_x$ yield $\mathbf{u}_{\out}(E_i),\mathbf{u}_{\out}(L_{\out})$.

\medskip

{\bf The pair $(\widetilde X_x,\widetilde D_x)$.}
If $\dim\sigma'_x=2$ or $1$, then we may interpret the pair
$(\overline B_x,\widetilde\Sigma_x)$ as a fan, hence defining a toric variety
$\tX_x$ with toric boundary $\tD_x$. In particular, $(\overline B_x,\widetilde
\Sigma_x)$ is the tropicalization of the pair $(\tX_x,\tD_x)$.
In both these cases, there is a morphism $\pi:\widetilde X_x \rightarrow
X_x$. Indeed, if $\dim\sigma'_x=2$, this is just a constant map to a point.
If $\dim\sigma'_x=1$, consider the quotient map $\overline B_x\rightarrow
\RR$ given by dividing out by the tangent space to $\sigma'_x$. 
In this case $\RR$ carries the fan $\Sigma_{\PP^1}$ defining $\PP^1$, 
and the quotient map induces a map of fans $\widetilde\Sigma_x \rightarrow
\Sigma_{\PP^1}$, hence defining a morphism $\pi:\widetilde X_x\rightarrow\PP^1$.
This $\PP^1$ can be identified with $X_x$. 

If, on the other hand, $\dim\sigma'_x=0$, then the refinement 
$\widetilde\Sigma_x$ of $\Sigma_x$ defines a toric blow-up
$\pi:\widetilde X_x\rightarrow X_x$.
We take $\tD_x$ to be the strict transform of $D_x$.

\medskip

{\bf Decorating $\tau_{\out}$.}
Finally, we define a curve class $\widetilde A_{\out}$ to turn $\tau_{\out}$
into a decorated type $\btau_{\out}$ for log maps to the pair
$(\tX_x,\tD_x)$. Note we already have a curve class 
$A_{\out}= \mathbf{A}(v_{\out}) \in H_2(X_x)$. Of course, if 
$\dim\sigma'_x=2$, this curve class is $0$, and if $\dim\sigma'_x=1$,
this curve class is some multiple of the class of $X_x$.

We first introduce some additional notation.
For each ray $\rho\in\widetilde\Sigma_x$, 
denote by $\widetilde D_{\rho}$ the corresponding irreducible component
of $\tD_x$. For each leg $L\in L(G_{\out})$, write $w_L$
for the index of $\mathbf{u}_{\out}(L)$, i.e., $\mathbf{u}_{\out}(L)$
is $w_L$ times a primitive tangent vector. This represents the order of
tangency imposed by the contact order $\mathbf{u}_{\out}(L)$ with the
divisor $\widetilde D_{\bsigma_{\out}(L)}$. 

\begin{lemma}
\label{lem:curve lifting}
There is at most one curve class $\widetilde A_{\out}\in H_2(\widetilde X_x)$
with $\pi_*\widetilde A_{\out}=A_{\out}$
and for $\rho$ any ray in $\widetilde\Sigma_x$,  
\begin{equation}
\label{eq:intersection numbers}
\widetilde A_{\out}\cdot \widetilde D_\rho = \sum_{L\in L(G_{\out})\atop
\bsigma_{\out}(L)=\rho} w_L.
\end{equation}
If $W_{\btau}\not=0$, then such an $\widetilde A_{\out}$ exists.
\end{lemma}

\begin{proof}
If $\dim\sigma'_x=2$ or $1$, the unique existence of a class
$\widetilde A_{\out}$ satisfying \eqref{eq:intersection numbers}
is easy. Indeed, there is a unique
tropical map of type $\tau_{\out}$, and by the balancing condition
at $v_{\out}$, this defines a balanced
tropical map to $\overline B_x$ with all legs 
mapping to rays of $\widetilde\Sigma_x$.
As is standard, this defines a curve class $\widetilde 
A_{\out}\in H_2(\widetilde X_x)$,
see for example \cite[Lem.~1.13]{GHK}. Furthermore, it is characterized
precisely by the intersection numbers with the boundary divisors as
in \eqref{eq:intersection numbers}.

If $\dim\sigma_x'=2$, then $\pi_*\widetilde A_{\out}=A_{\out}$ is trivial
as $A_{\out}=0$, and thus the last statement on the existence of
$\widetilde A_{\out}$ is vacuous.

If $\dim\sigma_x'=1$, then we need to verify that if
$W_{\btau}\not=0$, then $\pi_*\widetilde A_{\out}=A_{\out}$. 
Let $\btau_{v_{\out}}$ be the type of punctured map to $X/S$ corresponding
to the vertex $v_{\out}$ after splitting $\btau$ at the edges $E_1,
\ldots,E_q$.
Necessarily, if $W_{\btau}\not=0$, then the moduli space
$\scrM(X/S,\btau_{v_{\out}})$ is non-empty. The requirement that
this moduli space be non-empty then allows us to determine 
$A_{\out}$, using Corollary~\ref{cor:114}. In particular,
$A_{\out}=d[X_x]$ for some $d\ge 0$, and $d$ can be
determined by intersecting $A_{\out}$ with an irreducible
component of $D$ transverse to $X_v$. This is calculated as follows.
Let $\sigma_1',\sigma_2'\in\P'$ be the two two-cells containing $\sigma'_x$,
with additional vertices $v_1,v_2$ respectively not contained in 
$\sigma'_x$. Also write $\sigma_1,\sigma_2\in\P$ for the corresponding
cones in $\P$. Then the corresponding irreducible components $X_{v_i}$
of $D$ each meet $X_x$ transversally in one point.

By Corollary~\ref{cor:114}, we may now calculate $d=X_{v_1}\cdot 
\mathbf{A}(v_{\out})$ as follows. Let $\delta:\Lambda_{B,x}\rightarrow
\ZZ$ be the quotient map by the tangent space to $\sigma_x\in\P$, with
sign chosen so that elements of $\Lambda_{B,x}$ pointing into $\sigma_1$
map to positive integers. Then
\[
d=\sum_{E} \delta(\mathbf{u}(E)),
\]
where $E$ runs over edges and legs adjacent to $v_{\out}$ with $\bsigma(E)=
\sigma_1$. On the other hand, a simple toric argument shows that
given our definition of $\widetilde A_{\out}$, $\pi_* \widetilde 
A_{\out}= d[X_x]$ for
the same choice of $d$. This completes the argument in the $\dim\sigma'_x=1$
case.

Finally, consider the case $\dim\sigma'_x=0$.
For the uniqueness statement, we argue as in the proof of Lemma 
\ref{lem:subdivisionind}:
the intersection matrix of the exceptional locus is
negative definite. In particular, any two lifts $\widetilde A_{\out}$
of $A_{\out}$ differ by a linear combination of exceptional divisors.
The uniqueness of the lift with the given intersection numbers with the
boundary divisors then follows.

For the second statement, suppose $W_{\btau}\not=0$.
Let $\btau_{v_{\out}}$ be as in the $\dim\sigma'_x=1$ case,
so that if $W_{\btau}\not=0$, then the moduli space
$\scrM(X/S,\btau_{v_{\out}})$ is non-empty. By Theorem
\ref{thm:punctured to relative}, this moduli space is isomorphic to 
$\scrM(X_x,\bar\btau_{v_{\out}})$, where $\bar\btau_{v_{\out}}$
is the type of tropical map to $(\overline B_x,\Sigma_x)$
constructed from $\btau_{v_{\out}}$ as in 
\S\ref{sec:punctured versus relative}. On the other hand,
the type $\tau_{\out}$ is a type of tropical map to 
$(\overline B_x,\widetilde\Sigma_x)$ with the same underlying graph
and contact orders as $\bar\btau_{v_{\out}}$. 
It then follows as in the proof of Lemma 
\ref{lem:subdivisionind} that there must be at a curve class
$\widetilde A_{\out}$ with $\pi_*\widetilde A_{\out}=A_{\out}$,
determining a type $\btau_\out$ with $\scrM(\widetilde X_x,\btau_{\out})$
non-empty. However, necessarily $\widetilde A_{\out}$ satisfies
\eqref{eq:intersection numbers} by Corollary~\ref{cor:114}.
\end{proof}

{\bf The invariant $N_{\btau_{\out}}$.} We have now
constructed a pair $(\widetilde X_x,\widetilde D_x)$, a curve class
$\widetilde A_{\out}\in H_2(\widetilde X_x)$, and hence a class
$\btau_{\out}$ of log map to $\widetilde X_x$. Thus we obtain an
invariant $N_{\btau_{\out}}$ as defined in \eqref{eq:Nbeta def},
independent of the choice of $\widetilde\Sigma_x$ by Lemma
\ref{lem:subdivisionind}.
\end{construction}

The application of our gluing formalism is then:

\begin{theorem}
\label{thm:inductive structure}
We have
\begin{equation}
\label{eq:main result}
k_{\tau}W_{\btau} = {w_{L_{\out}}N_{\btau_{\out}} \prod_{i=1}^q k_{\tau_i} W_{\btau_i}
\over |\Aut(\btau_1,\ldots,\btau_q)|},
\end{equation}
where the automorphism group of the denominator is the set of permutations 
$\sigma$ of $\{1,\ldots,q\}$ such that $\btau_{i}$ is isomorphic to
$\btau_{\sigma(i)}$ as decorated types.
\end{theorem}

\begin{proof}
We have a standard gluing situation obtained by splitting $\btau$
at the edges $E_1,\ldots,E_q$. In general this gluing
situation will not be tropically transverse. However, in this case
a relatively mild birational \'etale modification of $X$ takes care of this.
Flatness of the map $\ul{\ev}$ in Theorem \ref{thm:gluing factorization} 
also is an
issue, but this is dealt with via Parker \cite{Parker} and Ranganathan's 
approach \cite{Ra} to this situation.

\medskip

{\bf Step I.} \emph{Refining $\Sigma(X)$.} Note that a refinement of
the polyhedral decomposition $\P'$ of $B'$ into rational convex 
polyhedra gives a refinement of $(B,\P)$, i.e., of $\Sigma(X)$, and hence
a log \'etale modification $\pi:\widetilde X\rightarrow X$.
In particular, we choose
a refinement $\widetilde\P'$
of $\P'$ with the following properties: (1) $x$ is a vertex of $\widetilde
\P'$; (2) the integral affine manifold with singularity 
$(\overline B_x, \Sigma_x)$ coincides
with $(\overline B_x,\widetilde\Sigma_x)$ in the notation
of Construction \ref{const:the construction}. It is not difficult 
to see (see e.g., the argument of \cite[Lem.~4.3]{AK})
that this
can be done as a series of toric blow-ups, and hence $\widetilde X\rightarrow
X$ is projective. We omit the details. However, we note that
we do not require $\widetilde X\rightarrow S$ to be normal crossings,
but now only log smooth.

\medskip

{\bf Step II.} \emph{Lifting the type $\btau$.} Having 
chosen the log \'etale modification $\widetilde X\rightarrow X$,
we use \cite{SJ} to choose a \emph{lift} of $\tau$ to $\Sigma(\widetilde
X)$. In general, there may be many choices of lift of a type, but
in the rigid case, the description of lifts given in \cite[\S4]{SJ}
reduces to a quite simple procedure. Given the map $h':G\rightarrow B'$,
we first take the minimal refinement $\widetilde G$ 
of the graph $G$ (i.e., subdivide edges
or legs via the addition of vertices) with the property that for any
edge or leg $E$ of $\widetilde G$, $h'(E)$ is contained in an element
of $\widetilde\P'$. There remains, however, some ambiguity as to
the treatment of $L_{\out}$, as it might have been subdivided into
a number of edges and one leg. We choose to discard all but the edge (or
leg) adjacent to $v_{\out}$, so that it is now a leg, which we take
to be the unique leg of $\widetilde G$. This provides a type $\tilde\tau=
(\widetilde G,\tilde\bsigma,\tilde\bu)$,
a lifting of $\tau$ in the terminology of \cite[\S4]{SJ}.

Now it is shown in the proof of \cite[Cor.~9.4]{SJ} that 
\begin{equation}
\label{eq:sams sum}
k_{\tau}W_{\btau} = k_{\tilde\tau}\sum_{\tilde\btau} W_{\tilde\btau},
\end{equation}
where the sum is over all decorations $\tilde\btau$ of $\tilde\tau$ with
$\pi_*\widetilde{\mathbf{A}}=\mathbf{A}$. What this means is that
if $v\in V(\widetilde G)$ is contained in the interior of an edge or
leg $E$ of $G$, then $\pi_*:H_2(\widetilde X_{\tilde\bsigma(v)})\rightarrow
H_2(X_{\bsigma(E)})$ satisfies $\pi_*\widetilde{\mathbf{A}}(v)=0$.
If on the other hand $v\in V(\widetilde G)$ is a vertex which is
also a vertex of $G$, then $\pi_*:H_2(\widetilde X_{\tilde\bsigma(v)})
\rightarrow H_2(X_{\bsigma(v)})$ satisfies 
$\pi_*\widetilde{\mathbf{A}}(v)=\mathbf{A}(v)$.

We now note that for a given type $\tilde\btau$ appearing in
\eqref{eq:sams sum}, $W_{\tilde\btau}=0$ unless
$\widetilde{\mathbf{A}}(v_{\out})$ coincides with the curve class
$\widetilde A_{\out}$ constructed from $\btau$ in 
Lemma \ref{lem:curve lifting}. Thus we may assume this equality
and the condition of Lemma \ref{lem:curve lifting} in the sequel.

The type $\tilde\btau$ now gives rise to the invariant 
$N_{\tilde\btau_{\out}}$, and
$N_{\btau_{\out}}=N_{\tilde\btau_{\out}}$ by Lemma \ref{lem:subdivisionind}.
We also have formula \eqref{eq:sams sum} for $W_{\btau_i}$ using the 
lift $\tilde\tau_i$ of $\tau_i$ induced by restricting the lift $\tilde\tau$
to $G_i$. Here we use what \cite{SJ} calls a maximally extended lift,
in that we don't remove any segments from the subdivided leg $L_{i,\out}$. 
We now observe it is sufficient to prove \eqref{eq:main result} after replacing
$\btau$ with $\tilde\btau$ and $\btau_i$ with $\tilde\btau_i$. Indeed,
if \eqref{eq:main result} has been proved in this case, then
the left-hand-side of \eqref{eq:main result} can be expanded
using \eqref{eq:sams sum}, giving
\[
k_{\tau}W_{\btau} = \sum_{\tilde\btau}
w_{L_\out} N_{\btau_{\out}} {\prod_{i=1}^q k_{\tilde\tau_i} W_{\tilde\btau_i}
\over |\Aut(\tilde\btau_1,\ldots,\tilde\btau_q)|}.
\]
Here $\tilde\btau_1,\ldots,\tilde\btau_q$ are the decorated types
induced by $\tilde\btau$. If instead, one sums over all choices of
decorations of the lifts $\tilde\tau_1,\ldots,\tilde\tau_q$, the
same choice of $\tilde\btau$ will occur
\[
|\Aut(\btau_1,\ldots,\btau_q)|/|\Aut(\tilde\btau_1,\ldots,\tilde\btau_q)|
\]
times. Thus we obtain
\begin{align*}
k_{\tau}W_{\btau} = {} & \sum_{\tilde\btau_1,\ldots,\tilde\btau_q}
w_{L_{\out}} N_{\btau_{\out}} \prod_{i=1}^q {k_{\tilde\tau_i}W_{\tilde\btau_i}
\over |\Aut(\btau_1,\ldots,\btau_q)|}\\
= {} & {w_{L_{\out}} N_{\btau_{\out}}\over |\Aut(\btau_1,\ldots,\btau_q)|}
\prod_{i=1}^q\sum_{\tilde\btau_i} k_{\tilde\tau_i}W_{\tilde\btau_i},
\end{align*}
which gives \eqref{eq:main result} using \eqref{eq:sams sum} again.

Thus, replacing $X$ with $\widetilde X$, we may assume that $x$
is a vertex of $\P'$ and that every edge or leg of $G$ adjacent to $v_{\out}$
maps to an edge of $\P'$ under $h'$. 

\medskip

{\bf Step III}. \emph{The gluing situation}. We now are assuming
$\btau$ satisfies $x=h'(v_{\out})$ is a vertex of $\P'$ and
every edge or leg adjacent to $v_{\out}$ is mapped to an edge of $\P'$ under
$h'$. Let $\btau_1,\ldots,\btau_q,\btau_{v_{\out}}$ be as usual,
and denote by $\tilde\btau_i$ the decorated type obtained by gluing
$\btau_1,\ldots,\btau_i$ and $\btau_{v_{\out}}$. Alternatively,
after splitting $\btau$ at the edges $E_{i+1},\ldots,E_q$, 
$\tilde\btau_i$ is the decorated type corresponding to the connected
component containing $v_{\out}$. In particular, $\tilde\btau_0=\btau_{v_{\out}}$
and $\tilde\btau_q=\btau$. For each edge $E_j$, write
$D_j:=\ul{X}_{\bsigma(E_j)}$.

Let $\ul{\ev}_i:\scrM(X,\tilde\btau_i)\rightarrow \prod_{j=i+1}^q 
D_j$
be the schematic evaluation map at the punctured points
corresponding to the legs $E_{i+1},\ldots,E_q$ of $\tilde\btau_i$.
Define $N_i$ to be the rational number such that
\[
\ul{\ev}_{i,*}[\scrM(X,\tilde\btau_i)]^{\virt} = N_i \prod_{j=i+1}^q
[D_j].
\]
Thus (using Theorem \ref{thm:punctured to relative}) $N_0=N_{\btau_{\out}}$
and $N_q=\deg [\scrM(X,\btau)]^{\virt}
=W_{\btau}|\Aut(\btau)|$. We note a simple calculation shows that the 
virtual dimension of $\scrM(X,\tilde\btau_i)$ is $q-i$, but this also
follows from the computations below.

We will now inductively determine $N_i$ from $N_{i-1}$.
We consider the gluing situation given by $\tilde\btau_i$
with set of splitting edges $\mathbf{E}=\{E_i\}$. After
splitting, we obtain the types $\tilde\btau_{i-1}$ and $\btau_i$.

\medskip

{\bf Step IV.} \emph{Calculating the tropical multiplicity.}
We calculate $\mu(\tilde\tau_i,\mathbf{E})$.
Let $Q_i$ be the basic monoid for the type $\tau_i$ and 
$\widetilde Q_i$
the basic monoid for the type $\tilde\tau_i$. 
Note that by rigidity of the types $\tau_i$ we have $Q_i^*=\ZZ$ for
all $i$. Similarly, the types $\tilde\tau_i$ are rigid, so
$\widetilde Q_i^*=\ZZ$. Further, $N_{\tau_{i,\out}}$ can be identified with
$Q_i^*\oplus\ZZ$ in such a way so that the map $Q_i^*\oplus \ZZ
\rightarrow P_{E_i}^*$ given by $(q_i,\ell_i)\mapsto \ev_{v_i}(q_i)+\ell_i
\mathbf{u}(L_{i,\out})$ coincides with $h_{\tau_i,*}:N_{\tau_{i,\out}}
\rightarrow N_{\bsigma_i(L_{i,\out})}$. We then obtain a commutative
diagram
\[
\xymatrix@C=30pt
{
&&0\ar[d]&0\ar[d]&\\
&0\ar[r]\ar[d]&\widetilde Q_i^*\ar[r]\ar[d]&\widetilde Q_{i-1}^*\ar[d]^=\\
0\ar[r]&N_{\tau_{i,\out}}\ar[d]^{h_{\tau_i,*}}\ar[r]&
\widetilde Q_{i-1}^*\times N_{\tau_{i,\out}}\ar[d]^{\Psi_i}\ar[r]&\widetilde
Q_{i-1}^*
\ar[d]\ar[r]&0\\
0\ar[r]&  P_{E_i}^*\ar[d]\ar[r]_=
& P_{E_i}^*\ar[d]\ar[r] & 0 &\\
& \coker h_{\tau_i,*}\ar[r]\ar[d]&\coker\Psi_i\ar[d]&&\\
&0&0&&
}
\]
giving a long exact sequence
\[
0\longrightarrow \widetilde Q_i^*\longrightarrow\widetilde Q_{i-1}^*
\longrightarrow
\coker h_{\tau_i,*} \longrightarrow \coker\Psi_i\longrightarrow
0.
\]
Thus we see that $\coker\Psi_i$ is finite, hence we are in a tropically
transverse gluing situation, and 
\begin{equation}
\label{eq:part way 1}
\mu(\tilde\tau_i,\mathbf{E})=|\coker\Psi_i| = k_{\tau_i}
|\widetilde Q_{i-1}^*/\widetilde Q_i^*|^{-1}.
\end{equation}

\medskip

{\bf Step V.} \emph{Gluing.} We have two evaluation maps
\[
\hbox{$\foM^{\ev}(\shX,\tau_i)\rightarrow \ul{X}_{\bsigma(E_i)}$
and $\foM^{\ev}(\shX,\tilde\tau_{i-1})\rightarrow \ul{X}_{\bsigma(E_i)}$}
\]
whose product gives the morphism $\ul{\ev}$ in Theorem 
\ref{thm:gluing factorization}. Since we are only evaluating at
one leg, the criterion for flatness of Theorem \ref{thm:flatness}
holds automatically. For example, for the leg $L_{i,\out}$ of $\btau_i$
and contraction $\phi:\tau'_i\rightarrow\tau_i$, we necessarily
have $\dim\bsigma'(L)-\dim\bsigma(L)\le 1$. Indeed, $\dim\bsigma(L)=2$
and $\dim\bsigma'(L)\le 3$ as $B$ is three-dimensional.\footnote{From
our point of view, this
is the fundamental point of Parker's and Ranganathan's
observation about gluing in the genus zero, triple point case:
the fact that the boundary $D$ only has triple points means this
numerical criterion for flatness always holds when gluing along one edge.}

Thus we may apply Theorems \ref{thm:gluing factorization} and
\ref{thm:general gluing degree} to see that
\begin{align}
\label{eq:part way 2}
\begin{split}
\phi'_*[\scrM(X,\tilde\btau_i)]^{\virt}
= {} & \mu(\tilde\tau_i,\mathbf{E})
[\scrM^{\sch}(X,\tilde\btau_i)]^{\virt}\\
= {} &
\mu(\tilde\tau_i,\mathbf{E}) \Delta^!([\scrM(X,\btau_i)]^{\virt}
\times [\scrM(X,\tilde\btau_{i-1})]^{\virt}).
\end{split}
\end{align}
On the other hand, we have a diagram
\[
\xymatrix@C=30pt
{
\scrM^{\sch}(X,\tilde\btau_i)\ar[d]_{\ul{\ev}''}\ar[r]&\scrM(X,\btau_i)\times
\scrM(X,\tilde\btau_{i-1})\ar[d]^{\ul{\ev}'}\\
D_i \times \prod_{j=i+1}^q D_j\ar[r]\ar[d]&
D_i\times \prod_{j=i}^qD_j\ar[d]\\
D_i\ar[r]_{\Delta} &D_i\times D_i
}
\]
where $\ul{\ev}'$ is the product of the evaluation map
$\scrM(X,\btau_i)\rightarrow D_i$ at $L_{i,\out}$ and
the evaluation map $\ul{\ev}_{i-1}$. The bottom two vertical arrows
are the obvious projections, and the composition of the right-hand
vertical arrows agrees with the morphism $\ul{\ev}\circ\hat\varepsilon$
of Theorem \ref{thm:gluing factorization}. Next, by definition
\[
\ul{\ev}'_*([\scrM(X,\btau_i)]^{\virt}\times
[\scrM(X,\tilde\btau_{i-1})]^{\virt})=
(\deg [\scrM(X,\btau_i)]^{\virt})N_{i-1}\left([p]\times
\prod_{j=i}^q [D_j]\right),
\]
where $p\in D_i$ is any closed point.
By compatibility of push-forward and Gysin pull-back, we see that
\begin{align*}
\ul{\ev}''_*[\scrM^{\sch}(X,\tilde\btau_i)]^{\virt}
= {} & \Delta^!\ul{\ev}'_*([\scrM(X,\btau_i)]^{\virt}\times[\scrM(X,
\tilde\btau_{i-1})]^{\virt})\\
= {} &
(\deg [\scrM(X,\btau_i)]^{\virt})N_{i-1}\left([p]\times
\prod_{j=i+1}^q [D_j]\right).
\end{align*}
Putting this together with \eqref{eq:part way 1} and \eqref{eq:part way 2}, 
we conclude that
\[
N_i = k_{\tau_i}|\widetilde Q_{i-1}^*/\widetilde Q_i^*|^{-1} 
\deg[\scrM(X,\btau_i)]^{\virt} N_{i-1}.
\]
Hence, inductively we obtain
\begin{equation}
\label{eq:almost there}
\deg [\scrM(X,\btau)]^{\virt} = N_q 
= N_{\btau_{\out}}|Q_0^*/Q_q^*|^{-1}\prod_{i=1}^q k_{\tau_i}
\deg[\scrM(X,\btau_i)]^{\virt}.
\end{equation}
Now note that $Q_0^*$ can be identified with $N_{\bsigma(v_{\out})}$,
as the vertex $v_{\out}$ of $G_{\tau_{\out}}$ can be placed at any
integral point of
$\bsigma(v_{\out})$. Further, $N_{\bsigma(L_{\out})}/N_{\bsigma(v_{\out})}$
can be identified with the integral tangent vectors of the ray 
$\bsigma_{\out}(L_{\out})$ in $\widetilde\Sigma_x$,
while $w_{L_{\out}}$ is the index of the image $\mathbf{u}_{\out}(L_{\out})$
of $\mathbf{u}(L_{\out})$ in $N_{\bsigma(L_{\out})}/N_{\bsigma(v_{\out})}$.
Thus $w_{L_{\out}} |Q_0^*/Q_q^*|$ coincides with $k_{\tau}$.

Next note that $\Aut(\btau)=\Aut(\btau_1,\ldots,\btau_q)
\times \prod_{i=1}^q \Aut(\btau_i)$. 
Thus multiplying both sides of \eqref{eq:almost there}  by
$k_{\tau}$ and dividing by $|\Aut(\btau)|$ gives the desired result.
\end{proof}

\subsection{Behaviour of $N_{\beta}$ under blow-ups.}
\label{subsec:blowups}

Here we will give another sample application of our gluing
formalism. This material does not represent anything radically new;
rather, it is a modern version of an argument appearing in \cite[\S5]{GPS},
albeit carried out there using older technology in a somewhat
restrictive circumstance. There are also related arguments in
\cite{Mandel}. The formula given here will
be essential in \cite{K3}.

As in \S\ref{subsect:invariant LP}, we fix a Looijenga pair $(X,D)$,
which we now take to be non-singular.
As before, we assume $D$ has at least three nodes. 
Write $D=D_1+\cdots+D_n$ the irreducible decomposition, in cyclic order,
so that $D_i\cdot D_{i+1}=1$, with indices taken modulo $n$. This
gives rise as before to an integral affine manifold with singularities 
$B$ along with its cone decomposition $\P$.

Consider a collection of distinct boundary divisors
$D_{j_1},\ldots,D_{j_s}$, which for convenience we assume are pairwise disjoint.
Choose, for each $k$,
distinct points $p_{k1},\ldots,p_{kn_k}\in D_{j_k}$, also distinct
from the nodes of $D$. Denote by $\pi:\widetilde X\rightarrow X$
the blow-up of $X$ at all of these points. Denote by
$E_{k\ell}$ the exceptional curve over $p_{k\ell}$. Let $\widetilde D$
be the strict transform of $D$, so that $(\widetilde X,\widetilde D)$
is also a Looijenga pair. This gives rise to an integral affine manifold
with singularities $\widetilde B$ along with its cone decomposition
$\widetilde\P$. Note that $\pi$ induces a bijection between the
components of $\widetilde D$ and $D$, which in turn induces a natural
piecewise linear identification between $(\widetilde B,\widetilde\P)$
and $(B,\P)$.

We fix a class $\tilde\beta$ of logarithmic map to $\widetilde X$
of genus zero and $q+1$ marked points with non-zero 
contact orders $u_1,\ldots,u_q,u_{\out}\in \widetilde B(\ZZ)$, 
all non-zero and contained in rays of $\widetilde\P$. 
As in \S\ref{subsect:invariant LP}, we assume $u_k$ 
is contained in a ray $\rho_{i_k}$
corresponding to the divisor $\widetilde D_{i_k}$. The attached
curve class is $\widetilde A\in H_2(\widetilde X)$. We thus obtain
the invariant $N_{\tilde\beta}$ of \eqref{eq:Nbeta def}.

In this situation, set 
\begin{equation}
\label{eq:wkl def}
w_{k\ell}:= \widetilde A\cdot E_{k\ell}.
\end{equation} 
Assume from now on that: 
\begin{equation}
\label{eq:w is nonnegative}
\hbox{$w_{k\ell}\ge 0$ for all $k,\ell$.}
\end{equation}
We denote
by $\mathbf{P}_{k\ell}=P_{k\ell 1}+\cdots+P_{k\ell \mu}$ an unordered
 partition of $w_{k\ell}$ into $\mu$ positive
integers, for some $\mu\ge 0$ (with $\mu=0$ only if $w_{k\ell}=0$).
Write $\mathbf{P}=(\mathbf{P}_{k\ell})$ a collection of partitions of
all $w_{k\ell}$. We write
$\Aut(\mathbf{P}_{k\ell})$ for the subgroup of permutations
$\sigma$ of $\{1,\ldots,\mu\}$ with $P_{k\ell\sigma(m)}=P_{k\ell m}$
for $1\le  m\le \mu$, and write $\Aut(\mathbf{P})=\prod_{k\ell} 
\Aut(\mathbf{P}_{k\ell})$.

For a given $\mathbf{P}$, we write $\beta(\mathbf{P})$ for the class
of logarithmic map to $X$ defined as follows. The curve class is
$A(\mathbf{P})=\pi_*\widetilde A$. The class has $q+1$ marked points,
still with contact orders $u_1,\ldots,u_q,u_{\out}$ using the 
piecewise linear identification of $(\widetilde B,\widetilde\P)$
and $(B,\P)$. It has an additional marked point for every $P_{k\ell m}$,
with contact order $P_{k\ell m} \nu_{j_k}$, where $\nu_{j_k}$ is
the primitive generator of the ray of $\P$ corresponding to $D_{j_k}$.

We then have:

\begin{theorem}
\begin{equation}
\label{eq:wowee}
N_{\tilde\beta} = \sum_{\mathbf P} {N_{\beta(\mathbf{P})}\over |\Aut(\mathbf{P})|}
\prod_{k,\ell,m} {(-1)^{P_{k\ell m}} \over P_{k\ell m}},
\end{equation}
where the sum is over all collections of partitions $\mathbf{P}$.
\end{theorem}

\begin{proof}
{\bf Step I.} \emph{Building the degeneration}. We build a log smooth
degeneration $\widetilde\shX\rightarrow \AA^1$ of $\widetilde X$,
following \cite[\S3.1]{HDTV}, which in turn is inspired by \cite{GPS},
as follows.\footnote{During this proof, we waive the typographic
convention that $\shX$ is the Artin fan of $X$.} We first construct a blow-up
$\shX \rightarrow X\times \AA^1$, blowing up the closed subscheme 
$\bigcup_{k=1}^s D_{j_k}\times \{0\}$. Next, we form the blow-up
$\widetilde\shX \rightarrow \shX$ with center the strict transform
of $\bigcup_{k,\ell} \{p_{k\ell}\}\times \AA^1$. Thus $\widetilde\shX_t\cong
\widetilde X$ for $t\not=0$. Let $\widetilde\shD\subseteq\widetilde\shX$ 
be the union of
$\widetilde\shX_0$ and the strict transform of $D\times \AA^1$.
We also write $\widetilde\shD_i$ for the strict transform of $D_i\times\AA^1$.
Then
$(\widetilde \shX,\widetilde \shD)$ is a log Calabi-Yau pair, and
the composition of 
\[
\tilde\pi:\widetilde\shX\rightarrow\shX \rightarrow X\times
\AA^1
\]
with the projection to $\AA^1$ gives a log smooth morphism
$g:\widetilde \shX\rightarrow \AA^1$, with $\AA^1$ carrying the divisorial
log structure given by $0\in\AA^1$.

We have the map $\Sigma(g):\Sigma(\widetilde\shX)\rightarrow \Sigma(\AA^1)
=\RR_{\ge 0}$, with $\Delta(\widetilde\shX)$ the inverse image of
$1$, a polyhedral complex. On the other hand, in the notation of 
\cite{Walls}, $\widetilde B:=|\Sigma(\widetilde\shX)|$ carries an
integral affine structure with singularities, and we may write
$\P_{\widetilde{\shX}}$ for the collection of cones of $\Sigma(\widetilde\shX)$.
Similarly, we write $\widetilde B':=|\Delta(\widetilde\shX)|$ for the fibre
of $g_{\trop}=\Sigma(g):\widetilde B\rightarrow \RR_{\ge 0}$ over $1$, carrying
an integral affine structure with singularities and the polyhedral
decomposition $\P'_{\widetilde\shX}$
 consisting of the cells of $\Delta(\widetilde
\shX)$. See \cite[Prop.~1.16]{Walls}.

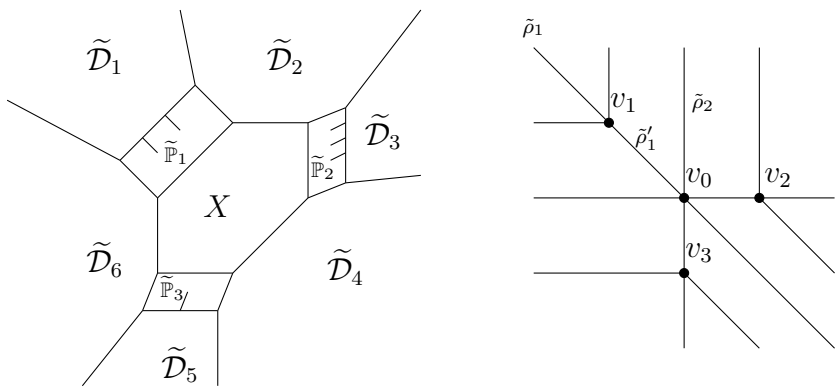
\begin{figure}
\centering
\begin{tikzpicture}
\draw (-1,-1) -- (-1,0);
\draw (-1,0) -- (0,1);
\draw (0,1) -- (1,1);
\draw (1,1) -- (1,0);
\draw (1,0) -- (0,-1);
\draw (0,-1) -- (-1,-1);
\draw (-1,0) -- (-1.5,.5);
\draw (-1.5,.5) -- (-3,1.3);
\draw (-1.5,.5) -- (-.5,1.5);
\draw (-1.2,.8) -- (-1,.6);
\draw (-.9,1.1) -- (-.7,.9);
\draw (-.5,1.5) -- (-.7, 2.5);
\draw (-.5,1.5) -- (0,1);
\draw (1,1) -- (1.5,1.2);
\draw (1.5, 1.2) -- (2.5, 2.5);
\draw (1.5,1.2) -- (1.5,.2);
\draw (1.5,1) -- (1.3,.9);
\draw (1.5,.8) -- (1.3,.7);
\draw (1.5,.6) -- (1.3,.5);
\draw (1.5,.2) -- (2.5,.3);
\draw (1.5,.2) -- (1,0);
\draw (0,-1) -- (-.2,-1.5);
\draw (-.2, -1.5) -- (-.2,-2.5);
\draw (-.2,-1.5) -- (-1.2,-1.5);
\draw (-.7,-1.5) -- (-.6, -1.25);
\draw (-1.2,-1.5) -- (-2.0,-2.5);
\draw (-1.2,-1.5) -- (-1,-1);

\node[above] at (-.2,-.4) {$X$};
\node[above] at (-1.7,1.5) {$\widetilde\shD_1$};
\node[above] at (.7,1.5) {$\widetilde\shD_2$};
\node[above] at (2.0,.5) {$\widetilde\shD_3$};
\node[above] at (1.5,-1.3) {$\widetilde\shD_4$};
\node[above] at (-.7,-2.6) {$\widetilde\shD_5$};
\node[above] at (-1.7,-1.2) {$\widetilde\shD_6$};

\node[above] at (-.75,.3) {$\scriptstyle \widetilde\PP_1$};
\node[above] at (1.2,0.1) {$\scriptstyle \widetilde\PP_2$};
\node[above] at (-.8,-1.5) {$\scriptstyle \widetilde\PP_3$};

\draw (8,0) -- (4,0);
\draw (6,2) -- (6,-2);
\draw (4,2) -- (8,-2);
\draw (5,1) -- (4,1);
\draw (5,1) -- (5,2);
\draw (7,0) -- (7,2);
\draw (7,0) -- (8,-1);
\draw (6,-1) -- (7,-2);
\draw (6,-1) -- (4,-1);

\draw (6,0)[black, fill=black] circle (.06cm);
\draw (5,1)[black, fill=black] circle (.06cm);
\draw (7,0)[black, fill=black] circle (.06cm);
\draw (6,-1)[black, fill=black] circle (.06cm);

\node[above] at (6.2,0) {$v_0$};
\node[above] at (5.2,1) {$v_1$};
\node[above] at (7.25,0) {$v_2$};
\node[above] at (6.2,-1) {$v_3$};

\node[above] at (4,2) {$\scriptstyle \tilde\rho_1$};
\node[above] at (5.5,.5) {$\scriptstyle \tilde\rho_1'$};
\node[above] at (6.25,1) {$\scriptstyle \tilde\rho_2$};
\end{tikzpicture}
\caption{The boundary divisor $\widetilde\shD$ of $\widetilde\shX$ and
the corresponding affine manifold $\widetilde B'$. The vertices are
all singular points of the affine structure.}
\label{fig}
\end{figure}
See Figure \ref{fig} for what $\widetilde\shD$ and $\widetilde B'$ will 
look like. In particular, the central fibre is a union 
\[
\widetilde\shX_0 = X\cup \bigcup_{k=1}^s \widetilde\PP_k,
\]
where $\PP_k$, $1\le k\le s$, is the exceptional divisor of the
first blow-up lying over $D_{j_k}\times\{0\}$, and $\widetilde\PP_k$
is the strict transform of $\PP_k$ under the second blow-up. Note
that $\PP_k$ is a Hirzebruch surface and $\widetilde\PP_k\rightarrow\PP_k$
is a blow-up at $n_k$ points. The exceptional curves of this blowup,
which we write as $e_{k1},\ldots,e_{kn_k}$, are by construction disjoint
from $X$. We also write $f_k\in H_2(\widetilde\PP_k)$ for the the class 
of a fibre of the ruling $\widetilde \PP_k\rightarrow D_{j_k}$.

We write the vertices of $\P'_{\widetilde\shX}$ as $v_0,\ldots,v_s$,
corresponding to the irreducible components $X,\widetilde\PP_1,\ldots,
\widetilde\PP_s$ respectively. We also have rays $\tilde\rho_i$, $1\le i\le 
n$ corresponding to the one-dimensional stratum $X\cap \widetilde\shD_i$
if $i\not \in \{j_1,\ldots,j_s\}$, and corresponding to 
the one-dimensional stratum $\widetilde\PP_k\cap \widetilde\shD_i$ if $i=j_k$.
Further, we have a segment $\tilde\rho_{j_k}'$ connecting $v_0$ and $v_k$,
corresponding to the stratum $\widetilde\PP_k\cap X$.

The structure of $\widetilde B'$ near
the vertices $v_1,\ldots,v_s$ has been analyzed in detail in 
\cite[\S3.3.1]{HDTV}. We note that in that reference, $X$ was a toric
variety, but the analysis away from the vertex $v_0$ remains the same.
In particular, consider the two one-cells $\tilde\rho_{j_k}$,
$\tilde\rho_{j_k}'$. Then by \cite[Cor.~3.7]{HDTV}, the tangent spaces to these
two one-cells are left invariant under affine monodromy about $v_k$.
More precisely, choose a point $x$ near $v_k$. Then
we may view the group
$\Lambda_{\tilde\rho_{j_k}}$ (resp.\ $\Lambda_{\tilde\rho_{j_k}'}$)
of integral tangent vectors to
$\tilde\rho_{j_k}$ (resp.\ $\tilde\rho_{j_k}'$)
as a well-defined sublattice of $\Lambda_x$ via parallel transport.
Furthermore, under these identifications, 
$\Lambda_{\tilde\rho_{j_k}}$ and $\Lambda_{\tilde\rho_{j_k}'}$ agree.
Thus $\tilde\rho_{j_k}\cup\tilde\rho_{j_k}'$ may be viewed as a straight line
through $v_k$, despite the singularity in the affine structure at $v_k$.

Given a punctured map $f:C^{\circ}/W\rightarrow\widetilde\shX$ 
defined over $\AA^1$ with $W$ a log point and type $\tau$, let $t\in\tau$
be such that the tropicalization $h_t:G\rightarrow \widetilde B$ factors
through $\widetilde B'$. Let $v\in V(G)$ be a vertex with $h_t(v)=v_k$ for
some $k>0$. If $E_1,\ldots,E_\ell$ are the edges or legs adjacent to $v$,
oriented away from $v$, let $u_i$ be the image of $\mathbf{u}(E_i)$ in
$\Lambda_x/\Lambda_{\tilde\rho_{j_k}}$, well-defined under parallel
transport. Here $x$ is the chosen point near $v_k$.
Then $h_t$ satisfies the weaker balancing condition
\begin{equation}
\label{eq:weak balancing}
\hbox{$\sum_{i=1}^\ell u_i=0$ in $\Lambda_x/\Lambda_{\tilde\rho_{j_k}}$}
\end{equation}
by \cite[Prop.~3.10]{HDTV}. 

As $\widetilde\shX$ is non-compact, we should be more precise about
what group our curve classes live in. Here we will take, for any
stratum $Y \subseteq \widetilde\shX$, $H_2(Y) := \Pic(Y)^*$. Note that
as $X$ is a rational surface, when $Y$ is a compact
stratum of $\widetilde\shX$, $H_2(Y)$ coincides
with the usual integral singular homology group $H_2(Y,\ZZ)$.

\medskip

{\bf Step II.} \emph{Degenerating the enumerative problem.}
We begin by choosing points $x_k \in \widetilde D_{i_k}$, $1\le k \le q$,
not coinciding with any double point of $\widetilde D$ and
not coinciding with any of the points $p_{k\ell}$. With
$\mathbf{x}=(x_1,\ldots,x_q)$, consider the moduli space
$\scrM(\widetilde X,\tilde \beta,\mathbf{x})$ of \eqref{eq:point constraint}.
Then $N_{\tilde\beta}=\deg[\scrM(\widetilde X,\tilde\beta,\mathbf{x})]^{\virt}$
by Proposition \ref{prop:alt def}.
Similarly, we may consider sections $\tilde x_k\subseteq \widetilde\shX$
which are the strict transforms of $x_k \times \AA^1$, and obtain
\[
\scrM(\widetilde\shX/\AA^1,\tilde\beta,\tilde{\mathbf{x}}):=
\scrM(\widetilde\shX/\AA^1,\tilde\beta)\times_{\prod_k \widetilde\shX}
\prod_{k=1}^q \tilde x_k.
\]
Of course, for $0\not=t\in\AA^1$, we have
\[
\scrM(\widetilde X,\tilde\beta,\mathbf{x})=\scrM(\widetilde\shX/\AA^1,
\tilde\beta,\tilde{\mathbf{x}})\times_{\AA^1} t,
\]
so as in \cite[Thm.~1.1]{ACGSI}, we have 
\[
\deg [\scrM(\widetilde X,\tilde\beta,\mathbf{x})]^{\virt} = 
\deg [\scrM(\widetilde \shX_0/0,\tilde\beta,\tilde{\mathbf{x}}_0)]^{\virt}.
\]
Here $\tilde{\mathbf{x}}_0=((\tilde x_k)_0)$ denotes the tuple of points
with $(\tilde x_k)_0=\tilde x_k \cap g^{-1}(0)$.
As in \cite[Thm.~5.4]{ACGSI}, we then have a decomposition
\begin{equation}
\label{eq:beta decomp}
N_{\tilde\beta}=\deg [\scrM(\widetilde\shX_0/0,\tilde\beta,
\tilde{\mathbf{x}}_0)]^{\virt}
=
\sum_{\btau} {m_{\tau}\over |\Aut(\btau)|} \deg
[\scrM(\widetilde\shX_0/0,\btau,\tilde{\mathbf{x}}_0)]^{\virt},
\end{equation}
where the sum is over all rigid decorated tropical maps to
$\widetilde B'$ which can arise as a degeneration of the class $\tilde\beta$.
As in 
Definition \ref{def:wall type},
we assume the decoration functions $\mathbf{A}$ are refined, i.e., have
$\mathbf{A}(v)\in H_2(\widetilde\shX_{\bsigma(v)})$ rather than
in $H_2(\widetilde\shX)$.

\medskip

{\bf Step III}.  \emph{Classifying rigid tropical maps: first steps}.
Let $h:G\rightarrow \widetilde B'$ be a rigid tropical map of
some type $\tau$ and decoration $\mathbf{A}$
contributing to the decomposition \eqref{eq:beta decomp}. 
Let $L_1,\ldots,L_q\in L(G)$ be
the legs corresponding to the marked points mapping to the tuple of
points of $\tilde{\mathbf{x}}_0$. Note $G$ has one additional leg, denoted
$L_{\out}$. Because of the assumption
that the $x_k$ are not double points of $\widetilde D$, necessarily
$h(L_k)$ is contained in the ray $\tilde\rho_{i_k}$. We may now apply the
proof of Lemma \ref{lem:vanishing} in this setting. 
Specifically, in that proof, we showed that the image of $h_s$
was contained in the one-skeleton of $\P$. Here,
the target affine manifold
with singularities $\widetilde B'$ is slightly more complicated than 
the $B$ of the proof of the lemma, but the same argument works as
the weaker balancing condition \eqref{eq:weak balancing} at the 
vertices $v_1,\ldots, v_s$ is still sufficient. This allows us to conclude 
that in any event, whether or not $h$
is rigid, its image lies in $\bigcup_{i=1}^n\tilde\rho_i\cup 
\bigcup_{k=1}^s \tilde\rho_{j_k}'$. However, rigidity also then implies that
vertices of $G$ must map to vertices of $\widetilde\P'$, as otherwise
the location of these vertices may be freely moved along the edge containing
their image. 

We now note that it is immediate that each such rigid type $\tau$
is tropically transverse in the sense of Definition 
\ref{def: tropically transverse2}. As in Step V of the proof of Theorem
\ref{thm:inductive structure}, since again $\widetilde\shD$ only 
has triple points,
we do not need to worry about the flatness hypothesis of
Theorem \ref{thm:tropically transverse}, and will use this gluing result
in what follows without reference to these hypotheses.

\medskip

{\bf Step IV.} \emph{Balancing at each vertex mapping to
$v_k$}.
Continue with the notation of the previous step, and let $v\in V(G)$
be a vertex such that $h(v)=v_k$ for some $k>0$. Let
$E_1,\ldots,E_m$
be the edges or legs adjacent to $v$ mapping to $\tilde\rho_{j_k}$. Note
these in fact must necessarily be legs, as if $E_i$ were an edge, the
other vertex of $E_i$ would map to the interior of $\tilde\rho_{j_k}$.
Similarly, let $E_1',\ldots,E_{m'}'$ be the edges or legs adjacent to 
$v$ mapping to $\tilde\rho_{j_k}'$. Note these are necessarily edges
with opposite endpoint mapping to $v_0$. Let $w_i$ (resp.\ $w_i'$) be
the index of $\mathbf{u}(E_i)$ (resp.\ $\mathbf{u}(E_i')$), and let
$w=\sum w_i$, $w'=\sum w_i'$. We now make a sequence of observations.

\medskip

\noindent \emph{Observation A}.
No edge or leg of $G$ maps to the one-dimensional cells of 
$\P'_{\widetilde\shX}$ adjacent to $v_k$ which are not $\tilde\rho_{j_k}$
or $\tilde\rho_{j_k}'$. Indeed, as these cells are unbounded, only
legs may map to such cells. As the
images of the legs $L_i$ all map to cells of the form $\tilde\rho_j$, 
the only possibility is
that $L_{\out}$ maps to such a cell. However, that would violate
the balancing condition \eqref{eq:weak balancing}.

\medskip 

\noindent \emph{Observation B.} Suppose $f:C\rightarrow \widetilde\shX_0$
is a punctured map defined over a log point in the moduli space
$\scrM(\widetilde\shX_0/0,\btau,\tilde{\mathbf{x}}_0)$. 
Let $\ul{C}_v\subseteq \ul{C}$ be the union of irreducible
components of $\ul{C}$ correponding to $v$ in the marking by $\btau$.
Then $f(\ul C_v)$ is of curve class $A:=\mathbf{A}(v)\in H_2(\widetilde\PP_k)$.
By Corollary~\ref{cor:114}, necessarily the intersection number of $A$ with
$\widetilde\shD_{j_k}\cap \widetilde \PP_k$ is $w$ and the intersection
number of $A$ with $\widetilde\PP_k\cap X$ is $w'$. Meanwhile, by
Observation A, the
intersection number of $A$ with the other two strata of $\widetilde\PP_k$
must be zero. As a consequence, $A$ is a linear combination of $f_k$
and the $e_{k\ell}$. Since $\ul{C}_v$ is connected, there is no choice but
for $A= w(f_k-e_{k\ell})+w' e_{k\ell}$ for some $\ell$.
In particular, the image $f(\ul{C}_v)$ is contained in a fibre of
$\pi|_{\widetilde\PP_k}:\widetilde\PP_k\rightarrow D_{j_k}$, 
reducible if $w\not=w'$.

\medskip

\noindent \emph{Observation C}.
We will now show that the contribution to $N_{\tilde\beta}$ is zero
unless either one of the following hold:
\begin{itemize}
\item[(IV.1)] $m=0$, i.e., there are no legs adjacent to $v$;
\item[(IV.2)] $m=1$ and $E_1=L_{\out}$; or 
\item[(IV.3)] $w=w'$. 
\end{itemize}
Case (IV.3) is really balancing
at $v$, which can now be viewed as a well-defined equality in 
$\Lambda_{\tilde\rho_j}$ by Step I.

Assume we are not in case (IV.1) or (IV.2), so that there is at least one 
leg $L_i$
adjacent to $v$. Assume that
$\scrM(\widetilde \shX_0/0,\btau,\tilde{\mathbf{x}}_0)$
is non-empty so that there exists a punctured map as in Observation B.
By Observation B, we have $f(\ul{C}_v)$ is a fibre of
$\widetilde\PP_k \rightarrow D_{j_k}$, reducible if $w\not=w'$. 
On the other hand, by the assumption that $L_i$ is
adjacent to $v$, there is 
a marked point of $C$ on $C_v$ map to $(\tilde x_i)_0$ under $f$.
However, since the points $x_i$ were chosen to be distinct from any
$p_{k\ell}$, $f(\ul{C}_v)$ cannot be reducible. Hence $w=w'$.

\medskip

\noindent \emph{Observation D}. We have $m'>0$. Indeed, if 
$m'=0$, then since $G$ is connected, the only possibility is
$G$ has one vertex $v$, which maps to $v_k$, and a number of legs,
mapping to $\tilde\rho_{j_k}$. By Observation B, we then have
$\mathbf{A}(v)= w e_{k\ell}$ for some $\ell$, and then necessarily
the total curve class of the type is also $c e_{k\ell}$. But this
curve class deforms to $w E_{k\ell}$ on a general fibre of $\widetilde
\shX\rightarrow \AA^1$, and then $w_{k\ell}=-w<0$, contradicting
\eqref{eq:w is nonnegative}.

\medskip

{\bf Step V.} \emph{Univalency or bivalency of each vertex mapping to $v_j$.}

Continuing with the notation of Step IV, assume that $\scrM(\widetilde
\shX_0/0,\btau,\tilde{\mathbf{x}}_0)$ is non-empty. We first observe that
we can't have two distinct legs $L_i,L_{i'}$ adjacent to $v$.
Indeed, by Observation B above,
$f(\ul{C}_v)$ lies
in a fibre of $\widetilde\PP_j\rightarrow D_{j_k}$. Thus
the corresponding marked points of $C$ contained in $\ul{C}_v$ 
must map to the same point of
$\widetilde\shD_{j_k}\cap \widetilde \PP_k$. However, since the $x_i$
are chosen to be distinct, this is not possible.

Now suppose that one of the following two cases hold.
\begin{itemize}
\item[(V.1)] Either (a) $L_{\out}$ is adjacent to $v$ 
and another leg $L_i$ is adjacent to $v$ or (b) 
$L_{\out}$ is adjacent to $v$ and balancing fails at $v$, i.e., $w\not=w'$
or (c) $m'>1$ and balancing fails at $v$. 
\item[(V.2)]
$m'>1$.
\end{itemize}
We will show the contribution to \eqref{eq:beta decomp} from the type
$\btau$ is zero in these cases.

We may split $\btau$ along the edges
$E_1',\ldots,E'_{m'}$, giving decorated types $\btau_v$ and
$\btau_{1},\ldots,\btau_{m'}$ with underlying graphs 
$G_v,G_1,\ldots,G_{m'}$. We note that here we use the fact
the genus of the underlying graph is zero, so that splitting at any edge
produces an additional connected component. In either Case (V.1) or (V.2), 
since $m'>0$ in any event by Observation D,
there is a $j$ such that $L_{\out}\not\in L(G_j)$.
Thus all legs in $G_j$ other
than the leg corresponding to $E_j'$ correspond to marked points
which are constrained to map to some subset of the points $\{(\tilde x_i)_0\}$.
From \cite[Prop.~3.30 and (4.18)]{ACGSII}, taking into account the 
just-mentioned
point-constraints, the virtual dimension of $\scrM(\widetilde\shX_0/0,
\btau_j,\tilde{\mathbf{x}}_0)$ is zero. 

If we are in Case (V.1), $\ul{C}_v$ maps to a fixed fibre of
$\widetilde\PP_k\rightarrow D_{j_k}$ (a reducible fibre in the unbalanced
case). Thus the image of the evaluation 
map $\scrM(\widetilde\shX_0/0, \btau_v,\tilde{\mathbf{x}}_0)\rightarrow
\widetilde\PP_k\cap X=D_{j_k}$ at the punctured point corresponding
to an edge $E_i'$ is a point, being either the image
of $(\tilde x_i)_0$ under the projection map 
$\widetilde\PP_k\rightarrow D_{j_k}$ in the case (a)
or a point $p_{k\ell}$ in the case (b) and (c). It then follows immediately
in this case that 
\[
\Delta^!\big([\scrM(\widetilde\shX_0/0,\btau_v,\tilde{\mathbf{x}}_0)]^{\virt}
\times [\scrM(\widetilde\shX_0/0,\btau_j,\tilde{\mathbf{x}}_0)]^{\virt}\big)=0,
\]
with $\Delta$ as in Theorem \ref{thm:gluing factorization} for the
gluing of $\btau_v$ and $\btau_j$. Hence by Theorem 
\ref{thm:tropically transverse}, the contribution from the type
$\btau$ is zero.

Next suppose we are in Case (V.2), but not Case (V.1).  
The only remaining possibility is that $m'>1$, there is only one leg
adjacent to $v$, and the type $\btau$ is balanced at $v$. 
Suppose $L_{\out}$ is the only leg adjacent to $v$. 
In this case, the evaluation map $\ul{\ev}_v:\scrM(\widetilde\shX_0/0,
\btau_v,\tilde{\mathbf{x}}_0)\rightarrow D_{j_k}$ will be
surjective. On the other hand, in this case $\scrM(\widetilde\shX_0/0,
\btau_i,\tilde{\mathbf{x}}_0)$ is virtual dimension $0$ for all $i$,
and another application of Theorem \ref{thm:tropically transverse}
shows the contribution from the type $\btau$ vanishes. A similar argument
applies if $L_i$ is the leg adjacent to $v$, this time because
the evaluation map $\ul{\ev}_v$ has image a point and the virtual
dimension of $\scrM(\widetilde\shX_0/0,
\btau_i,\tilde{\mathbf{x}}_0)$ is zero for at least one $i$.

In summary, if $\btau$ does contribute non-trivially to
\eqref{eq:beta decomp}, either $v$ is univalent, necessarily with 
$\mathbf{A}(v)$ a positive multiple of $f_k-e_{k\ell}$ for some $k$,
or $v$ is bivalent and balanced, with $\mathbf{A}(v)$ a positive
multiple of $f_k$.

\medskip

{\bf Step VI.} \emph{Calculating the non-trivial contributions from 
rigid tropical maps}. The only remaining possibility for a rigid
tropical map contributing non-trivially to the right-hand side of
\eqref{eq:beta decomp} is now as follows. There is one vertex
$w$ with $h(w)=v_0$. Attached to $w$ are those legs $L$ amongst
$L_1,\ldots,L_q,L_{\out}$ for which $\bsigma(L)=\tilde\rho_k$ with
$k\not\in\{j_1,\ldots,j_s\}$. For those legs $L$ with
$\bsigma(L)=\tilde\rho_{j_k}$, we instead have an edge adjacent to $w$
mapping surjectively to $\tilde\rho_{j_k}'$, with opposite vertex $v_L$
bivalent and adjacent to $L$, which maps surjectively to $\tilde\rho_{j_k}$.
Finally, there may be an additional set of edges adjacent to $w$
mapping to the various $\tilde\rho_{j_k}'$, with opposite vertex being
univalent. For each $k,\ell$, we will have some number $m_{k\ell}$
of such univalent vertices mapping to $v_j$, with attached curve classes
$P_{k\ell 1}(f_k-e_{k\ell}),
\ldots, P_{k\ell m_{k\ell}}(f_k-e_{k\ell})$, for some
positive integers $P_{k\ell 1},\ldots,P_{k\ell m_{k\ell}}$. Note these curve 
classes
also determine the contact order for the edge adjacent to such a univalent
vertex, again by Corollary~\ref{cor:114}.

Recall from Step I that curve classes take values in the duals of the
Picard groups of strata. In particular, $H_2(X\times\AA^1)\cong H_2(X)$.
If $\btau$ is a decorated type contributing to \eqref{eq:beta decomp} and
$A(\btau)$ is the total curve class of the decorated type $\btau$,
then under the blow-down $\tilde\pi:\widetilde\shX\rightarrow X\times\AA^1$,
$\pi_*A(\btau)$ must agree with
$A:=\pi_*\widetilde A$. Since the curve class $\mathbf{A}(v)$ is 
contracted by $\tilde\pi$ for any $v \not=w$, we see that $\mathbf{A}(w)=A$.
Further, since $\widetilde A\cdot E_{k\ell}=w_{k\ell}$ (see
\eqref{eq:wkl def})
it follows by intersecting $A(\btau)$
with the exceptional divisor over
the strict transform of $p_{k\ell}\times\AA^1$ that
$\mathbf{P}_{k\ell}=P_{k\ell 1}+\cdots+P_{k\ell m_{k\ell}}$ is a partition of 
$w_{k\ell}$. In conclusion, the type $\btau_w$ associated with the vertex
$w$ may now be viewed as a class of map to $X$, and is precisely 
the type $\beta(\mathbf{P})$, $\mathbf{P}=(\mathbf{P}_{k\ell})$.

We are now ready to compute the contribution of this type to the
right-hand-side of \eqref{eq:beta decomp}. First, we note that for
each univalent vertex $v$ with associated curve class 
$P_{k\ell m} (f_k-e_{k\ell})$, we may apply Theorem 
\ref{thm:punctured to relative} and \cite[Prop.~5.2]{GPS} to see that
$\scrM(\widetilde\shX_0/0,\btau_w,\tilde{\mathbf{x}}_0)$ is
virtual dimension zero with virtual fundamental class of degree
$(-1)^{P_{k\ell m}}/P_{k\ell m}^2$. On the other hand, if $v$ is a bivalent
vertex mapping to $v_k$ with associated curve class $d f_k$, then
$\scrM(\widetilde\shX_0/0,\btau_w,\tilde{\mathbf{x}}_0)$ is
easily seen to consist of one curve mapping $d:1$ to a non-singular fibre of
$\widetilde\PP_k\rightarrow D_{j_k}$, totally branched over
$\widetilde\PP_k\cap \widetilde \shD_{j_k}$ and $\widetilde\PP_k\cap X$.
This map is unobstructed, and the fundamental class of the moduli space
is degree $1/d$.

A simple calculation now shows that the multiplicity $\mu(\tau)$
is $\prod_{E\in E(G)} w_E$, where $w_E$ is the index of the contact
order $\mathbf{u}(E)$. If $E$ is an edge adjacent to a bivalent vertex
with associated curve class $d f_k$, then $w_E=d$, while if $E$ is
an edge adjacent to a univalent vertex with associated curve class
$P_{k\ell m}(f_k-e_{k\ell})$, we have $w_E=P_{k\ell m}$. The desired
contribution to $N_{\tilde\beta}$ from $\btau$ in \eqref{eq:beta decomp}
now agrees with the summand in \eqref{eq:wowee} corresponding to the
collection of partitions $\mathbf{P}$. Indeed, this follows from the
description of the virtual degrees of the moduli spaces above, the
calculation of the multiplicity $\mu(\tau)$, the definition of 
$N_{\beta(\mathbf{P})}$, and Theorem \ref{thm:tropically transverse}.
\end{proof}


\end{document}